\documentclass[11pt,a4paper]{amsart}
\usepackage[utf8]{inputenc}
\usepackage[T1]{fontenc}
\usepackage[english]{babel}

\usepackage{sidecap}

\usepackage{verbatim}
\usepackage{lmodern}
\usepackage{amsmath}
\usepackage{amssymb} 
\usepackage{amsthm} 
\usepackage{thmtools}
\usepackage{enumitem}
\usepackage[margin=1.2in]{geometry}
\usepackage{graphicx}
\usepackage[hidelinks]{hyperref}
\usepackage{mathtools}
\usepackage{indentfirst}
\usepackage{tabularx}
\usepackage{bbm}
\usepackage{bbold}
\usepackage[capitalise]{cleveref}
\usepackage{stmaryrd}
\usepackage{tabularx}

\usepackage{xfrac}

\usepackage[
textwidth=2cm,
textsize=small,
colorinlistoftodos]
{todonotes} 

\usepackage{tikz} 
\usepackage{tikz-cd} 
\newcommand{\btk}{\begin{tikzcd}}
\newcommand{\etk}{\end{tikzcd}}
\usetikzlibrary{arrows,calc,decorations.markings}

\usepackage{xparse}

\newenvironment{tz}{\begin{center}\begin{tikzpicture}}{\end{tikzpicture}\end{center}}
\tikzstyle{d}=[double distance=.3ex]
\tikzstyle{w}=[preaction={draw=white,-,line width=5pt}]

\usepackage{xparse}

\usepackage{blindtext}

\makeatletter
\newcommand{\@bbify}[1]{
  \ifcsname b#1\endcsname
  \message{WARNING: Overwriting b#1 with blackboard letter!}
  \fi
  \expandafter\edef\csname b#1\endcsname
  {\noexpand\ensuremath{\noexpand\mathbb #1}\noexpand\xspace}}
\newcommand{\@calify}[1]{
  \ifcsname c#1\endcsname
  \message{WARNING: Overwriting c#1 with calligraphic letter!}
  \fi 
  \expandafter\edef\csname c#1\endcsname
  {\noexpand\ensuremath{\noexpand\mathcal #1}\noexpand\xspace}}
\newcommand{\@bfify}[1]{
  \ifcsname bf#1\endcsname
  \message{WARNING: Overwriting c#1 with bold letter!}
  \fi
  \expandafter\edef\csname bf#1\endcsname
  {\noexpand\ensuremath{\noexpand\mathbf #1}\noexpand\xspace}}
\newcounter{@letter}\stepcounter{@letter}
\loop\@bbify{\Alph{@letter}}\@calify{\Alph{@letter}}\@bfify{\Alph{@letter}}
\ifnum\the@letter<26\stepcounter{@letter}\repeat
\makeatother

\tikzset{%
node distance=1.5cm, la/.style={scale=0.8}, lasmall/.style={scale=0.75}, over/.style={auto=false,fill=white,inner sep=1.5pt, minimum size=0, outer sep=0},
    symbol/.style={%
        draw=none,
        every to/.append style={%
            edge node={node [sloped, allow upside down, auto=false]{$#1$}}},
            
    }, pro/.style={postaction={decorate,decoration={
        markings,
        mark=at position .5 with {\node at (0,0) {$\bullet$};}
      }},
      inner sep=.9ex,
      },
      prosmall/.style={postaction={decorate,decoration={
        markings,
        mark=at position .5 with {\node at (0,0) {$\scriptstyle \bullet$};}
      }},
      inner sep=.9ex,
      },
  n/.style={double equal sign distance, -implies}, t/.style={double distance=2.5pt, -implies, postaction={draw,-}},
}

\newcommand{\arrowdot}{
\ensuremath{\begin{tikzpicture}
\node (A) at (0,-.4) {};
\node (B) at (.4,-.4) {};
\draw[->, line width=.1ex] (0,-.6) -- (.4,-.6);
\node[shape=circle, fill=black, scale=0.35] (A) at  (.17,-.6) {};
\end{tikzpicture}
}}

\newcommand{\Arrowdot}{
\ensuremath{\begin{tikzpicture}
\node (A) at (0,-.4) {};
\node (B) at (.4,-.4) {};
\draw[n, line width=.1ex] (0,-.6) -- (.4,-.6);
\node[shape=circle, fill=black, scale=0.35] (A) at  (.17,-.6) {};
\end{tikzpicture}
}}

\makeatletter
\def\makeslashed#1#2#3#4#5{#1{\mathpalette{\sla@{#2}{#3}{#4}}{#5}}}
\def\@mathlower#1#2#3{\setbox0=\hbox{$\m@th#2#3$}\lower#1\ht0\box0}
\def\mathlower#1#2{\mathpalette{\@mathlower{#1}}{#2}}

\makeatother
\newcommand\dhxrightarrow[2][]{%
\mathrel{\ooalign{$\xrightarrow[#1\mkern4mu]{#2\mkern4mu}$\cr%
\hidewidth$\rightarrow\mkern4mu$}}
}

\newcommand\tailxrightarrow[2][]{%
\mathrel{\ooalign{$\xrightarrow[#1\mkern4mu]{#2\mkern4mu}$\cr%
\hidewidth$\Yright\mkern14mu$}}
}

\newcommand{\fto}{\twoheadrightarrow}
\newcommand{\cto}{\rightarrowtail}
\newcommand{\wto}{\xrightarrow{{\smash{\mathlower{0.8}{\sim}}}}}
\newcommand{\cwto}{\tailxrightarrow{{\smash{\mathlower{0.8}{\sim}}}}}
\newcommand{\fwto}{\dhxrightarrow{{\smash{\mathlower{0.8}{\sim}}}}}

\newcommand{\Set}{\mathrm{Set}}

\newcommand{\folds}{\mathrm{\textbf{FOLDS}}}

\newcommand{\Fib}{\mathrm{Fib}}
\newcommand{\Cof}{\mathrm{Cof}}

\newcommand{\an}{\ensuremath{\mathrm{An}}}
\newcommand{\trfib}{\ensuremath{\mathrm{TrFib}}}
\newcommand{\cof}{\ensuremath{\mathrm{cof}}}

\newcommand{\nfib}{\ensuremath{\mathrm{NFib}}}

\newcommand{\mono}{\ensuremath{\mathrm{mono}}}
\newcommand{\epi}{\ensuremath{\mathrm{epi}}}

\newcommand{\Lan}{\mathrm{Lan}}

\newcommand{\var}{\mathrm{var}}

\newcommand{\hor}{\mathrm{hor}}
\newcommand{\ver}{\mathrm{ver}}
\newcommand{\shor}{\mathrm{shor}}
\newcommand{\sver}{\mathrm{sver}}
\renewcommand{\deg}{\mathrm{deg}}

\newcommand{\fibrant}{\mathrm{fib}}

\newcommand{\Path}{\mathrm{Path}}
\newcommand{\id}{\mathrm{id}}

\newcommand{\op}{\mathrm{op}}
\newcommand{\hop}{\mathrm{hop}}

\newcommand{\Sq}{\mathrm{Sq}}

\newcommand{\ps}{\mathrm{ps}}

\renewcommand{\cE}{\mathcal{E}}
\newcommand{\C}{\mathcal{C}}
\newcommand{\I}{\mathcal{I}}
\newcommand{\J}{\mathcal{J}}
\renewcommand{\cM}{\mathcal{M}}

\renewcommand{\cR}{\mathcal{R}}
\newcommand{\Wf}{\mathcal{W}_f}
\newcommand{\W}{\mathcal{W}}
\newcommand{\dblcat}{\mathrm{DblCat}}
\newcommand{\twocat}{2\mathrm{Cat}}

\newcommand{\cat}{\mathrm{Cat}}
\renewcommand{\cL}{\mathcal{L}}

\newcommand{\Eadj}{E_\mathrm{adj}}

\newcommand{\verteq}{\rotatebox{90}{$\,=$}}

\DeclareFontFamily{U}{dmjhira}{}
\DeclareFontShape{U}{dmjhira}{m}{n}{ <-> dmjhira }{}
\DeclareRobustCommand{\yo}{\text{\usefont{U}{dmjhira}{m}{n}\symbol{"48}}}

\newcommand{\sq}[5]{{#1}\colon({#4} \; ^{{#2}}_{\substack{{#3}}} \; {#5})}

\newlist{rome}{enumerate}{7}
\setlist[rome]{label=(\roman*)}

\newtheorem{theorem}{Theorem}[section]
\newtheorem{cor}[theorem]{Corollary}
\newtheorem{prop}[theorem]{Proposition}
\newtheorem{lem}[theorem]{Lemma}
\declaretheorem[name=Theorem,numbered=yes]{theoremA}

\theoremstyle{definition}
\newtheorem{defn}[theorem]{Definition}
\newtheorem{terminology}[theorem]{Terminology}
\newtheorem{ex}[theorem]{Example}
\newtheorem{notation}[theorem]{Notation}

\newtheorem{constr}[theorem]{Construction}

\theoremstyle{remark}
\newtheorem{rem}[theorem]{Remark}

\crefname{theorem}{Theorem}{Theorems}
\crefname{cor}{Corollary}{Corollaries}
\crefname{prop}{Proposition}{Propositions}
\crefname{lem}{Lemma}{Lemmas}

\crefname{defn}{Definition}{Definitions}
\crefname{terminology}{Terminology}{Terminologies}
\crefname{ex}{Example}{Examples}
\crefname{notation}{Notation}{Notations}
\crefname{descr}{Description}{Descriptions}
\crefname{constr}{Construction}{Constructions}

\crefname{rem}{Remark}{Remarks}

\makeatletter
\let\c@equation\c@theorem
\makeatother
\numberwithin{equation}{section}

\usepackage{blindtext}

\renewcommand\thepart{\Roman{part}.}

\makeatletter
\renewcommand\part{%
  \par
  \addvspace{4ex}%
  \@afterindenttrue
  \secdef\@part\@spart
}

\def\@part[#1]#2{%
    \ifnum \c@secnumdepth >\m@ne
      \refstepcounter{part}%
      
      \addcontentsline{toc}{section}{\hspace{-.5cm} \bfseries\thepart\hspace{1em}#1}%
    \else
      \addcontentsline{toc}{section}{#1}%
    \fi
    {\parindent \z@ \raggedright
     \interlinepenalty \@M
     \normalfont
     \thispagestyle{empty}
     \ifnum \c@secnumdepth >\m@ne
      \centering\large\textsc{\textbf{\thepart}}\nobreakspace
     \fi
     \centering\large\textsc{\textbf{#2}}
     \par}%
    \nobreak
    \vskip .3cm
    \@afterheading}
\def\@spart#1{%
      \addcontentsline{toc}{part}{#1}%
    {\parindent \z@ \raggedright
     \interlinepenalty \@M
     \normalfont
     \thispagestyle{plain}
     \centering\large\textsc{\textbf{#1}}
     \par}%
    \nobreak
    \vskip .3cm
    \@afterheading}
\makeatother

\title{On the equivalence invariance of formal category theory}

\author[P.\ Verdugo]{Paula Verdugo}
\address{Max-Planck-Institut für Mathematik, 53072 Bonn, Germany}
\email{verdugo@mpim-bonn.mpg.de}

\begin{document}

\setcounter{tocdepth}{1}

\begin{abstract}
We prove a result of equivalence invariance of formal category theory for statements that can be expressed within an equipment. To do this, we exploit Henry and Bardomiano Mart\'inez's link between Makkai's FOLDS (first order logic with dependent sorts) and abstract homotopy theory. In the process, we construct a model structure on the category $\dblcat$ of double categories and double functors, whose trivial fibrations are the double functors that are surjective on objects, full on horizontal and vertical morphisms, and fully faithful on squares, and whose fibrant objects are the equipments.
\end{abstract}
\maketitle

\tableofcontents


\section{Introduction}

Category theory abstracts the main aspects of constructions that we repeat in different contexts, to provide a framework where we can do them independently of specific information from particular instances. That abstraction lends itself to the same phenomena.  Nowadays, we are interested in different types of category theory, namely, ordinary category theory, enriched category theory, internal category theory, etc., where we want to have analogous kinds of notions\textemdash limits, adjunctions, Kan extensions, etc.\textemdash and prove the same kind of results\textemdash universal properties, monadicity theorems, etc. \emph{Formal category theory} seeks to identify and encode the essential elements of these constructions and proofs in order to provide a unified framework for them.

Since many of these notions can be defined in terms of categories, functors, and natural transformations, one might think that a $2$-category may be a sufficient framework for formal category theory. However, to capture universal properties of categorical notions, one needs an ``extra direction'' of arrows, and thus a more robust framework for formal category theory is a double category with a particular property known as an \emph{equipment}. Early on, in  \cite{CarboniKellyWood} Carboni, Kelly, and Wood explore incipient ideas of using squares relating relations and functions in equipment-like structures to advance on aspects of formal category theory. Verity takes this further in \cite{Verity} by using double categories and thus providing a more robust framework to study formal category theory; later Cruttwell and Shulman have also worked on this \cite{Shulman_Cruttwell2010,Shulman_equipments}.

A compelling case of the importance of this approach is made in recent work by Riehl and Verity \cite{elements}. They develop a significant part of (among other structures) $(\infty,1)$-category theory by means of studying their homotopy $2$-category. However, it turns out that some concepts cannot be formulated in such $2$-category and they use instead the notion of virtual equipment. Moreover, in future work they will show that the virtual equipment associated to the $\infty$-cosmoi of $(\infty,1)$-categories is actually a  (pseudo)equipment.

An \emph{equipment} is a double categorical framework for formal category theory in that 
\begin{enumerate}
    \item definitions of basic categorical concepts such as adjunctions or Kan extensions can be defined internally to any equipment, and
    \item expected theorems relating these concepts can be proven.
\end{enumerate}

The idea is that the objects of the equipment can be thought of as ``categories'' in some abstract sense, the vertical morphisms can be thought of as ``functors,'' the horizontal morphisms can be thought of as ``profunctors,'' and the squares can be thought of as ``natural transformations.'' A natural question to ask is whether two equipments that are ``equivalent'' in a suitable sense define the ``same formal category theory.'' The main result of this paper gives a precise meaning to both of these terms and shows that this does in fact hold.

\subsection*{First order language with dependent sorts}

The question about what mathematical statements are included in the phrase ``same formal category theory'' is a subtle one. Famously, not all mathematical statements about $1$-categories are invariant under equivalence; a well-known example of this is the statement ``this category has a single object''. Similarly, not all mathematical statements about $2$-categories are invariant under biequivalence; an example of this is the statement ``this $2$-category has a $2$-terminal object''\textemdash we can prove that any biequivalent $2$-category will have a \emph{biterminal} object but not necessarily $2$-terminal. Experts will recognize that these examples implicitly involve equality between objects in the first case, and between ($1$)-morphisms in the latter.

In each of the settings above, Makkai introduced a formal language in which it is not possible to express statements of that kind. That is, all statements in Makkai's formal language for $1$-category theory are invariant under equivalence of $1$-categories, while all statements in Makkai's formal language for $2$-category theory are invariant under equivalence of $2$-categories. Makkai's formal framework is called First Order Logic with Dependent Sorts (FOLDS) and is defined in reference to a \emph{signature} category, which can specialize to the case of $1$- and $2$-category theory \cite{makkai}.

In related work, Riehl and Verity apply Makkai's FOLDS to develop a formal language for \emph{virtual} equipments and establish the equivalence invariance of formal categorical results in that related framework \cite{elements}, which they use to establish the model independence of $(\infty,1)$-category theory. It would be interesting to explore whether one can also apply the methods described next to this case.

\subsection*{Quillen model structures}

Another formalism for equivalence-invariant mathematics are Quillen model structures \cite{Quillen}. The idea in this context is that additional structure borne by a $1$-category can be used to guarantee that certain constructions involving its objects are equivalence invariant. Part of the structure is a class of ``weak equivalences'' which are meant to codify the appropriate notion of equivalence in a given framework.

As a consequence, whenever we are in front of a reasonable notion of ``sameness'' in a category, it is natural to ask whether we can find a model structure that allows us to study the homotopy theory captured by such a class, i.e.\ a model structure whose weak equivalences are exactly such a class. Notably, there is a canonical model structure on $\cat$ whose weak equivalences are precisely the equivalences of categories, as well as a canonical model structure on $\twocat$ whose weak equivalences are the biequivalences between $2$-categories. With the previous discussion in mind, these remarks raise two natural questions. Namely,
\begin{enumerate}
    \item What is the ``correct'' notion of equivalence for double categories, or more importantly, between equipments?
    \item Is there a model structure on the category $\dblcat$ of double categories and double functors whose weak equivalences are the double functors that answer the first question?
\end{enumerate}

In terms of question (1), we want to use \emph{double biequivalences}, which are a $2$-categorical analogue of notions of equivalences between double categories that are already present in the literature; see \cref{defn:doublebieq}. For (2), we know that double biequivalences have been used as weak equivalences, or closely related to weak equivalences \cite{MSV,whi}. However, with the foresight of equipments playing a fundamental role, we are after a model structure structure whose \emph{fibrant objects}\textemdash which are especially well-behaved objects in a model structure\textemdash are exactly the equipments. In this paper we answer that question on the affirmative, for now the reader may ignore the parenthesis in the statement below; a precise statement is found in \cref{MS_equipments}. 

\begin{theoremA}\label{intro_MS_equipments} There is a model structure on $\dblcat$ whose fibrant objects are the equipments (and whose cofibrations are the $\cI$-cofibrations of \cref{def:cofib}). Moreover, the weak equivalences between equipments are the double biequivalences.
\end{theoremA}

Given the preeminence of equipments in the literature of formal category theory, this result should be of independent interest.

\subsection*{Equivalence invariance}
In forthcoming work, Henry and Bardomiano Mart\'inez show that these settings can be related, explaining some of the technical results of Makkai's FOLDS \cite{HenryBardomiano}. Following their ideas, we prove that if the FOLDS signature is connected to a model structure in a particular way (see \cref{sec:homotopy_meets_folds}), then a weak equivalence between fibrant objects preserves the truth of any mathematical statement that is expressible in the FOLDS language. 

Applying this, we achieve the objectives outlined above:
\begin{enumerate}
    \item giving precise meaning to a ``language for formal category theory,'' and
    \item proving that double biequivalent equipments have the same formal category theory.
\end{enumerate}

We verify that our model structure of \cref{intro_MS_equipments} is connected to the FOLDS signature for double categories in the required way, which allows us to deduce the desired result which we prove in a precise way in \cref{double_equiv_same_formulae} but paraphrase as follows.

\begin{theoremA} If $F\colon\bA\to\bB$ is double biequivalence between equipments, then it preserves the truth of any mathematical statement that is expressible in the FOLDS language.
\end{theoremA}

\subsection*{Acknowledgments}
This is a chapter of the author's PhD thesis. The author is deeply grateful to Emily Riehl for presenting the question in the first place, generously sharing her ideas and insight throughout the preparation, and her feedback on earlier drafts that greatly improved the presentation. The author is also deeply grateful to Dominic Verity for sharing his insight on FOLDS and many illuminating conversations on double categories during and before the preparation of this paper, for his thoughtful comments on early versions. The author would also like to thank both of them and Lyne Moser for their assistance in latex\textemdash especially for their help straightening up some disheveled diagrams. Thank you also to John Bourke that indicated a mistake in an earlier proof of \cref{prop:trivfibareweSq}, and to Lyne Moser and Maru Sarazola with whom we fixed it in \cite{MSVdouble_equivalences}. The author gratefully acknowledges support from the CoACT, and NSF grant DMS-1652600 for periods of this work.

\section{Background on model structures}\label{sec:intro_MS}

Model categories provide a framework to do homotopy theory in diverse contexts by abstracting fundamental properties of the category of topological spaces. One of the main components is a distinguished class of morphisms, called \emph{weak equivalences}, that encode the notion of sameness in the context involved.  Model categories, however, come with additional structure that allows for a strong machinery to construct intrinsically homotopical notions.

\subsection{Elementary definitions}
Model structures were introduced by Quillen in \cite{Quillen}, and have since  been incorporated in the research of different areas. Historically, the definition consists on giving a complete and cocomplete category together with the specification of three classes of distinguished morphisms of it, which satisfy a list of axioms. Nowadays it is common practice to encode these conditions with the notion of weak factorization systems, initiated by \cite{Joyal_Tierney}, and that is the approach we take in this brief introduction.

\begin{defn} Let $\cC$ be a category, and $l\colon A\to B$ and $r\colon X\to Y$ be two morphisms in $\cC$. When for every commutative diagram as below left there is a lift $h$ as below right

\begin{tz}
    \node[](1) {$A$};
    \node[right of=1](2) {$X$};
    \node[below of=1](3) {$B$};
    \node[below of=2](4) {$Y$};

    \draw[->] (1) to node[above,la]{$f$} (2);
    \draw[->] (3) to node[below,la]{$g$} (4);
    \draw[->] (1) to node[left,la]{$l$} (3);
    \draw[->] (2) to node[right,la]{$r$} (4);

    \node[right of=2,xshift=2cm](1) {$A$};
    \node[right of=1](2) {$X$};
    \node[below of=1](3) {$B$};
    \node[below of=2](4) {$Y$};

    \draw[->] (1) to node[above,la]{$f$} (2);
    \draw[->] (3) to node[below,la]{$g$} (4);
    \draw[->] (1) to node[left,la]{$l$} (3);
    \draw[->] (2) to node[right,la]{$r$} (4);
    \draw[->,dashed] (3) to node[above,la]{$h$} (2);
\end{tz}
with $f=hl$ and $g=rh$, we say that the map $r$ has the \emph{right lifting property} with respect to $l$, and equivalently that $l$ has the \emph{left lifting property} with respect to $r$.
\end{defn}

\begin{defn} Let $\cC$ be a category, and $\cI\subset\cC$ a class of morphisms in $\cC$. A morphism $f$ in $\cC$ has the right (resp.\ left) lifting property with respect to $\cI$ if $f$ has the right (resp.\ left) lifting property with respect to all morphisms in $\cI$.
\end{defn}

\begin{notation}
If $\mathcal{D}$ is any class of morphisms in $\cC$, we use $\mathcal{D}^\boxslash$ to denote the class of morphisms in $\cC$ that have the right lifting property with respect to every morphism in $\mathcal{D}$, and $^\boxslash\mathcal{D}$ to denote the class of morphisms in $\cC$ that have the left lifting property with respect to every morphism in $\mathcal{D}$. 
\end{notation}

We now can define weak factorization systems. 

\begin{defn}Let $\cC$ be a category. A weak factorization system on $\cC$ is a pair $(\mathcal{L},\mathcal{R})$ of classes $\mathcal{L}$ and $\mathcal{R}$ of morphisms in $\cC$ such that:
\begin{itemize}
    \item[(wfs1)] every morphism $f$ in $\cC$ can be factored as $f=rl$ with $r$ in $\mathcal{R}$ and $l$ in $\mathcal{L}$, and
    \item[(wfs2)] $\mathcal{R}$ consists of all the morphisms that have the right lifting property with respect to $\mathcal{L}$ and $\mathcal{L}$ consists of all the morphisms that have the left lifting property with respect to $\mathcal{R}$\textemdash that is $\mathcal{R}=\mathcal{L}^\boxslash$ and $\mathcal{L}=^\boxslash\!\mathcal{R}$.
\end{itemize}
\end{defn}

\begin{ex}\label{ex:wfs_fromI} In any locally presentable category $\cC$, any set $\cI$ of morphisms of $\cC$ produces a weak factorization system $(^\boxslash(\mathcal{I}^\boxslash), \mathcal{I}^\boxslash)$. Moreover, it admits a functorial factorization that is accessible.
\end{ex}

The next result consists of basic properties of weak factorization systems that are useful to have in mind when using them to define model structures. 

\begin{prop} Let $(\mathcal{L},\mathcal{R})$ be a weak factorization system on a category $\cC$. Then the following statements hold. 
\begin{enumerate}
    \item The classes $\mathcal{L}$ and $\mathcal{R}$ contain isomorphisms and are closed under composition and retracts. 
    \item The left class $\mathcal{L}$ is closed under coproducts, pushouts, and transfinite compositions. 
    \item The right class $\mathcal{R}$ is closed under products, pullbacks, and transfinite inverse limits.
\end{enumerate}
\end{prop}

\begin{defn}A weak factorization system $(\cL,\mathcal{R})$ on a category $\cC$ is \emph{cofibrantly generated} if there exists a set $\cI$ of morphisms in $\cC$ such that $\cR=\cI^\boxslash$. We call such $\cI$ the \emph{generating set} of the weak factorization system. 
\end{defn}

\begin{rem} When $\cI$ is the generating set of a weak factorization system  $(\cL,\mathcal{R})$ on a category $\cC$ we have that $\cL=^\boxslash\! (\cI^\boxslash)$, and consequently $\cI\subset \cL$.
\end{rem}

\begin{defn}A \emph{model category} is a complete and cocomplete category $\mathcal{M}$, together with three distinguished classes of morphisms in $\mathcal{M}$:
\begin{rome}
    \item a class $\Cof$ of cofibrations ($\cto$),
    \item a class $\Fib$ of fibrations ($\fto$), and
    \item a class $\mathcal{W}$ of weak equivalences ($\wto$),  
\end{rome}
verifying the following conditions. 
\begin{itemize}
    \item[(mc1)] The class $\mathcal{W}$ satisfies the 2-out-of-3 condition, this is, given two composable morphisms $f$ and $g$, if any two of the three morphisms $f,g$ and $fg$ are in $\W$ then so is the third.
    \item[(mc2)] The pairs $(\Cof,\Fib\cap\W)$ and $(\Cof\cap\W,\Fib)$ are weak factorization systems on $\mathcal{M}$.
\end{itemize}
\end{defn}

We often write $(\mathcal{M},\Cof,\Fib,\W)$ or just $\mathcal{M}$ to denote a model category as above, and in that case we say that $(\Cof,\Fib,\W)$ is a model structure on $\mathcal{M}$.

\begin{defn}In a model category $(\mathcal{M},\Cof,\Fib,\W)$, 
\begin{rome}
    \item a morphism in $\Cof\cap\W$ is called a \emph{trivial cofibration} ($\cwto$), and
    \item a morphism in $\Fib\cap\W$ is called a \emph{trivial fibration} ($\fwto$).
\end{rome}
\end{defn}

\begin{defn}A model category $(\cM,\Cof,\Fib,\cW)$ is \emph{cofibrantly generated} if the weak factorization systems $(\Cof,\Fib\cap\cW)$ and $(\Cof\cap\cW,\Fib)$ are cofibrantly generated. Let $\cI$ and $\cJ$ be their respective generating sets, then a morphism in $\cI\subseteq\Cof$ is called \emph{generating cofibration} and a morphism in $\cJ\subseteq\Cof\cap\cW$ is called \emph{generating trivial cofibration}.
\end{defn}

Since we require that model structures are complete and cocomplete, they admit both an initial and a terminal object. This allows us to define two distinguished classes of objects that play an important role when studying the formal homotopy theory determined by a model structure.

\begin{defn}Let $\mathcal{M}$ be a model category, and let $\emptyset$ be its initial object and $\ast$ its terminal object. An object $X$ in $\mathcal{M}$ is called \emph{cofibrant} when the unique morphism $\emptyset\to X$ is a cofibration, and it is called \emph{fibrant} when the unique morphism $X\to\ast$ is a fibration.
\end{defn}

It is known that fibrant-cofibrant objects of a model structure $\cM$ are especially well behaved. This motivates the following definitions, which give a method to replace any given $X\in\cM$ by a weakly equivalent object that is both fibrant and cofibrant. 

\begin{defn}\label{def:firbant_replacement} Let $(\mathcal{M},\Cof,\Fib,\W)$ be a model category, and $X$ and object in $\cM$.
\begin{rome}
    \item Since the pair $(\Cof,\Fib\cap\cW)$ is a weak factorization system on $\cM$, the unique morphism $\emptyset\to X$ from the initial object to $X$ admits a factorization in $\cM$ as 

\begin{tz}
    \node[](1) {$\emptyset$};
    \node[right of=1](2) {$X^c$};
    \node[right of=2](3) {$X$};

    \draw[->>] (2) to node[below, la]{$\sim$} node[above,la]{$q_X$} (3);
    \draw[>->] (1) to (2);
\end{tz}
where $X^c$ is a cofibrant object, and $q_X\colon X^c\to X$ is a trivial fibration. The pair $(X^c,q_X)$ is called a \emph{cofibrant replacement} of $X$; we will often call just $X^c$ a cofibrant replacement. 

\item Since the pair $(\Cof\cap\cW,\Fib)$ is a weak factorization system on $\cM$, the unique morphism $X\to \ast$ from $X$ to the final object admits a factorization in $\cM$ as 

\begin{tz}
    \node[](1) {$X$};
    \node[right of=1](2) {$X^f$};
    \node[right of=2](3) {$\ast$};

    \draw[->>] (2) to (3);
    \draw[>->] (1) to node [below, la]{$\sim$} node[above,la]{$j_X$} (2);
\end{tz}
where $X^f$ is a fibrant object, and $j_X\colon X\to X^f$ is a trivial cofibration. The pair $(X^f,j_X)$ is called a \emph{fibrant replacement} of $X$; we will often call just $X^f$ a fibrant replacement. 

\end{rome}
\end{defn}

\subsection{Fibrantly induced model structures} Given the data of a (potential) model category, it is often hard to prove by hand that such data satisfy the required axioms. For this reason, several approaches have been developed to transfer a model structure from a known one, or construct it from certain known data; an especially effective one consists on right/left-lifting a model structure in the presence of an adjunction, see \cite{HKRS} and \cite{GKR_lifting_AMS}. In this subsection, we summarize another result on the existence of a model structure from certain data, from \cite{fibrantly_induced}.

This result is designed to be used in a setting where one has a locally presentable category $\cC$, and classes of desired cofibrations and fibrant objects in mind, together with classes of weak equivalences and fibrations between these fibrant objects, and wishes to extend this to a model structure on the category $\cC$. We rely on the existence of an auxiliary weak factorization system $(\an,\nfib)$ generated by a set $\J$, that serves a similar role as the \emph{anodyne extensions} and \emph{naive fibrations} of \cite{cisinski}, from which we borrow this terminology.

\begin{defn}
An object $X\in \C$ is \emph{naive fibrant} if the unique morphism $X\to 1$ to the terminal object is a naive fibration. 
\end{defn}

\begin{defn}
Given an object $X\in\C$, a \emph{naive fibrant replacement of $X$} is an anodyne extension
\[\iota \colon X \to X'\]
such that $X'$ is naive fibrant. Similarly, given a morphism $f \colon X\to Y$ in $\C$, \emph{naive fibrant replacement of  $f$} is a commutative square
\begin{tz}
\node[](1) {$X$}; 
\node[right of=1](2) {$Y$}; 
\node[below of=1](1') {$X'$}; 
\node[below of=2](2') {$Y'$}; 

\draw[->] (1) to node[above,la]{$f$} (2); 
\draw[->] (1) to node[left,la]{$\iota_X$} (1'); 
\draw[->] (2) to node[right,la]{$\iota_Y$} (2'); 
\draw[->] (1') to node[below,la]{$f'$} (2');
\end{tz}
where $X'$, $Y'$ are naive fibrant, and $\iota_X$, $\iota_Y$ are anodyne extensions.
\end{defn}

\begin{rem}
Note that we always have naive fibrant replacements, as we can build them from the  factorization system $(\an,\nfib)$.
\end{rem}

We stated that we would also have in mind a class $\Wf$ of morphisms between naive fibrant objects that contains isomorphisms, which we think of as weak equivalences between fibrant objects. With this class and the previous definitions, we construct a new class of morphisms $\W$ which will be the (global) weak equivalences.

\begin{defn}\label{def:we}
A morphism $f \colon X \to Y$ in $\C$ is a \emph{weak equivalence} if there exists a naive fibrant replacement of $f$
\begin{tz}
\node[](1) {$X$}; 
\node[right of=1](2) {$Y$}; 
\node[below of=1](1') {$X'$}; 
\node[below of=2](2') {$Y'$}; 

\draw[->] (1) to node[above,la]{$f$} (2); 
\draw[->] (1) to node[left,la]{$\iota_X$} (1'); 
\draw[->] (2) to node[right,la]{$\iota_Y$} (2'); 
\draw[->] (1') to node[below,la]{$f'$} (2');
\end{tz}
such that $f'$ is in $\Wf$. We denote by $\W$ the class of weak equivalences.
\end{defn}

Recall that in a locally presentable category $\cC$, the pair $(^\boxslash(\cI^\boxslash),\I^\boxslash)$ forms a weak factorization system for any set $\cI$ of morphisms in $\cC$.

\begin{theorem}[{\cite[Theorem 2.8 and Proposition 2.21]{fibrantly_induced}}]\label{thm:fibrantly_generated}
Let $\C$ be a locally presentable category, $\I$ be a set of morphisms in $\C$, and $(\an,\nfib)$ be a weak factorization system in $\C$ generated by a set such that $\an\subseteq\ ^\boxslash(\cI^\boxslash)$. Suppose in addition that we have a class $\Wf$ of morphisms in $\C$ between naive fibrant objects such that
\begin{enumerate}[label=(\arabic*)]
    \item\label{ax:trivfib} $\cI^\boxslash\subseteq\W$, where $\W$ is described in \cref{def:we},
    \item\label{2of6Wf} $\Wf$ has 2-out-of-6, 
    \item \label{accessibility} there exists a class $\overline{\cW}$ of morphisms such that $\Wf$ is the restriction of $\overline{\cW}$ to the morphisms between naive fibrant objects and $\overline{\cW}$ considered as a full subcategory of $\C^{\mathbbm{2}}$ is accessible,
 \item \label{Path} for every naive fibrant object $X$, there is a factorization of the diagonal morphism 
    \[ X\xrightarrow{w} \Path X\xrightarrow{p} X\times X \]
    such that $w\in \Wf$ and $p\in \nfib$,
    \item\label{fibwe} $\nfib\cap\Wf\subseteq \cI^\boxslash$.
\end{enumerate}
Then there exists a combinatorial model structure on $\C$ with cofibrations given by the $\I$-cofibrations, fibrant objects given by the naive fibrant objects, and weak equivalences given by the morphisms in $\W$. Furthermore, weak equivalences (resp.\ fibrations) between fibrant objects are precisely the morphisms in $\Wf$ (resp.\ $\nfib$).
\end{theorem}

\section{Background on double categories}\label{sec:intro_dblcats}
In this section we present a different kind of $2$-dimensional categories. As before, we briefly present the main aspects of the theory that we will need. The interested reader can refer to \cite{Grandis} for a detailed account of double categories, and \cite{KellyStreet} for a more concise treatment. 

\begin{defn}\label{def:dblcat}
A \emph{double category} $\bA$ consists of
\begin{rome}
\item objects $A$, $B$, $C$, $\ldots$,
\item horizontal morphisms $f\colon A\arrowdot A'$ with composition denoted by $g\circ f$ or $gf$,
\item vertical morphisms $u\colon A\to B$ with composition denoted by $v\bullet u$ or $vu$,
\item squares (or $2$-cells) $\sq{\alpha}{f}{g}{u}{v}$ of the form 
\begin{tz}
\node (A) at (0,0) {$A$};
\node (B) at (1.5,0) {$A'$};
\node (A') at (0,-1.5) {$B$};
\node (B') at (1.5,-1.5) {$B'$};
\draw[->,pro] (A) to node[above,scale=0.8] {$f$} (B);
\draw[->,pro] (A') to node[below,scale=0.8] {$g$} (B');
\draw[->] (A) to node[left,scale=0.8] {$u$} (A');
\draw[->] (B) to node[right,scale=0.8] {$v$} (B');

\node[scale=0.8] at (.75,-.75) {$\alpha$};
\end{tz}
with both horizontal composition along their vertical boundaries and vertical composition along their horizontal boundaries, and
\item horizontal identities $\id_A\colon A\to A$ and vertical identities $e_A\colon A\arrowdot A$ for each object~$A$, vertical identity squares $\sq{e_f}{f}{f}{\id_A}{\id_B}$ for each horizontal morphism $f\colon A\arrowdot A'$, horizontal identity squares $\sq{\id_u}{\id_A}{\id_{B}}{u}{u}$ for each vertical morphism $u\colon A\to B$, and identity squares $\square_A=\id_{e_A}=e_{\id_A}$ for each object $A$,
\end{rome}
such that all compositions are unital and associative, and such that the horizontal and vertical compositions of squares satisfy the interchange law.
\end{defn}

\begin{defn}
Let $\bA$ and $\bB$ be double categories. A \emph{double functor} $F\colon \bA\to \bB$ consists of maps on objects, horizontal morphisms, vertical morphisms, and squares, which are compatible with domains and codomains and preserve all double categorical compositions and identities strictly.
\end{defn}

\begin{notation}
We write $\dblcat$ for the category of double categories and double functors.
\end{notation}

The category of double categories is cartesian closed, we proceed to present the necessary definitions to describe the internal hom. See \cite[\S 3.2.7]{Grandis} for more explicit definitions.

\begin{defn} \label{def:natural_transformations}
Let $F,G\colon \bA\to \bB$ be double functors. A \emph{horizontal natural transformation} $h\colon F\Rightarrow G$ consists of 
\begin{rome}
    \item a horizontal morphism $h_A\colon FA\arrowdot GA$ in $\bB$, for each object $A\in \bA$, and
    \item a square $\sq{h_u}{h_A}{h_{A'}}{Fu}{Gu}$ in $\bB$, for each vertical morphism $u\colon A\to A'$ in $\bA$,
\end{rome}  
such that the assignment of squares is functorial with respect to the composition of vertical morphisms, and these data satisfy a naturality condition with respect to horizontal morphisms and squares.

Similarly, let $F'\colon\bA\to\bB$ be a double functor, a \emph{vertical natural transformation} $r\colon F\Rightarrow F'$ consists of 
\begin{rome}
    \item a vertical morphism $r_A\colon FA\to FB$ in $\bB$, for each object $A\in \bA$, and
    \item a square $\sq{r_f}{Ff}{F'f}{r_A}{r_{A'}}$ in $\bB$, for each horizontal morphism $f\colon A\arrowdot A'$ in $\bA$,
\end{rome}  
satisfying transposed conditions.
\end{defn}

\begin{defn}\label{def:modif} Given horizontal natural transformations $h\colon F\Rightarrow G$ and $k\colon F'\Rightarrow G'$, and vertical natural transformations $r\colon F\Rightarrow F'$ and $s\colon G\Rightarrow G'$, a \emph{modification} $\sq{\mu}{h}{k}{r}{s}$ consists of 
\begin{rome}
    \item a square $\sq{\mu_A}{h_A}{k_A}{r_A}{s_A}$ in $\bB$, for each object $A\in \bA$,
\end{rome} 
satisfying horizontal and vertical coherence conditions with respect to the squares of the natural transformations $h$, $k$, $r$, and $s$.
\end{defn}

\begin{defn}\label{internalHom}
Let $\bA$ and $\bB$ be double categories. We define the \emph{hom double category} $\llbracket\bA,\bB\rrbracket$ whose 
\begin{rome}
    \item objects are the double functors $\bA\to \bB$,
    \item horizontal morphisms are the horizontal natural transformations,
    \item vertical morphisms are the vertical natural transformations, and
    \item squares are the modifications.
\end{rome} 
\end{defn}

\begin{prop}[{\cite[Proposition 2.11]{FPP}}] \label{prop:adjinternalhom}
For every double category $\bA$, there is an adjunction

$$\dblcat(\bA,\llbracket\bB,\bC\rrbracket))\cong\dblcat(\bA\times\bB,\bC)$$

\end{prop}

\begin{rem}\label{rem:upgraded_adjunction_dbl_hom}Using the associativity of $\times$ and the fact that every double category is the colimit of the terminal double category $\mathbb{1}$, the double category $\bH\mathbb{2}$ free on a horizontal morphism, the double category $\bV\mathbb{2}$ free on a vertical morphism, and the double category $\bH\mathbb{2}\times\bV\mathbb{2}$ free on a non-trivial square, we can see that the isomorphism above extends to an isomorphism of double categories 
$$\llbracket \bA,\llbracket\bB,\bC\rrbracket\rrbracket\cong\llbracket \bA\times\bB,\bC\rrbracket.$$
    
\end{rem}

It is usual to want to consider laxer conditions in \cref{def:natural_transformations} and adjusted coherence conditions in \cref{def:modif}, as we briefly describe below; see \cite[\S 3.8]{Grandis} for precise definitions.

\begin{defn}\label{def:pseudo_natural_transformations}
Let $F,G\colon \bA\to \bB$ be double functors. A \emph{horizontal pseudo natural transformation} $h\colon F\Rightarrow G$ consists of 
\begin{rome}
    \item a horizontal morphism $h_A\colon FA\arrowdot GA$ in $\bB$, for each object $A\in \bA$, 
    \item a square $\sq{h_u}{h_A}{h_{A'}}{Fu}{Gu}$ in $\bB$, for each vertical morphism $u\colon A\to A'$ in $\bA$, and
    \item a vertically invertible square $\sq{h_f}{(Gf) h_A}{h_B (Ff)}{e_{FA}}{e_{GB}}$ in $\bB$, for each horizontal morphism $f\colon A\arrowdot B$ in $\bA$, expressing a pseudo naturality condition for horizontal morphisms.
    \end{rome} 
These assignments of squares are functorial with respect to compositions of horizontal and vertical morphisms, and these data satisfy a naturality condition with respect to squares.

Similarly, one can define a transposed notion of \emph{vertical pseudo natural transformation} between double functors.
\end{defn}

\begin{defn}\label{def:modif_ps}
A \emph{modification} in a square of horizontal and vertical pseudo natural transformations is defined similarly to \cref{def:modif}, with the horizontal and vertical coherence conditions taking the pseudo data of the transformations into account. See \cite[Definition 3.8.3]{Grandis}
\end{defn}

\begin{defn}\label{def:pseudohomdouble}
Let $\bA$ and $\bB$ be double categories. We define the \emph{pseudo hom double category} $\llbracket\bA,\bB\rrbracket_{\mathrm{ps}}$ whose
\begin{rome}
    \item objects are the double functors $\bA\to \bB$,
    \item horizontal morphisms are the horizontal pseudo natural transformations,
    \item vertical morphisms are the vertical pseudo natural transformations, and
    \item squares are the modifications.
\end{rome}
\end{defn}

In \cite{Bohm} B\"ohm introduced a Gray tensor product for $\dblcat$ such that, analogously to the $2$-categorical case, the corresponding hom-double categories make use of the notions of horizontal and vertical pseudo transformations, and modifications between them, assembling into \cref{def:pseudohomdouble}.

\begin{prop}[{\cite[\S 3]{Bohm}}] \label{prop:Bohm}
There is a symmetric monoidal structure on $\dblcat$ given by the Gray tensor product 
 \[ \otimes_\mathrm{Gr}\colon \dblcat\times \dblcat\to \dblcat. \]
 Moreover, this monoidal structure is closed. That is, for all double categories $\bA$, $\bB$, and $\bC$, there is an isomorphism
\[ \dblcat(\bA,\llbracket\bB,\bC\rrbracket_\mathrm{ps})\cong\dblcat(\bA\otimes_\mathrm{Gr} \bB, \bC), \]
natural in $\bA$, $\bB$ and $\bC$.
\end{prop}

\begin{rem}Similarly to \cref{rem:upgraded_adjunction_dbl_hom} the isomorphism above extends to an isomoprshim of double categories

$$\llbracket \bA,\llbracket\bB,\bC\rrbracket\rrbracket_{\ps}\cong\llbracket \bA\times\bB,\bC\rrbracket_{\ps}.$$
    
\end{rem}

One of the ways we can associate, a double category to any $2$-category is via the square functor, originally presented in \cite{ehresmann0}.

\begin{defn}
The \emph{square functor} $\Sq\colon \twocat\to \dblcat$ sends a $2$-category $\cA$ to the double category $\Sq\cA$ whose objects are the objects of $\cA$, whose horizontal and vertical morphisms are the morphisms of $\cA$, and whose squares 
\begin{tz}
\node[](1) {$A$}; 
\node[below of=1](2) {$A'$}; 
\node[right of=1](3) {$B$}; 
\node[right of=2](4) {$B'$};

\draw[->] (1) to node[left,la]{$u$} (2); 
\draw[->,pro] (1) to node[above,la]{$f$} (3); 
\draw[->,pro] (2) to node[below,la]{$f'$} (4); 
\draw[->](3) to node[right,la]{$v$} (4); 
 
\node[la] at ($(1)!0.5!(4)$) {$\alpha$};
\end{tz}
are the $2$-morphisms $\alpha\colon vf\Rightarrow f'u$ in $\cA$.
\end{defn}

\subsection{Double categories as internal categories}\label{subsec:dblcats_as_internal} To make some of our constructions more intuitive, it will be useful to see that double categories can be seen as interncal categories to the category $\cat$. In this section we recall the general notion of internal category, and show this concurrence. 

\begin{defn}
An \emph{internal category} $\bA$ to $\cE$ consists of the following data in $\cE$
\begin{rome}
\item two objects $\bA_0,\bA_1$, where $\bA_0$ is called the object of \emph{objects} and $\bA_1$ the object of \emph{morphisms},
\item source and target maps $s,t\colon \bA_1\to \bA_0$, 
\item a composition map $c\colon \bA_1\times_{\bA_0}\bA_1\to \bA_1$ with the pullback defined as depicted below, 
\item an identity map $i\colon \bA_0\to \bA_1$.
\end{rome}
In summary, the data of an internal category in $\cE$ is given by 

\begin{tz}
\node[](1) {$\bA_1\times_{\bA_0}\bA_1$};
\node[right of=1](2) {$\bA_1$};
\node[below of=1](3) {$\bA_1$};
\node[below of=2](4) {$\bA_0$};

\draw[->] (1) to node[above,la]{$\pi_r$} (2);
\draw[->] (1) to node[left,la]{$\pi_l$} (3);
\draw[->] (2) to node[right,la]{$t$} (4);
\draw[->] (3) to node[below,la]{$s$} (4);

\node[right of=2,yshift=-0.75cm,xshift=3cm](1)    {$\bA_1\times_{\bA_0}\bA_1$};
\node[right of=1,xshift=1em](2) {$\bA_1$};
\node[right of=2](3) {$\bA_0$};

\draw[->] (1) to node[above,la]{$\circ$} (2);
\draw[->] ($(2.east)-(0,10pt)$) to node[below,la]{$t$} ($(3.west)-(0,10pt)$);
\draw[->] ($(2.east)+(0,10pt)$) to node[above,la]{$s$} ($(3.west)+(0,10pt)$);
\draw[->] (3) to node[above,la]{$i$} (2);
\end{tz}

In addition, we ask this data to be subject to the commutative diagrams in $\cE$ below, that determine the source and target of identities and composites

\begin{tz}
\node[](1) {$\bA_0$}; 
\node[below of=1](2) {$\bA_1$}; 
\node[left of=2](3) {$\bA_0$};
\node[right of=2](4) {$\bA_0$}; 

\draw[->] (1) to node[left,la]{$i$} (2);
\draw[->] (2) to node[below,la]{$t$} (4);
\draw[d] (1) to (4);
\draw[d] (1) to (3);
\draw[->] (2) to node[below,la]{$s$} (3);

\node[right of=1,xshift=3cm](5) {$\bA_1$};
\node[below of=5](6) {$\bA_0$};
\node[right of=5,xshift=.5cm](1) {$\bA_1\times_{\bA_0} \bA_1$}; 
\node[below of=1](3) {$\bA_1$};
\node[right of=1,xshift=.5cm](2) {$\bA_1$}; 
\node[below of=2](4) {$\bA_0$}; 

\draw[->] (1) to node[above,la]{$\pi_1$} (2);
\draw[->] (2) to node[right,la]{$t$} (4);
\draw[->] (1) to node[left,la]{$c$} (3);
\draw[->] (3) to node[below,la]{$t$} (4);
\draw[->] (3) to node[below,la]{$s$} (6);
\draw[->] (5) to node[left,la]{$s$} (6);
\draw[->] (1) to node[above,la]{$\pi_0$} (5);
\end{tz}
as well as the commutative diagram below that encodes unitality of the composition

\begin{tz}
\node[](5) {$\bA_1\cong \bA_1\times_{\bA_0} \bA_0$};
\node[right of=5,xshift=2cm](1) {$\bA_1\times_{\bA_0} \bA_1$}; 
\node[below of=1](3) {$\bA_1$};
\node[right of=1,xshift=2cm](2) {$\bA_0\times_{\bA_0}\bA_1\cong \bA_1$}; 

\draw[->] (2) to node[above,la]{$i\times \id_{\bA_1}$} (1);
\draw[d] (2) to (3);
\draw[->] (1) to node[left,la]{$c$} (3);
\draw[d] (5) to (3);
\draw[->] (5) to node[above,la]{$\id_{\bA_1}\times i$} (1);
\end{tz}
 and the one below that encodes associativity of the composition.  
\begin{tz}
\node[](1) {$\bA_1\times_{\bA_0} \bA_1\times_{\bA_0} \bA_1$}; 
\node[below of=1](3) {$\bA_1\times_{\bA_0} \bA_1$};
\node[right of=1,xshift=3cm](2) {$\bA_1\times_{\bA_0} \bA_1$}; 
\node[below of=2](4) {$\bA_1$}; 

\draw[->] (1) to node[above,la]{$c\times \id_{\bA_1}$} (2);
\draw[->] (2) to node[right,la]{$c$} (4);
\draw[->] (1) to node[left,la]{$\id_{\bA_1}\times c$} (3);
\draw[->] (3) to node[below,la]{$c$} (4);
\end{tz}

\end{defn}

\begin{defn}
Let $\bA$ and $\bB$ be internal categories to $\cE$. An \emph{internal functor} $F\colon \bA\to \bB$ consists of maps $F_0\colon \bA_0\to \bB_0$ and $F_1\colon \bA_1\to \bB_1$ in $\cE$ satisfying the following conditions. 
\begin{enumerate}
    \item \emph{Compatibility with source and target}: the following diagrams commute in~$\cE$. 
\begin{tz}
\node[](1) {$\bA_1$}; 
\node[below of=1](3) {$\bB_1$};
\node[right of=1](2) {$\bA_0$}; 
\node[below of=2](4) {$\bB_0$}; 

\draw[->] (1) to node[above,la]{$s$} (2);
\draw[->] (2) to node[right,la]{$F_0$} (4);
\draw[->] (1) to node[left,la]{$F_1$} (3);
\draw[->] (3) to node[below,la]{$s$} (4);

\node[right of=2,xshift=3cm](1) {$\bA_1$}; 
\node[below of=1](3) {$\bB_1$};
\node[right of=1](2) {$\bA_0$}; 
\node[below of=2](4) {$\bB_0$}; 

\draw[->] (1) to node[above,la]{$t$} (2);
\draw[->] (2) to node[right,la]{$F_0$} (4);
\draw[->] (1) to node[left,la]{$F_1$} (3);
\draw[->] (3) to node[below,la]{$t$} (4);
\end{tz}
    \item \emph{Compatibility with identity}: the following diagram commutes in $\cE$. 
\begin{tz}
\node[](1) {$\bA_0$}; 
\node[below of=1](3) {$\bB_0$};
\node[right of=1](2) {$\bA_1$}; 
\node[below of=2](4) {$\bB_1$}; 

\draw[->] (1) to node[above,la]{$i$} (2);
\draw[->] (2) to node[right,la]{$F_1$} (4);
\draw[->] (1) to node[left,la]{$F_0$} (3);
\draw[->] (3) to node[below,la]{$i$} (4);
\end{tz}
    \item 
    \emph{Compatibility with composition}: the following diagram commutes in $\cE$. 
\begin{tz}
\node[](1) {$\bA_1\times_{\bA_0} \bA_1$}; 
\node[below of=1](3) {$\bB_1\times_{\bB_0} \bB_1$};
\node[right of=1,xshift=.5cm](2) {$\bA_1$}; 
\node[below of=2](4) {$\bB_1$}; 

\draw[->] (1) to node[above,la]{$c$} (2);
\draw[->] (2) to node[right,la]{$F_1$} (4);
\draw[->] (1) to node[left,la]{$F_1\times F_1$} (3);
\draw[->] (3) to node[below,la]{$c$} (4);
\end{tz}
\end{enumerate}
\end{defn}

\begin{notation}
We write $\cat (\cE)$ for the category of internal categories to $\cE$ and internal functors.
\end{notation}

\begin{rem}\label{rem:dblcat_internal} A double category is an internal category in $\cat$, as evidenced by the following diagram

\begin{tz}
    \node[](1)    {$\bA_{h,s}\times_{\bA_{o,v}}\bA_{h,s}$};
    \node[right of=1,xshift=2em](2) {$\bA_{h,s}$};
    \node[right of=2](3) {$\bA_{o,v}$};

    \draw[->] (1) to node[above,la]{$\circ$} (2);
    \draw[->] ($(2.east)-(0,10pt)$) to node[below,la]{$c$} ($(3.west)-(0,10pt)$);
    \draw[->] ($(2.east)+(0,10pt)$) to node[above,la]{$d$} ($(3.west)+(0,10pt)$);
    \draw[->] (3) to node[above,la]{$i$} (2);
\end{tz}
in which 
\begin{itemize}
    \item[-] $\bA_{o,v}$ is the category of objects and vertical morphisms, 
    \item[-] $\bA_{h,s}$ is the category of horizontal morphisms and squares,
    \item[-] the functors $d,c\colon\bA_{h,s}\to\bA_{o,v}$ that assign the domain/codomain object for each horizontal morphism, and the domain/codomain vertical morphism for each square,
    \item[-] the functor $i\colon\bA_{o,v}\to\bA_{h,s}$ that assigns the horizontal identity morphism for each object, and the identity square on a vertical morphism for each vertical morphism,
    \item[-] the functor $\bA_{h,s}\times_{\bA_{o,v}}\bA_{h,s}\xrightarrow{\circ}\bA_{h,s}$ given by the composition of horizontal morphisms, and the horizontal composition of squares.
\end{itemize}
\end{rem}

\begin{rem}
   Similarly, we could use the diagram below instead.
   \begin{tz}
    \node[](1)    {$\bA_{v,s}\times_{\bA_{o,h}}\bA_{v,s}$};
    \node[right of=1,xshift=2em](2) {$\bA_{v,s}$};
    \node[right of=2](3) {$\bA_{o,h}$};

    \draw[->] (1) to node[above,la]{$\circ$} (2);
    \draw[->] ($(2.east)-(0,10pt)$) to node[below,la]{$c$} ($(3.west)-(0,10pt)$);
    \draw[->] ($(2.east)+(0,10pt)$) to node[above,la]{$d$} ($(3.west)+(0,10pt)$);
    \draw[->] (3) to node[above,la]{$i$} (2);
\end{tz}
in which 
\begin{itemize}
    \item[-] $\bA_{o,h}$ is the category of objects and horizontal morphisms, 
    \item[-] $\bA_{v,s}$ is the category of vertical morphisms and squares,
    \item[-] the other components of the diagram are analogous to \cref{rem:dblcat_internal}.
\end{itemize}

\end{rem}

\begin{defn}
    Let $\cE$ be a category, we define the opposite operation on $\cat(\cE)$, denoted by $(-)^\mathrm{hop}\colon \cat(\cE)\to\cat{\cE}$ that flips the source and target maps. This is, given $\bA$ an internal category in $\cE$

    \begin{tz}
    \node[](1)    {$\bA_1\times_{\bA_0}\bA_1$};
    \node[right of=1,xshift=1em](2) {$\bA_1$};
    \node[right of=2](3) {$\bA_0$};

    \draw[->] (1) to node[above,la]{$\circ$} (2);
    \draw[->] ($(2.east)-(0,10pt)$) to node[below,la]{$t$} ($(3.west)-(0,10pt)$);
    \draw[->] ($(2.east)+(0,10pt)$) to node[above,la]{$s$} ($(3.west)+(0,10pt)$);
    \draw[->] (3) to node[above,la]{$i$} (2);
    \end{tz}
it gives us the internal category in $\cE$ determined by

    \begin{tz}
    \node[](1)    {$\bA_1\times_{\bA_0}\bA_1$};
    \node[right of=1,xshift=1em](2) {$\bA_1$};
    \node[right of=2](3) {$\bA_0$};

    \draw[->] (1) to node[above,la]{$\circ$} (2);
    \draw[->] ($(2.east)-(0,10pt)$) to node[below,la]{$s$} ($(3.west)-(0,10pt)$);
    \draw[->] ($(2.east)+(0,10pt)$) to node[above,la]{$t$} ($(3.west)+(0,10pt)$);
    \draw[->] (3) to node[above,la]{$i$} (2);
    \end{tz}

\end{defn}

\begin{rem} Let $\bA$ be a double category, seen as an internal category as in \cref{rem:dblcat_internal}. Then given a morphism in $\bA_{h,s}$ as below
    \begin{tz}
    \node[](1) {$A$};
    \node[right of=1](2) {$A'$};
    \node[below of=1](3) {$C$};
    \node[below of=2](4) {$C'$};

    \draw[->,pro] (1) to node[above,la]{$f$} (2);
    \draw[->,pro] (3) to node[below,la]{$g$} (4);
    \draw[->] (1) to node[left,la]{$u$} (3);
    \draw[->] (2) to node[right,la]{$v$} (4);

    \node[la] at ($(1)!0.5!(4)$) {$\alpha$};
    \end{tz}
we have that, in $\bA^\mathrm{hop}$, it is $\mathrm{dom}(f)=A'$,  $\mathrm{cod}(f)=A$, $\mathrm{vdom}(\alpha)=v$,  $\mathrm{vcod}(\alpha)=u$.
\end{rem}

\subsection{Companions in a double category}\label{subsec:companions}
Companion pairs relate in a canonical way horizontal morphisms and vertical morphisms in a double category; see for example \cite{DawsonParePronk} and \cite{GPAdjointsDblCats}.

\begin{defn}
A \emph{companion pair} in a double category $\bA$ is a tuple $(f,u,\varphi,\psi)$ consisting of a horizontal morphism $f\colon A\to B$, a vertical morphism $u\colon A\arrowdot B$, and two squares 
\begin{tz}
\node[](1) {$A$}; 
\node[below of=1](2) {$B$}; 
\node[right of=1](3) {$B$}; 
\node[right of=2](4) {$B$};

\draw[->] (1) to node[left,la]{$u$} (2); 
\draw[->,pro] (1) to node[above,la]{$f$} (3); 
\draw[d,pro] (2) to (4); 
\draw[d](3) to (4); 
 
\node[la] at ($(1)!0.5!(4)$) {$\varphi$};

\node[right of=3,xshift=2cm](1) {$A$}; 
\node[below of=1](2) {$A$}; 
\node[right of=1](3) {$A$}; 
\node[right of=2](4) {$B$};

\draw[->] (3) to node[right,la]{$u$} (4); 
\draw[->,pro] (2) to node[below,la]{$f$} (4); 
\draw[d,pro] (1) to (3); 
\draw[d](1) to (2); 
 
\node[la] at ($(1)!0.5!(4)$) {$\psi$};
\end{tz}
satisfying the following pasting equalities. 
\begin{tz}
\node[](1') {$A$}; 
\node[below of=1'](2') {$A$};
\node[right of=1'](1) {$A$}; 
\node[below of=1](2) {$B$}; 
\node[right of=1](3) {$B$}; 
\node[right of=2](4) {$B$};

\draw[->] (1) to node[left,la]{$u$} (2); 
\draw[->,pro] (1) to node[above,la]{$f$} (3); 
\draw[d,pro] (2) to (4); 
\draw[d](3) to (4); 
 
\node[la] at ($(1)!0.5!(4)$) {$\varphi$};
\draw[->,pro] (2') to node[below,la]{$f$} (2); 
\draw[d,pro] (1') to (1); 
\draw[d](1') to (2'); 
 
\node[la] at ($(1')!0.5!(2)$) {$\psi$};

\node[right of=3](1) {$A$}; 
\node at ($(1)!0.5!(4)$) {$=$};
\node[below of=1](2) {$A$}; 
\node[right of=1](3) {$B$}; 
\node[below of=3](4) {$B$}; 
\draw[d] (1) to (2); 
\draw[d] (3) to (4); 
\draw[->,pro] (1) to node[above,la]{$f$} (3); 
\draw[->,pro] (2) to node[below,la]{$f$} (4); 
\node[la] at ($(1)!0.5!(4)$) {$e_f$};

\node[right of=3,xshift=1cm,yshift=.75cm](1) {$A$}; 
\node[below of=1](2) {$A$}; 
\node[right of=1](3) {$A$}; 
\node[right of=2](4) {$B$};
\node[below of=2](2') {$B$}; 
\node[right of=2'](4') {$B$};

\draw[->] (2) to node[left,la]{$u$} (2'); 
\draw[d,pro] (2') to (4'); 
\draw[d](4) to (4'); 
 
\node[la] at ($(2)!0.5!(4')$) {$\varphi$};

\draw[->] (3) to node[right,la]{$u$} (4); 
\draw[->,pro] (2) to node[above,la]{$f$} (4); 
\draw[d,pro] (1) to (3); 
\draw[d](1) to (2); 
 
\node[la] at ($(1)!0.5!(4)$) {$\psi$};

\node[right of=3,yshift=-.75cm](1) {$A$}; 
\node[below of=1](2) {$B$}; 
\draw[->] (1) to node(a)[left,la]{$u$} (2);
\node at ($(a)!0.5!(4)$) {$=$};
\node[right of=1](3) {$A$}; 
\node[below of=3](4) {$B$}; 
\draw[d,pro] (1) to (3); 
\draw[d,pro] (2) to (4); 
\draw[->] (3) to node[right,la]{$u$} (4);
\node[la] at ($(1)!0.5!(4)$) {$\id_u$};
\end{tz}
We say that a horizontal morphism $f$ (resp.\ a vertical morphism $u$) \emph{has a vertical (resp.\ horizontal) companion} if there is a companion pair $(f,u,\varphi,\psi)$.
\end{defn}

\begin{rem}\label{unicitycompanion} Companions are unique up to unique $2$-isomorophism. See \cite[Remark 4.17]{fibrantly_induced} or \cite[Section 4.1.1]{Grandis}.
\end{rem}

\begin{ex} \label{ex:companions} Given a $2$-category $\cA$, any two isomorphic $1$-cells are companions to each other in $\Sq\cA$. In particular, the square category $\Sq\cA$ of a $2$-category $\cA$ has all companions as it is enough to consider the identity $2$-morphism. 
\end{ex}

\begin{lem}\label{lem:Sq2_free}The double category $\Sq\mathbb{2}$ is the free double category on a companion pair. 
\end{lem}

\begin{proof} Let us write $f\colon A\to B$ for the only non-trivial $1$-cell in the $2$-category $\mathbb{2}$, and we write $\id_f$ for the identity $2$-cell at $f$. By inspection, we can see that the only squares in $\Sq\mathbb{2}$ are as below

\begin{tz}
\node[](1) {$A$}; 
\node[below of=1](2) {$B$}; 
\node[right of=1](3) {$B$}; 
\node[right of=2](4) {$B$};

\draw[->] (1) to node[left,la]{$f$} (2); 
\draw[->,pro] (1) to node[above,la]{$f^\bullet$} (3); 
\draw[d,pro] (2) to (4); 
\draw[d](3) to (4); 
 
\node[la] at ($(1)!0.5!(4)$) {$\id_f$};

\node[right of=3,xshift=1cm](1) {$A$}; 
\node[below of=1](2) {$A$}; 
\node[right of=1](3) {$A$}; 
\node[right of=2](4) {$B$};

\draw[->] (3) to node[right,la]{$f$} (4); 
\draw[->,pro] (2) to node[below,la]{$f^\bullet$} (4); 
\draw[d,pro] (1) to (3); 
\draw[d](1) to (2); 
 
\node[la] at ($(1)!0.5!(4)$) {$\id_f$};

\node[right of=3,xshift=1cm](1) {$A$}; 
\node[below of=1](2) {$A$}; 
\node[right of=1](3) {$B$}; 
\node[right of=2](4) {$B$};

\draw[d] (1) to (2); 
\draw[->,pro] (1) to node[above,la]{$f^\bullet$} (3); 
\draw[->,pro] (2) to node[below,la]{$f^\bullet$} (4); 
\draw[d](3) to (4); 
 
\node[la] at ($(1)!0.5!(4)$) {$e_f$};

\node[right of=3,xshift=1cm](1) {$A$}; 
\node[below of=1](2) {$B$}; 
\node[right of=1](3) {$A$}; 
\node[right of=2](4) {$B$};

\draw[d,pro] (1) to (3); 
\draw[->] (1) to node[left,la]{$f$} (2); 
\draw[->] (3) to node[right,la]{$f$} (4); 
\draw[d,pro](2) to (4); 
 
\node[la] at ($(1)!0.5!(4)$) {$\id_f$};
\end{tz}
which are forced to verify the pasting equalities in the definition of companions.
\end{proof}

\begin{prop}[{\cite[Proposition 4.20]{fibrantly_induced}}]\label{companionslifting}
Let $\bA$ be a double category. Then, a horizontal morphism $f$ in $\bA$ has a vertical companion if and only if there is a lift in the diagram below left, and a vertical morphism $u$ in $\bA$ has a horizontal companion if and only if there is a lift in the diagram below right. 
\begin{tz}
\node[](1) {$\bH\mathbbm 2$}; 
\node[below of=1](2) {$\Sq\mathbbm 2$}; 
\node[right of=1](1') {$\bA$}; 
\draw[->] (1) to node[above,la]{$f$} (1'); 
\draw[->] (1) to (2); 
\draw[->,dashed] (2) to node[pos=0.4,xshift=3pt,right,la]{$(f,u,\varphi,\psi)$} (1');
\node[right of=1',xshift=2cm](1) {$\bV\mathbbm 2$}; 
\node[below of=1](2) {$\Sq\mathbbm 2$}; 
\node[right of=1](1') {$\bA$}; 
\draw[->] (1) to node[above,la]{$u$} (1'); 
\draw[->] (1) to (2); 
\draw[->,dashed] (2) to node[pos=0.4,xshift=3pt,right,la]{$(f,u,\varphi,\psi)$} (1');
\end{tz}
\end{prop}

\begin{prop}[{\cite[Proposition 4.22]{fibrantly_induced}}]\label{rmk:companionsforequivs}
Let $\bA$ be a double category and let $(u,v,\eta,\varepsilon)$ be a vertical adjoint equivalence in $\bA$. If $u$ and $v$ have horizontal companions, then there is a lift in the following diagram
 \begin{tz}
\node[](1) {$\bV \Eadj$}; 
\node[below of=1](2) {$\Sq \Eadj$}; 
\node[right of=1,xshift=.5cm](1') {$\bA$}; 
\draw[->] (1) to node[above,la]{$(u,v,\eta,\epsilon)$} (1'); 
\draw[->] (1) to (2); 
\draw[->,dashed] (2) to (1');
\end{tz}
\end{prop}

\subsection{Conjoints in a double category} Companion pairs encode the notion of adjointness between a horizontal morphism and a vertical morphism in a double category. 

\begin{defn}
A \emph{conjoint pair}, or conjunction, in a double category $\bA$ is a tuple $(f,u,\varepsilon,\eta)$ consisting of a horizontal morphism $f\colon B\to A$, a vertical morphism $u\colon A\arrowdot B$, and two squares

\begin{tz}
\node[](1) {$B$}; 
\node[below of=1](2) {$B$}; 
\node[right of=1](3) {$A$}; 
\node[right of=2](4) {$B$};

\draw[d] (1) to (2);
\draw[->,pro] (1) to node[above,la]{$f$} (3); 
\draw[d,pro] (2) to (4); 
\draw[->](3) to node[right,la]{$u$} (4); 
 
\node[la] at ($(1)!0.5!(4)$) {$\varepsilon$};

\node[right of=3,xshift=2cm](1) {$A$}; 
\node[below of=1](2) {$B$}; 
\node[right of=1](3) {$A$}; 
\node[right of=2](4) {$A$};

\draw[d] (3) to (4); 
\draw[->,pro] (2) to node[below,la]{$f$} (4); 
\draw[d,pro] (1) to (3); 
\draw[->](1) to node[left,la]{$u$} (2); 
 
\node[la] at ($(1)!0.5!(4)$) {$\eta$};
\end{tz}

satisfying the following pasting equalities. 
\begin{tz}
\node[](1') {$B$}; 
\node[below of=1'](2') {$B$};
\node[right of=1'](1) {$A$}; 
\node[below of=1](2) {$B$}; 
\node[right of=1](3) {$A$}; 
\node[right of=2](4) {$A$};

\draw[->] (1) to node[left,la]{$u$} (2); 
\draw[d,pro] (1) to (3); 
\draw[->,pro] (2) to node[below,la]{$f$} (4); 
\draw[d](3) to (4); 
 
\node[la] at ($(1)!0.5!(4)$) {$\eta$};

\draw[d,pro] (2') to (2); 
\draw[->,pro] (1') to node[above,la]{$f$} (1); 
\draw[d](1') to (2'); 
 
\node[la] at ($(1')!0.5!(2)$) {$\varepsilon$};

\node[right of=3](1) {$B$}; 
\node at ($(1)!0.5!(4)$) {$=$};
\node[below of=1](2) {$B$}; 
\node[right of=1](3) {$A$}; 
\node[below of=3](4) {$A$}; 
\draw[d] (1) to (2); 
\draw[d] (3) to (4); 
\draw[->,pro] (1) to node[above,la]{$f$} (3); 
\draw[->,pro] (2) to node[below,la]{$f$} (4); 
\node[la] at ($(1)!0.5!(4)$) {$e_f$};

\node[right of=3,xshift=1cm,yshift=.75cm](1) {$A$}; 
\node[below of=1](2) {$B$}; 
\node[right of=1](3) {$A$}; 
\node[right of=2](4) {$A$};
\node[below of=2](2') {$B$}; 
\node[right of=2'](4') {$B$};

\draw[d] (2) to (2'); 
\draw[d,pro] (2') to (4'); 
\draw[->](4) to node[left,la]{$u$} (4'); 
 
\node[la] at ($(2)!0.5!(4')$) {$\varepsilon$};

\draw[d] (3) to (4); 
\draw[->,pro] (2) to node[above,la]{$f$} (4); 
\draw[d,pro] (1) to (3); 
\draw[->](1)to node[left,la]{$u$} (2); 
 
\node[la] at ($(1)!0.5!(4)$) {$\eta$};

\node[right of=3,yshift=-.75cm](1) {$A$}; 
\node[below of=1](2) {$B$}; 
\draw[->] (1) to node(a)[left,la]{$u$} (2);
\node at ($(a)!0.5!(4)$) {$=$};
\node[right of=1](3) {$A$}; 
\node[below of=3](4) {$B$}; 
\draw[d,pro] (1) to (3); 
\draw[d,pro] (2) to (4); 
\draw[->] (3) to node[right,la]{$u$} (4);
\node[la] at ($(1)!0.5!(4)$) {$\id_u$};
\end{tz}
We say that a horizontal morphism $f$ (resp.\ a vertical morphism $u$) \emph{has a vertical (resp.\ horizontal) conjoint} if there is a conjoint pair $(f,u,\varphi,\psi)$.
\end{defn}

\begin{ex}
    Let $\cA$ be a $2$-category. In the double category $\Sq(\cA)$ the conjoint pairs correspond to adjunctions in $\cA$ with $u$ the left adjoint and $f$ the right adjoint. 
\end{ex}

\begin{lem}\label{free_conjoint}The double category $\Sq\mathbb{2}^\hop$ is the free double category on a conjoint pair. 
\end{lem}

\begin{proof}
It follows from \cref{lem:Sq2_free} by observing that they are dual under horizontal duality.
\end{proof}

\begin{prop}\label{conjointslifting}
Let $\bA$ be a double category. Then, a horizontal morphism $f$ in $\bA$ has a vertical conjoint if and only if there is a lift in the diagram below left, and a vertical morphism $u$ in $\bA$ has a horizontal conjoint if and only if there is a lift in the diagram below right. 
\begin{tz}
\node[](1) {$\bH\mathbbm 2$}; 
\node[below of=1](2) {$\Sq\mathbbm 2^\hop$}; 
\node[right of=1](1') {$\bA$}; 
\draw[->] (1) to node[above,la]{$f$} (1'); 
\draw[->] (1) to (2); 
\draw[->,dashed] (2) to node[pos=0.4,xshift=3pt,right,la]{$(f,u,\varphi,\psi)$} (1');
\node[right of=1',xshift=2cm](1) {$\bV\mathbbm 2$}; 
\node[below of=1](2) {$\Sq\mathbbm 2^\hop$}; 
\node[right of=1](1') {$\bA$}; 
\draw[->] (1) to node[above,la]{$u$} (1'); 
\draw[->] (1) to (2); 
\draw[->,dashed] (2) to node[pos=0.4,xshift=3pt,right,la]{$(f,u,\varphi,\psi)$} (1');
\end{tz}
\end{prop}

\subsection{Equipments}

Equipments provide a rich framework in which to express formal category theory, as presented by Wood \cite{Wood} in the context of (weak) $2$-categories. Verity codifies this idea in the context of double categories \cite{Verity}, which gives a more natural and easier to manipulate perspective. See also Shulman's framed bicategories \cite{Shulman_framed_bicategories} and \cite{Shulman_equipments}.

\begin{defn}
A double category $\bA$ is an \emph{equipment} if every vertical morphism has both a companion and a conjoint. 
\end{defn}

\section{Model structure for equipments}\label{ch:equipments_MS}

In this section we construct a model structure on double categories whose fibrant objects are equipments, and whose generating cofibrations are (up to a change in direction) the same of the model structure on \cite{whi}.\footnote{We swap the horizontal and vertical directions with respect to that work, with an eye to future work where we will want to work with pseudo and virtual double categories, having the horizontal direction to be the profunctorial one.}
The main theorem of this section is as follows.

\begin{theorem}\label{MS_equipments} There is a model structure on $\dblcat$ whose fibrant objects are the equipments and whose cofibrations are the $\cI$-cofibrations of \cref{def:cofib}. Moreover, the weak equivalences between equipments are the double biequivalences.
\end{theorem}

The construction of this model structure, both in strategy and technique, is essentially that of \cite[Section 6]{fibrantly_induced}, with some modifications needed to include the \emph{conjoint} requirement for vertical arrows. In consequence, we will skip proofs when the translation is direct, and include some details when it is illustrative.

\subsection{Our starting data} Since we want to use \cref{thm:fibrantly_generated}, we need to provide the data of sought fibrant objects, cofibrations, a weak factorization system $(\an,\nfib)$ generated by a set $\cJ$, and a notion of weak equivalences between the proposed fibrant objects.

\subsubsection{Cofibrations}
We begin by presenting a set of generating cofibrations. 

\begin{defn}\label{def:cofib}
Let $\cI$ be the set containing the following double functors: 
\begin{rome}
\item the unique morphism $\emptyset\to \mathbbm{1}$, 
\item the inclusion $\mathbbm{1}\sqcup\mathbbm{1}\to \bH\mathbbm{2}$, 
\item the inclusion $\mathbbm{1}\sqcup\mathbbm{1}\to \bV\mathbbm{2}$, 
\item the inclusion $\partial(\bH\mathbbm{2}\times\bV\mathbbm{2})\to \bH\mathbbm{2}\times \bV\mathbbm{2}$ of the boundary of the square,
\item the double functor $\bH\mathbbm{2}\times\bV\mathbbm{2}\sqcup_{\partial(\bH\mathbbm{2}\times\bV\mathbbm{2})}\bH\mathbbm{2}\times\bV\mathbbm{2}\to \bH\mathbbm{2}\times\bV\mathbbm{2}$ sending the two parallel squares to the same square.
\end{rome}

The \emph{cofibrations} are the double functors in $\cof(\I)=^\boxslash(\cI^\boxslash)$. 
\end{defn}

\begin{rem} The reader familiar with the canonical model structures on the categories $\cat$ and $\twocat$ will find these maps to be a natural choice. Further insight on why these are the ones we wish to consider for our purpose is given in \cref{under_triangle_D_Reedy_cofibrant} and \cref{lem:surviving_latching_maps}.
\end{rem}

These are the same set of cofibrations that the model structure in \cite{whi}, and so also the same description of the trivial fibrations is valid. 

\begin{prop}\cite[Proposition 3.9]{whi} \label{prop:trivfib}
A double functor $F\colon \bA\to \bB$ is a trivial fibration if and only if it is surjective on objects, full on horizontal and vertical morphisms, and fully faithful on squares. 
\end{prop}

\subsubsection{Naive fibrant objects}

We now introduce a set of anodyne extensions with the purpose of encoding the data of equipments as naive fibrant objects.

\begin{defn}\label{def:anodyne}
Let $\cJ$ be the set containing the following double functors:
\begin{rome}
\item the inclusion $\mathbbm{1}\to \bV\Eadj$ of one of the two objects, 
\item the inclusion $\bV\mathbbm{2}\to \Sq\mathbbm{2}$,
\item the inclusion $\bV\mathbbm{2}\to \Sq\mathbbm{2}^\hop$ 
\item the inclusion $\bH\mathbbm{2}\to \bH\Sigma\bI$, 
\item the inclusion $\bV\mathbbm{2}\to \bV\Sigma\bI$, 
\end{rome}
The \emph{anodyne extensions} are the double functors in $\cof(\cJ)$.
\end{defn} 

Note that by \cref{companionslifting} and \cref{conjointslifting}, conditions (ii) and (iii) guarantees that the naive fibrant objects will be such that every vertical morphism have both a companion and a conjoint.

We have the following description of the naive fibrations.

\begin{prop}\label{descr:naive_fibrations} A double functor $F\colon\bA\to\bB$ is a naive fibration if and only if it satisfies the following conditions:
\begin{enumerate}
    \item[(f1)] for every object $A\in\bA$ and vertical equivalence $v\colon B\to FA$ in $\bB$, there is a vertical equivalence $u\colon A'\to A$ in $\bA$ such that $Fu=v$,
    \item[(f2)] for every vertical morphism $u\colon A\to C$ in $\bA$ and companion pair $(g,Fu,\chi,\omega)$ in $\bB$, there is a companion pair $(f,u,\varphi,\psi)$ in $\bA$ such that $Ff=g$, $F\varphi=\chi$, and $F\omega=\psi$, 
    \item[(f3)]  for every vertical morphism $u\colon A\to C$ in $\bA$ and conjoint pair $(g,Fu,\varepsilon,\eta)$ in $\bB$, there is a conjoint pair $(f,u,\varphi,\psi)$ in $\bA$ such that $Ff=g$, $F\varepsilon=\chi$, and $F\eta=\psi$, 
    \item[(f4)] for every horizontal morphism $f\colon A\to C$ in $\bA$ and $2$-isomorphism $\beta\colon g\cong~Ff$ in~$\bfH\bB$, there is a $2$-isomorphism $\alpha\colon f'\cong f$ in $\bfH\bA$ such that $(\bfH F)\alpha=\beta$,
    \item[(f5)] for every vertical morphism $u\colon A\to C$ in $\bA$ and $2$-isomorphism $\beta\colon v\cong Fu$ in~$\bfV\bB$, there is a $2$-isomorphism $\alpha\colon u'\cong u$ in $\bfV\bA$ such that $(\bfV F)\alpha=\beta$.
\end{enumerate}
\end{prop}

\begin{rem}
Comparing the generating set $\cJ$ in \cref{def:anodyne} with the generating set in \cite[Definition 6.2]{fibrantly_induced}, we see that they are very similar, but change (ii) so that instead of having companions on both directions, we have both companions and conjoints for vertical morphisms. Consequently, our \cref{descr:naive_fibrations} differs from \cite[Proposition 6.3]{fibrantly_induced} \textemdash besides the consistent change of horizontal-vertical direction\textemdash in the content of (f3).
\end{rem}

\begin{prop}\label{equip_naive_fibrant_objects}
    A double category is naive fibrant if and only if it is an equipment. 
\end{prop}

\begin{proof}If the double functor $\bA\to \mathbb{1}$ is a naive fibration, it follows from \cref{companionslifting} and \cref{conjointslifting} that $\bA$ must be an equipment. It is easy to check that if $\bA$ is an equipment then the double functor $\bA\to \mathbb{1}$ has the right lifting property with respect to (i)-(v). 
\end{proof}

\begin{prop}\label{nfib_between_equipments}
    Let $\bA$ be an equipment. A double functor $F\colon\bA\to\bB$ is a naive fibration if and only if it verifies (f1), (f4), (f5).
\end{prop}

\begin{proof}

The proof that (f4) implies (f2) is analogous to the proof of \cite[Proposition 6.7]{fibrantly_induced}. 
Now, note that under horizontal duality, we have that axioms (f1), (f4), and (f5) are self-dual, and that (f2) and (f3) are mutually dual. Therefore the statement that (f4) implies (f2) is dual, under horizontal duality, to the statement that (f4) implies (f3).
\end{proof}

\begin{prop}\label{prop:triv_fib__btwn_equipments_is_nfib}Let $\bA$ and $\bB$ be equipments. Then a trivial fibration $F\colon\bA\to\bB$ is a naive fibration. 
\end{prop}

\begin{proof}
To prove that $F\colon\bA\to\bB$ is a naive fibration we will use the description given in \cref{nfib_between_equipments}. To show (f1), we consider $A$ in $\bA$ and a vertical equivalence $v\colon B\to FA$ in $\bB$. Since $F$ is surjective on objects there exists $A'$ in $\bA$ such that $B=FA'$, therefore since $F$ is full on vertical morphisms and we now have $v\colon FA'\to FA$ in $\bB$, there exists $u\colon B'\to A$ in $\bA$ such that $Fu=v$. Using that $F$ is fully faithful on squares, it is easy to show that $u$ is a vertical equivalence. 

Conditions (f4) and (f5) are analogous; we show (f4). Given a horizontal morphism $u\colon A\arrowdot C$ in $\bA$ and a $2$-isomorphism $\beta\colon g\cong Ff$ in $\bfH\bB$, we first use that $F$ is full on horizontal morphisms to obtain a horizontal morphism $g'$ in $\bA$ such that $Fg'=g$, and then use the fully faithfulness on squares of $F$ with $\sq{\beta}{Fg'}{Ff}{\id}{\id}$. It is easy to show that the resulting square corresponds to a $2$-isomorphism in $\bfH\bA$.

\end{proof}

\subsubsection{Weak equivalences}
Finally, we describe the weak equivalences between fibrant objects. In particular, these are given by the class of weak equivalences of the model structure of \cite{MSV} (up to the change in directions).\footnote{In forthcoming work with Moser and Sarazola, the repeated use of these weak equivalences is put into context and explained.}

We begin by recalling some definitions.

\begin{defn}\label{def:weakly_invertible_square}
Let $\bA$ be a double category. A square $\sq{\alpha}{a}{a'}{u}{u'}$ in $\bA$ is weakly vertically invertible if it is an equivalence in the $2$-category $\bA_{h,s}$. Concretely, if there is a square $\sq{\gamma}{a'}{a}{v}{v'}$ in $\bA$ together with four horizontally invertible squares $\eta,\eta',\varepsilon,\varepsilon'$ as in the following pasting equalities. 

\begin{tz}
    \node[](1) {$A'$};
    \node[right of=1](2) {$A'$};
    \node[right of=2](3) {$B'$};
    \node[below of=2](4) {$A$};
    \node[right of=4](5) {$B$};
    \node[below of=4](7) {$A'$};
    \node[left of=7](6) {$A'$};
    \node[right of=7](8) {$B'$};

    \draw[d,pro] (1) to (2);
    \draw[->,pro] (2) to node[above,la]{$a'$} (3);
    \draw[->,pro] (4) to node[above,la]{$a$} (5);
    \draw[d,pro] (6) to (7);
    \draw[->,pro] (7) to node[below,la]{$a'$} (8);

    \draw[d] (1) to (6);
    \draw[->] (2) to node[left,la]{$v$} (4);
    \draw[->] (3) to node[right,la]{$v'$} (5);
    \draw[->] (4) to node[left,la]{$u$} (7);
    \draw[->] (5) to node[right,la]{$u'$} (8);

    \node[la] at ($(1)!0.5!(7)-(0,5pt)$) {{$\cong$}};
    \node[la] at ($(1)!0.5!(7)+(0,5pt)$) {$\eta$};
    \node[la] at ($(2)!0.5!(5)$) {{$\gamma$}};
    \node[la] at ($(4)!0.5!(8)$) {{$\alpha$}};

    \node[right of=3](1') {$A'$};
    \node[right of=1'](2') {$B'$};
    \node[right of=2'](3') {$B'$};
    \node[below of=3'](5') {$B$};
    \node[below of=5'](8') {$B'$};
    \node[left of=8'](7') {$B'$};
    \node[left of=7'](6') {$A'$};

    \draw[->,pro] (1') to node[above,la]{$a'$} (2');
    \draw[d,pro] (2') to (3');
    \draw[->,pro] (6') to node[below,la]{$a'$} (7');
    \draw[d,pro] (7') to (8');
    \draw[d] (1') to (6');
    \draw[d] (2') to (7');
    \draw[->] (3') to node[right,la]{$v'$} (5');
    \draw[->] (5') to node[right,la]{$u'$} (8');

    \node[la] at ($(1')!0.5!(7')$) {{$e_{a'}$}};
    \node[la] at ($(2')!0.5!(8')-(0,5pt)$) {{$\cong$}};
    \node[la] at ($(2')!0.5!(8')+(0,5pt)$) {$\eta'$};

    \node[la] at  ($(3)!0.5!(6')$) {$=$};

    \node[below of=6](1) {$A$};
    \node[right of=1](2) {$A$};
    \node[right of=2](3) {$B$};
    \node[below of=1](4) {$A'$};
    \node[below of=4](7) {$A$};
    \node[right of=7](8) {$A$};
    \node[right of=8](9) {$B$};

    \draw[d,pro] (1) to (2);
    \draw[->,pro] (2) to node[above,la]{$a$} (3);
    \draw[d,pro] (7) to (8);
    \draw[->,pro] (8) to node[above,la]{$a$} (9);
    \draw[->] (1) to node[left,la]{$u$} (4);
    \draw[->] (4) to node[left,la]{$v$} (7);
    \draw[d] (2) to (8);
    \draw[d] (3) to (9);

    \node[la] at ($(2)!0.5!(9)$) {{$e_{a}$}};
    \node[la] at ($(1)!0.5!(8)-(0,5pt)$) {{$\cong$}};
    \node[la] at ($(1)!0.5!(8)+(0,5pt)$) {$\varepsilon$};

    \node[right of=3](1') {$A$};
    \node[right of=1'](2') {$B$};
    \node[right of=2'](3') {$B$};
    \node[below of=1'](4') {$A'$};
    \node[below of=2'](5') {$B'$};
    \node[below of=4'](7') {$A$};
    \node[right of=7'](8') {$B$};
    \node[right of=8'](9') {$B$};

    \draw[->,pro] (1') to node[above,la]{$a$} (2');
    \draw[d,pro] (2') to (3');
    \draw[->,pro] (4') to node[above,la]{$a'$} (5');
    \draw[->,pro] (7') to node[below,la]{$a$} (8');
    \draw[d,pro] (8') to (9');
    \draw[->] (1') to node[left,la]{$u$} (4');
    \draw[->] (4') to node[left,la]{$v$} (7');
    \draw[->] (2') to node[right,la]{$u'$} (5');
    \draw[->] (5') to node[right,la]{$v'$} (8');
    \draw[d] (3') to (9');

    \node[la] at ($(1')!0.5!(5')$) {$\alpha$};
    \node[la] at ($(4')!0.5!(8')$) {$\gamma$};
    \node[la] at ($(2')!0.5!(9')-(0,5pt)$) {{$\cong$}};
    \node[la] at ($(2')!0.5!(9')+(0,5pt)$) {$\varepsilon'$};

    \node[la] at ($(3)!0.5!(7')$) {$=$};
\end{tz}

We call $\gamma$ a weak inverse of $\alpha$.
\end{defn}

By \cite[Lemma A.1.1]{lyne}, we know that that this weak inverse is unique.

\begin{defn}\label{defn:doublebieq}
A double functor $F\colon \bA\to \bB$ is a \emph{double biequivalence} if it satisfies the following conditions: 
\begin{enumerate}
    \item[(w1)] for every object $B\in \bB$, there is a vertical equivalence $v\colon B\xrightarrow{\simeq} FA$ in $\bB$,
    \item[(w2)] for every pair of objects $A,C\in \bA$ and every vertical morphism $v\colon FA\to FC$ in $\bA$, there is a vertical morphism $u\colon A\to C$ in $\bA$ and a $2$-isomorphism $\beta\colon v\cong Fu$ in $\bfV\bB$,
    \item[(w3)] for every horizontal morphism $f\colon B\arrowdot B'$ in $\bB$, there is a horizontal morphism $g\colon A\arrowdot A'$ in $\bA$, and a weakly vertically invertible square $\beta$ in $\bB$ as below,
\begin{tz}
\node[](1) {$B$}; 
\node[below of=1](3) {$FA$}; 
\node[right of=1](2) {$B'$}; 
\node[right of=3](4) {$FA'$};

\draw[->] (1) to node[left,la]{$\simeq$} (3);
\draw[->,pro] (1) to node[above,la]{$f$} (2); 
\draw[->,pro] (3) to node[below,la]{$Fg$} (4); 
\draw[->](2) to node[right,la]{$\simeq$} (4); 
 
\node[la] at ($(1)!0.5!(4)+(5pt,0)$) {\rotatebox{90}{$\simeq$}};
\node[la] at ($(1)!0.5!(4)-(5pt,0)$) {$\beta$};
\end{tz}
\item[(w4)] it is fully faithful on squares.
\end{enumerate}
\end{defn}

\begin{rem}
It is worth noting that this definition is not symmetric with respect to the vertical and horizontal directions.
\end{rem}

\begin{rem} \label{rem:dblbieq}
The class of double biequivalences satisfies $2$-out-of-$6$ and is accessible when considered as a full subcategory of $\dblcat^{\mathbbm{2}}$, as it is the class of weak equivalences of the combinatorial model structure on $\dblcat$ from \cite{MSV}\textemdash up to the change in directions. The reason we changed directions is because we plan to generalize these results to virtual equipments, and we want the horizontal direction to be the profunctorial direction for diagramatic reasons.
\end{rem}

We set these equivalences to be the ones we want as weak equivalences between fibrant objects.

\begin{defn}
    A double functor $F\colon\bA\to\bB$ between equipments is a \emph{weak equivalence} if it is a double biequivalence.
\end{defn}

\begin{lem}\label{weakequiv_w3}
Let $\bA$ and $\bB$ be equipments. Then for a double functor $F\colon\bA\to \bB$ we obtain an equivalent definition to \cref{defn:doublebieq} by substituting (w3) by the following condition:
\begin{enumerate}
    \item[(w3')] for every pair of objects $A,A'$ in $\bA$ and every horizontal morphism $f\colon FA\arrowdot FA'$ in $\bB$, there is a $2$-isomorphism $\beta\colon f\cong Fg$ in $\bfH\bB$.  
\end{enumerate}
\end{lem}

\begin{proof} The dual proof to \cite[Proposition 6.9]{fibrantly_induced} works, using the dual version of Moser's lemma \cite[Lemma A.2.1]{Moser}.
\end{proof}

\begin{rem}\label{prop:trivfib_btwn_equipments_are_we}It follows from the description \cref{prop:trivfib} and the definition of weak equivalence between fibrant objects given by \cref{weakequiv_w3}, that if $\bA$ and $\bB$ are equipments, then a trivial fibration $F\colon\bA\to\bB$ is weak equivalence. 
\end{rem}

\begin{prop}\label{prop:nfib_and_biequiv_is_trivfib}
    Let $\bA$ and $\bB$ be equipments. Let  $F\colon\bA\to\bB$ be a double functor that is both a double biequivalence and a naive fibration, then it is a trivial fibration. 
\end{prop}

\begin{proof}
    Given an object $B$ in $\bB$, we know by (w1) that there is a vertical equivalence $v\colon FA\xrightarrow{\simeq} B$ in $\bB$, and thus by (f1) a vertical equivalence $u\colon A\xrightarrow{\simeq} A'$ in $\bA$ such that $Fu=v$. In particular we have $FA'=B$, which proves that $F$ is surjective on objects. That $F$ is fully faithful on squares is (w4). 

    It remains to show that $F$ is full on both horizontal and vertical morphisms. For vertical morphisms, let $v\colon FA\to FB$ be a vertical morphism in $\bB$. By (w2), there is a $2$-isomorphism $\beta v\cong Fu$ in $\bfV\bB$, and thus by (f5) we know there is a $2$-isomorphism $\alpha\colon u'\cong u$ in $\bfV\bA$ such that $(\bfV F)\alpha=\beta$. In particular, $Fu'=v$ as we wanted. The fullness on horizontal morphisms requires a pivotal step. Given a horizontal morphism $f\colon FA\arrowdot FB$, by \cref{weakequiv_w3}, there is a $2$-isomorphism $\beta\colon f\cong Fg$ in $\bfH\bB$, for some horizontal morphism $g$ in $\bA$. Then, analogously to the vertical case, (f4) gives us what we wanted.
\end{proof}

\begin{cor} Let $\bA$ and $\bB$ be equipments. Then a double functor $F\colon \bA\to \bB$ is a trivial fibration if and only if it is both a double biequivalence and a naive fibration.
\end{cor}

\begin{proof}
    It follows directly from \cref{prop:nfib_and_biequiv_is_trivfib}, \cref{prop:trivfib_btwn_equipments_are_we}, and \cref{prop:triv_fib__btwn_equipments_is_nfib}.
\end{proof}

\subsection{Existence of the model structure}
It will be useful to give an explicit fibrant replacement. 

\begin{constr} 
Let $\bA$ be a double category. We define $\bA^\fibrant$ to be the double category given by the following pushout in $\dblcat$.
\begin{tz}
\node[](1) {$\bigsqcup_{\mathrm{Ver}(\bA)} \bV\mathbbm{2}$};
\node[below of=1](2) {$\bigsqcup_{\mathrm{Ver}(\bA)} (\Sq\mathbbm{2}\sqcup_{\bV\mathbbm{2}}\Sq\mathbbm{2}^{\mathrm{hop}})$}; 
\node[right of=1,xshift=2.5cm](3) {$\bA$}; 
\node[below of=3](4) {$\bA^\fibrant$}; 
\node at ($(4)-(.4cm,-.4cm)$) {$\ulcorner$};

\draw[->] (1) to (2); 
\draw[->] (1) to (3); 
\draw[->] (2) to (4);
\draw[->] (3) to node[right,la]{$j_\bA$} (4);
\end{tz}
This construction extends to a functor $(-)^\fibrant\colon \dblcat\to \dblcat$. 
\end{constr}

\begin{rem}\label{fib_in2steps}The construction above could have been expressed in two steps: first adding companions for every vertical map and then conjoints for every vertical map, or viceversa. This follows by observing that pushing out along $\bV\mathbbm{2}\to\Sq\mathbbm{2}\to\Sq\mathbbm{2}\sqcup_{\bV\mathbbm{2}}\Sq\mathbbm{2}^{\mathrm{hop}}$ and 
$\bV\mathbbm{2}\to\Sq\mathbbm{2}^{\mathrm{hop}}\to\Sq\mathbbm{2}\sqcup_{\bV\mathbbm{2}}\Sq\mathbbm{2}^{\mathrm{hop}}$ commutes. Informally, this is reasonable, as the pushouts are not adding new vertical morphisms.
\end{rem}

\begin{rem}\label{jA_ff} The double functor $j_\bA\colon \bA\to \bA^\fibrant$ is fully faithful on squares. This follows directly from \cref{fib_in2steps} and \cite[Remark 6.12]{fibrantly_induced}. 
\end{rem}

\begin{prop}
Let $\bA$ be a double category. The double functor $j_\bA\colon \bA\to \bA^\fibrant$ is a naive fibrant replacement of $\bA$. 
\end{prop}

\begin{proof}
The double functor $j_\bA$ is an anodyne extension as a pushout of coproducts of anodyne extensions. To see that $\bA^\fibrant$ is an equipment, we observe that given a vertical morphisms in $\bA^\fibrant$, it must be in the image of $\bA$, as there are no added vertical morphisms, and therefore a companion and a conjoint for such a morphism was added through the pushout.
\end{proof}

\begin{prop}  \label{prop:trivfibareweSq}
If $F\colon \bA\to \bB$ is a trivial fibration, then $F^\fibrant\colon \bA^\fibrant\to \bB^\fibrant$ is a trivial fibration between equipments. 
\end{prop}

\begin{proof} As we have mentioned before, our trivial fibrations (and in consequence cofibrations) are exactly those of the model structure in \cite{whi}, or the canonical trivial fibrations in \cite{MSVdouble_equivalences}. Therefore, \cite[Proposition 4.10]{MSVdouble_equivalences} implies that our trivial fibrations are stable under pushouts along our cofibrations. 

In view of the above together with \cref{prop:trivfib_btwn_equipments_are_we}, it is enough to show that when we consider the pushout 
\begin{tz}
\node[] (1) {$\bA$};
\node[right of=1](2) {$\bB$};
\node[below of=1](3) {$\bA^\fibrant$};
\node[right of=3](4) {$P$};

\draw[->] (1) to node[left,la]{$j_\bA$} (3);
\draw[->] (1) to node[above,la]{$F$} (2);
\draw[->] (2) to node[right,la]{$j$} (4);
\draw[->] (3) to (4);
\end{tz}
we know that $j\colon\bB\to P$ is a fibrant replacement of $\bB$. Firstly, $j$ is an anodyne extension as they are closed under pushouts, since they are the left class of a weak factorization system. That $P$ is an equipment follows from the fact that we know the map $\bA^\fibrant\to P$ is a trivial fibrations and \cite[Lemma 6.3]{MSVdouble_equivalences}.
\end{proof}

The above proposition together with \cref{prop:trivfib_btwn_equipments_are_we} guarantee the following.

\begin{cor}\label{cor:trivfib_are_we}Every trivial fibration is a weak equivalence.
\end{cor}

\subsubsection{Path object}

In order to use \cref{thm:fibrantly_generated}, we want to show the existence of a path object for every equipment. The same candidate that we used in \cite[Section 6]{fibrantly_induced} works. Thus, for any equipment $\bA$, we consider 
\[\bA\to \llbracket\Sq\Eadj,\bA\rrbracket_{\mathrm{ps}} \to \bA\times \bA. \]
induced by the double functor $\mathbbm{1}\sqcup \mathbbm{1}\to \Sq\Eadj\to \mathbbm{1}$.

The proof that this is indeed a path objects follows the same strategy as where it was introduced, adapted to the case of equipments. We start by proving the following.

\begin{prop}\label{path_object_equipment} For any equipment $\bA$, the pseudo hom double category $\llbracket\Sq\Eadj,\bA\rrbracket_{\mathrm{ps}}$ is an equipment as well.
\end{prop}

\begin{proof}
We show that every vertical morphism in $\llbracket\Sq\Eadj,\bA\rrbracket_{\mathrm{ps}}$ admits a conjoint; the proof that every vertical morphism admits a companion is as in \cite[Proposition 6.15]{fibrantly_induced}.

Let $F,G\colon\Sq\Eadj\to\bA$ be double functors and $\lambda\colon F\Rightarrow G$ a vertical pseudo natural transformation. We want to construct a conjoint pair $(\kappa,\lambda,\varepsilon,\eta)$ in $\llbracket\Sq\Eadj,\bA\rrbracket_{\mathrm{ps}}$ consisting of a horizontal pseudo natural transformation $\kappa\colon G\Arrowdot F$ and two modifications $\varepsilon, \eta$ with the appropriate boundary.

We know that $\bA$ is an equipment, and thus for each object $0$ and $1$ in $\Sq\Eadj$ we can pick conjoint pairs $(\kappa_0,\lambda_0, \varepsilon_0,\eta_0)$ and $(\kappa_1,\lambda_1,\varepsilon_1,\eta_1)$ in $\bA$. Now for each generating vertical morphisms $u\colon 0\to 1$ and $v\colon 1\to 0$ in $\Sq\Eadj$ we define $\kappa_u$ and $\kappa_v$ as the following pastings. 

\begin{tz}
\node[](1') {$G0$}; 
\node[right of=1'](2') {$F0$};
\node[below of=1'](3') {$G1$}; 
\node[below of=2'](4') {$F1$}; 
\draw[->,pro] (1') to node[above,la]{$\kappa_0$} (2');
\draw[->,pro] (3') to node[below,la]{$\kappa_1$} (4');
\draw[->] (1') to node[left,la]{$Gu$} (3');
\draw[->] (2') to node[right,la]{$Fu$} (4');

\node[la] at ($(1')!0.5!(4')$) {$\kappa_u$};

\node[right of=2',yshift=0.75cm](1){$G0$};
\node[right of=1](2) {$F0$};
\node[below of=1](3) {$G0$};
\node[below of=2](4) {$G0$};
\draw[->,pro] (1) to node[above,la]{$\kappa_0$} (2);
\draw[d,pro] (3) to (4);
\draw[d] (1) to (3);
\draw[->] (2) to node[right,la]{$\lambda_0$} (4);

\node[la] at ($(1)!0.5!(4)$) {$\varepsilon_0$};

\node[below of=3](5) {$G1$};
\node[below of=4](6) {$G1$};
\draw[->] (3) to node[left,la]{$Gu$} (5);
\draw[->] (4) to node[right,la]{$Gu$} (6);
\draw[d,pro] (5) to (6);

\node[la] at ($(3)!0.5!(6)$) {$\id_{Gu}$};

\node[la] at ($(2')!0.5!(5)+(0,0.35cm)$) {$\coloneqq$};

\node[right of=2](1') {$F0$};
\node[right of=1'](2') {$F0$};
\node[below of=1'](3') {$F1$};
\node[below of=2'](4') {$F1$};
\draw[d,pro] (1') to (2');
\draw[d,pro] (3') to (4');
\draw[->] (1') to node[left,la]{$Fu$} (3');
\draw[->] (2') to node[right,la]{$Fu$} (4');

\node[la] at  ($(1')!0.5!(4')$) {$\id_{Fu}$};

\node[below of=3'](5'){$G1$};
\node[below of=4'](6'){$F1$};
\draw[->] (3') to node[left,la]{$\lambda_1$} (5');
\draw[d] (4') to (6');
\draw[->,pro] (5') to node[below,la]{$\kappa_1$} (6');

\node[la] at ($(3')!0.5!(6')$) {$\eta_1$};

\draw[d,pro](2) to (1');
\draw[d,pro](6) to (5');

\node[la] at ($(4)!0.5!(3')$) {$\lambda_u^{-1}$};

\end{tz}

\begin{tz}
\node[](1') {$G1$}; 
\node[right of=1'](2') {$F1$};
\node[below of=1'](3') {$G)$}; 
\node[below of=2'](4') {$F0$}; 
\draw[->,pro] (1') to node[above,la]{$\kappa_0$} (2');
\draw[->,pro] (3') to node[below,la]{$\kappa_1$} (4');
\draw[->] (1') to node[left,la]{$Gv$} (3');
\draw[->] (2') to node[right,la]{$Fv$} (4');

\node[la] at ($(1')!0.5!(4')$) {$\kappa_v$};

\node[right of=2',yshift=0.75cm](1){$G1$};
\node[right of=1](2) {$F1$};
\node[below of=1](3) {$G1$};
\node[below of=2](4) {$G1$};
\draw[->,pro] (1) to node[above,la]{$\kappa_1$} (2);
\draw[d,pro] (3) to (4);
\draw[d] (1) to (3);
\draw[->] (2) to node[right,la]{$\lambda_1$} (4);

\node[la] at ($(1)!0.5!(4)$) {$\varepsilon_1$};

\node[below of=3](5) {$G0$};
\node[below of=4](6) {$G0$};
\draw[->] (3) to node[left,la]{$Gv$} (5);
\draw[->] (4) to node[right,la]{$Gv$} (6);
\draw[d,pro] (5) to (6);

\node[la] at ($(3)!0.5!(6)$) {$\id_{Gv}$};

\node[la] at ($(2')!0.5!(5)+(0,0.35cm)$) {$\coloneqq$};

\node[right of=2](1') {$F1$};
\node[right of=1'](2') {$F1$};
\node[below of=1'](3') {$F0$};
\node[below of=2'](4') {$F0$};
\draw[d,pro] (1') to (2');
\draw[d,pro] (3') to (4');
\draw[->] (1') to node[left,la]{$Fv$} (3');
\draw[->] (2') to node[right,la]{$Fv$} (4');

\node[la] at  ($(1')!0.5!(4')$) {$\id_{Fv}$};

\node[below of=3'](5'){$G0$};
\node[below of=4'](6'){$F0$};
\draw[->] (3') to node[left,la]{$\lambda_0$} (5');
\draw[d] (4') to (6');
\draw[->,pro] (5') to node[below,la]{$\kappa_0$} (6');

\node[la] at ($(3')!0.5!(6')$) {$\eta_0$};

\draw[d,pro](2) to (1');
\draw[d,pro](6) to (5');

\node[la] at ($(4)!0.5!(3')$) {$\lambda_v^{-1}$};

\end{tz}

Horizontally, for each generating horizontal morphism $f\colon 0\arrowdot 1$ and $g\colon 1 \arrowdot 0$ in $\Sq\Eadj$, we consider $\kappa_f$ and $\kappa_g$ to be the following pastings.

\begin{tz}
\node[](1) {$G0$}; 
\node[right of=1](2) {$F0$};
\node[right of=2](3) {$F1$};
\node[below of=1](4) {$G0$};
\node[right of=4](5) {$G1$};
\node[right of=5](6) {$F1$};

\draw[d] (1) to (4);
\draw[d] (3) to (6);
\draw[->,pro] (1) to node[above,la]{$\kappa_0$} (2);
\draw[->,pro] (2) to node[above,la]{$Ff$} (3);
\draw[->,pro] (4) to node[below,la]{$Gf$} (5);
\draw[->,pro] (5) to node[below,la]{$\kappa_1$} (6);

\node[la] at ($(1)!0.5!(6)$) {$\kappa_f$};

\node[right of=3](1) {$G0$};
\node[below of=1](5) {$G0$};

\node[la] at ($(3)!0.5!(5)$) {$\coloneqq$};

\node[right of=1](2) {$F0$};
\node[right of=2](3) {$F1$};
\node[right of=3](4) {$F1$};
\node[below of=1](5) {$G0$};
\node[right of=5](6) {$G0$};
\node[right of=6](7) {$G1$};
\node[right of=7](8) {$F1$};

\draw[->,pro] (1) to node[above,la]{$\kappa_0$} (2);
\draw[->,pro] (2) to node[above,la]{$Ff$} (3);
\draw[d,pro] (3) to (4);

\draw[d] (1) to (5);
\draw[->] (2) to node[right,la]{$\lambda_0$} (6);
\draw[->] (3) to node[right,la]{$\lambda_1$} (7);
\draw[d] (4) to (8);

\draw[d,pro] (5) to (6);
\draw[->,pro] (6) to node[below,la]{$Gf$} (7);
\draw[->,pro] (7) to node[below,la]{$\kappa_1$} (8);

\node[la] at ($(1)!0.5!(6)$) {$\varepsilon_0$};
\node[la] at ($(2)!0.5!(7)$) {$\lambda_f$};
\node[la] at ($(3)!0.5!(8)$) {$\eta_1$};

\end{tz}

\begin{tz}
\node[](1) {$G1$}; 
\node[right of=1](2) {$F1$};
\node[right of=2](3) {$F0$};
\node[below of=1](4) {$G1$};
\node[right of=4](5) {$G0$};
\node[right of=5](6) {$F0$};

\draw[d] (1) to (4);
\draw[d] (3) to (6);
\draw[->,pro] (1) to node[above,la]{$\kappa_1$} (2);
\draw[->,pro] (2) to node[above,la]{$Fg$} (3);
\draw[->,pro] (4) to node[below,la]{$Gg$} (5);
\draw[->,pro] (5) to node[below,la]{$\kappa_0$} (6);

\node[la] at ($(1)!0.5!(6)$) {$\kappa_g$};

\node[right of=3](1) {$G1$};
\node[below of=1](5) {$G1$};

\node[la] at ($(3)!0.5!(5)$) {$\coloneqq$};

\node[right of=1](2) {$F1$};
\node[right of=2](3) {$F0$};
\node[right of=3](4) {$F0$};
\node[below of=1](5) {$G1$};
\node[right of=5](6) {$G1$};
\node[right of=6](7) {$G0$};
\node[right of=7](8) {$F0$};

\draw[->,pro] (1) to node[above,la]{$\kappa_1$} (2);
\draw[->,pro] (2) to node[above,la]{$Fg$} (3);
\draw[d,pro] (3) to (4);

\draw[d] (1) to (5);
\draw[->] (2) to node[right,la]{$\lambda_1$} (6);
\draw[->] (3) to node[right,la]{$\lambda_0$} (7);
\draw[d] (4) to (8);

\draw[d,pro] (5) to (6);
\draw[->,pro] (6) to node[below,la]{$Gg$} (7);
\draw[->,pro] (7) to node[below,la]{$\kappa_0$} (8);

\node[la] at ($(1)!0.5!(6)$) {$\varepsilon_1$};
\node[la] at ($(2)!0.5!(7)$) {$\lambda_g$};
\node[la] at ($(3)!0.5!(8)$) {$\eta_0$};

\end{tz}

The vertical inverses of $\kappa_f$ and $\kappa_g$ are as below. 

\begin{tz}
\node[](1') {$G0$}; 
\node[right of=1'](2') {$G1$};
\node[right of=2'](3') {$F1$};
\node[below of=1'](4') {$G0$};
\node[right of=4'](5') {$F0$};
\node[right of=5'](6') {$F1$};

\draw[d] (1') to (4');
\draw[d] (3') to (6');
\draw[->,pro] (1') to node[above,la]{$Gf$} (2');
\draw[->,pro] (2') to node[above,la]{$\kappa_1$} (3');
\draw[->,pro] (4') to node[below,la]{$\kappa_0$} (5');
\draw[->,pro] (5') to node[below,la]{$Ff$} (6');

\node[la] at ($(1')!0.5!(6')$) {$(\kappa_f)^{-1}$};

\node[right of=3',yshift=1.5cm](1) {$G0$};
\node[right of=1](2) {$G1$};
\node[right of=2](3) {$F1$};

\node[below of=2](7) {$G1$};
\node[right of=7](8) {$F1$};
\node[right of=8](9) {$F0$};
\node[right of=9](10) {$F1$};

\node[above of=10](5) {$F1$};

\node[below of=7](12) {$G1$};
\node[left of=12](11) {$G0$};
\node[right of=12](13) {$G0$};
\node[right of=13](14) {$F0$};

\node[below of=11](16) {$G0$};
\node[below of=13](18) {$G0$};
\node[right of=18](19) {$F0$};
\node[right of=19](20) {$F1$};

\draw[->,pro] (1) to node[above,la]{$Gf$} (2);
\draw[->,pro] (2) to node[above,la]{$\kappa_1$} (3);
\draw[d,pro] (3) to (5);

\draw[->,pro] (7) to node[above,la]{$\kappa_1$} (8);
\draw[->,pro] (8) to node[above,la]{$Fg$} (9);
\draw[->,pro] (9) to node[above,la]{$Ff$} (10);

\draw[->,pro] (11) to node[above,la]{$Gf$} (12);
\draw[->,pro] (12) to node[above,la]{$Gg$} (13);
\draw[->,pro] (13) to node[above,la]{$\kappa_0$} (14);

\draw[d,pro] (16) to (18);
\draw[->,pro] (18) to node[below,la]{$\kappa_0$} (19);
\draw[->,pro] (19) to node[below,la]{$Ff$} (20);

\draw[d] (1) to (11);
\draw[d] (11) to (16);

\draw[d] (2) to (7);
\draw[d] (7) to (12);

\draw[d] (3) to (8);
\draw[d] (13) to (18);

\draw[d] (9) to (14);
\draw[d] (14) to (19);

\draw[d] (5) to (10);
\draw[d] (10) to (20);

\node[la] at ($(1)!0.5!(12)$) {$e_{Gf}$};
\node[la] at ($(2)!0.5!(8)$) {$e_{\kappa_1}$};
\node[la] at ($(7)!0.5!(14)$) {$\kappa_g$};
\node[la] at ($(9)!0.5!(20)$) {$e_{Ff}$};
\node[la] at ($(13)!0.5!(19)$) {$e_{\kappa_0}$};

\node[la] at ($(3)!0.5!(10)-(7pt,0)$) {$F\theta^{-1}$};
\node[la] at ($(3)!0.5!(10)+(7pt,0)$) {\rotatebox{90}{$\cong$}};

\node[la] at ($(11)!0.5!(18)-(7pt,0)$) {$G\xi^{-1}$};
\node[la] at ($(11)!0.5!(18)+(7pt,0)$) {\rotatebox{90}{$\cong$}};

\node[la] at ($(3')!0.5!(11)$) {$\coloneqq$};
\end{tz}

\begin{tz}
\node[](1') {$G1$}; 
\node[right of=1'](2') {$G0$};
\node[right of=2'](3') {$F0$};
\node[below of=1'](4') {$G1$};
\node[right of=4'](5') {$F1$};
\node[right of=5'](6') {$F0$};

\draw[d] (1') to (4');
\draw[d] (3') to (6');
\draw[->,pro] (1') to node[above,la]{$Gg$} (2');
\draw[->,pro] (2') to node[above,la]{$\kappa_0$} (3');
\draw[->,pro] (4') to node[below,la]{$\kappa_1$} (5');
\draw[->,pro] (5') to node[below,la]{$Fg$} (6');

\node[la] at ($(1')!0.5!(6')$) {$(\kappa_g)^{-1}$};

\node[right of=3',yshift=1.5cm](1) {$G1$};
\node[right of=1](2) {$G0$};
\node[right of=2](3) {$F0$};

\node[below of=2](7) {$G0$};
\node[right of=7](8) {$F0$};
\node[right of=8](9) {$F1$};
\node[right of=9](10) {$F0$};

\node[above of=10](5) {$F0$};

\node[below of=7](12) {$G0$};
\node[left of=12](11) {$G1$};
\node[right of=12](13) {$G1$};
\node[right of=13](14) {$F1$};

\node[below of=11](16) {$G1$};
\node[below of=13](18) {$G1$};
\node[right of=18](19) {$F1$};
\node[right of=19](20) {$F0$};

\draw[->,pro] (1) to node[above,la]{$Gg$} (2);
\draw[->,pro] (2) to node[above,la]{$\kappa_0$} (3);
\draw[d,pro] (3) to (5);

\draw[->,pro] (7) to node[above,la]{$\kappa_0$} (8);
\draw[->,pro] (8) to node[above,la]{$Ff$} (9);
\draw[->,pro] (9) to node[above,la]{$Fg$} (10);

\draw[->,pro] (11) to node[above,la]{$Gg$} (12);
\draw[->,pro] (12) to node[above,la]{$Gf$} (13);
\draw[->,pro] (13) to node[above,la]{$\kappa_1$} (14);

\draw[d,pro] (16) to (18);
\draw[->,pro] (18) to node[below,la]{$\kappa_1$} (19);
\draw[->,pro] (19) to node[below,la]{$Fg$} (20);

\draw[d] (1) to (11);
\draw[d] (11) to (16);

\draw[d] (2) to (7);
\draw[d] (7) to (12);

\draw[d] (3) to (8);
\draw[d] (13) to (18);

\draw[d] (9) to (14);
\draw[d] (14) to (19);

\draw[d] (5) to (10);
\draw[d] (10) to (20);

\node[la] at ($(1)!0.5!(12)$) {$e_{Gg}$};
\node[la] at ($(2)!0.5!(8)$) {$e_{\kappa_0}$};
\node[la] at ($(7)!0.5!(14)$) {$\kappa_f$};
\node[la] at ($(9)!0.5!(20)$) {$e_{Fg}$};
\node[la] at ($(13)!0.5!(19)$) {$e_{\kappa_1}$};

\node[la] at ($(3)!0.5!(10)-(7pt,0)$) {$F\xi$};
\node[la] at ($(3)!0.5!(10)+(7pt,0)$) {\rotatebox{90}{$\cong$}};

\node[la] at ($(11)!0.5!(18)-(7pt,0)$) {$G\theta$};
\node[la] at ($(11)!0.5!(18)+(7pt,0)$) {\rotatebox{90}{$\cong$}};

\node[la] at ($(3')!0.5!(11)$) {$\coloneqq$};

\end{tz}

To verify that these are indeed inverses of $\kappa_f$ and $\kappa_g$ we rely on the conjoint equalities coming from  $(\kappa_0,\lambda_0, \varepsilon_0,\eta_0)$ and $(\kappa_1,\lambda_1,\varepsilon_1,\eta_1)$, that we know that  $\lambda_f$ and $\lambda_g$ are weakly inverse to each other, together with the triangle identities of $(\theta,\xi)$. The naturality of $\kappa$ is a consequence of the conjoint equalities of $(\kappa_0,\lambda_0, \varepsilon_0,\eta_0)$ and $(\kappa_1,\lambda_1,\varepsilon_1,\eta_1)$, combined with the naturality of $\lambda$. Finally, we can use the same conjoint equalities to show that $\varepsilon$ and $\eta$ are modifications.
\end{proof}

\begin{prop}\label{prop:pathobject}Let $\bA$ be an equipment. Then the factorization of the diagonal
\[ \bA\xrightarrow{w} \llbracket\Sq\Eadj,\bA\rrbracket_{\mathrm{ps}} \xrightarrow{p} \bA\times \bA \]
induced by $\mathbbm{1}\sqcup \mathbbm{1}\to \Sq\Eadj\to \mathbbm{1}$ is a path object. 
\end{prop}

\begin{proof} The proof is completely analogous to \cite[Proposition 6.16]{fibrantly_induced}, as we have adapted all results needed for this to be the case. 
\end{proof}

\subsubsection{Proof of the model structure} It only remains to put together all the results we have, and the proof follows as in \cite[Theorem 6.1]{fibrantly_induced}.

\begin{proof}[Proof of \cref{MS_equipments}]

We use \cref{thm:fibrantly_generated} taking the set $\cI$ to be as in \cref{def:cofib}, the weak factorization system $(\an,\nfib)$ to be the one generated by the set $\cJ$ of \cref{def:anodyne}, and the class $\Wf$ to be the double biequivalences between equipments (as equipments are the naive fibrant objects by \cref{equip_naive_fibrant_objects}). 

We have that condition \ref{ax:trivfib} follows from \cref{cor:trivfib_are_we}, conditions \ref{2of6Wf} and \ref{accessibility} follow from \cref{rem:dblbieq}, that the path object condition \ref{Path} follows from \cref{prop:pathobject}, and that \ref{fibwe} is the content of \cref{prop:nfib_and_biequiv_is_trivfib}.

\end{proof}

\section{Background on Reedy categories}

Reedy categories allow us to prove technical lemmas with simple categorical constructions that provide insight in certain constructions and definitions in homotopy theory. For this section we will follow \cite{Reedy_RV}, and see also \cite[Appendix C]{elements}.

\begin{defn}
    A \emph{Reedy category} is a small category $\cC$ together with a degree function $\deg\colon\mathrm{obj}(\cC)\to\mathbb{N}$, and two subcategories $\overrightarrow{\cC}$ and $\overleftarrow{\cC}$ that contain all objects of $\cC$, and satisfy the following conditions.
    \begin{enumerate}
        \item Every arrow $\alpha$ in $\cC$ has a unique factorization $\alpha=\overrightarrow{\alpha}\circ\overleftarrow{\alpha}$ with $\overrightarrow{\alpha}\in\overrightarrow{\cC}$ and $\overleftarrow{\alpha}\in\overleftarrow{\cC}$.
        \item If $\alpha\colon c'\to c$ in $\cC$ is non-identity arrow in $\overrightarrow{\cC}$ (respectively $\overleftarrow{\cC}$) then $\deg(c')<\deg(c)$ (respectively $\deg(c')>\deg(c)$).
    \end{enumerate}
We often write $(\cC,\overrightarrow{\cC},\overleftarrow{\cC})$ for a Reedy category as above. 
\end{defn}

\begin{defn}\label{leibniz_construction}
    Let $\otimes\colon\mathcal{K}\times\cL\to \mathcal{M}$ be a bifunctor, and suppose $\mathcal{M}$ admits all pushouts. Then the Leibniz construction provides a bifunctor $\widehat{\otimes}\colon\mathcal{K}^\mathbb{2}\times\cL^\mathbb{2}\to \mathcal{M}^\mathbb{2}$ between arrow categories that assigns, to each pair $f\in\mathcal{K}^\mathbb{2},g\in\mathcal{L}^\mathbb{2}$, and object $f\widehat{\otimes} g\in \mathcal{M}^\mathbb{2}$ given by the universal property of the pushout as illustrated below. 

\begin{tz}
\node[](1) {$K\otimes L$};
\node[right of=1,xshift=2.5cm](2) {$K'\otimes L$};
\node[below of=1](3) {$K\otimes L'$};
\node[below of=2](4) {$(K'\otimes L)\cup_{K\otimes L}(K\otimes L')$};
\node[below of=4,xshift=4cm,yshift=0.5cm](5) {$K'\otimes L'$};

\draw[->] (1) to node[above,la]{$f\otimes L$}   (2);
\draw[->] (1) to node[left,la]{$K\otimes g$}    (3);
\draw[->] (2) to  (4);
\draw[->] (3) to  (4);
\draw[->,dashed] (4) to node[below,la]{$f\widehat{\otimes} g$} (5);

\draw[->,bend left=20] (2) to node[above,xshift=2pt,la]{$K'\otimes g$} (5);
\draw[->,bend right=20] (3) to node[below,la]{$f\otimes L'$} (5);

\node at ($(4)+(-9pt,9pt)$) {$\ulcorner$};
\end{tz}

\end{defn}

\begin{rem}
When the bifunctor $\otimes$ corresponds to a monoidal product, the Leibniz construction is usually called the pushout product.
\end{rem}

\begin{notation}
 For any small category $\cC$, given an object $c\in\cC$ we write $\cC_c$ for the covariant representable functor determined by $c$, and $\cC^c$ for the contravariant one. 
\end{notation}

\begin{defn} Let $\cC$ be a Reedy category, and $c$ an object in $\cC$ of degree $n$. We define the boundary of the representables $\cC_c$ and $\cC^c$ by 
$$\partial\cC_c\coloneqq \mathrm{sk}_{n-1}\cC_c \text{ in }\Set^\cC$$
$$\partial\cC^c\coloneqq \mathrm{sk}_{n-1}\cC^c \text{ in }\Set^{\cC^\op}$$
where these are the subfunctors of arrows that factor through an object of degree at most $n-1$.
\end{defn}

\begin{defn}\label{weighted_colimit}
Let $\mathcal{I}$ be a small category, and $\mathcal{D}$ a cocomplete category. We define the \emph{weighted colimit bifunctor} as the bifunctor $\Set^{\cI^\op}\times\mathcal{D}^{\cI}\xrightarrow{-\ast_{\cI}-}\mathcal{D}$ characterized by the following two axioms. Let us call the input in $\Set^{\cI^\op}$ the \emph{weight}.
\begin{enumerate}
    \item If the weight is representable, that is, it is of the form $\cI^i\coloneqq \cI(-,i)$ for $i$ in $\cI$, then $$\cI^i\ast_\cI F\coloneqq F(i).$$ 
    \item It is cocontinuous in the weight; that is, if $W=\mathrm{colimit}_KW_K$, then 
    $$(\mathrm{colim}_KW_K)\ast_\cI F\cong\mathrm{colim}_K (W_K\ast_\cI F).$$
\end{enumerate}
\end{defn}

\begin{defn}\label{weighted_limit}
Let $\mathcal{J}$ be a small category, and $\mathcal{D}$ a complete category. We define the \emph{weighted limit bifunctor} as the bifunctor $(\Set^\cJ)^\op\times\mathcal{D}^{\cJ}\xrightarrow{\{-.-\}^\cJ}\mathcal{D}$ characterized by the following two axioms. Let us call the input in $\Set^\cJ$ the \emph{weight}.
\begin{enumerate}
    \item If the weight is representable, that is, it is of the form $\cJ_j\coloneqq \cJ(j,-)$ for $j$ in $\cJ$, then $$\{\cJ_j,G\}^\cJ=G(j).$$ 
    \item If the weight is $W=\mathrm{colimit}_KW_K$, then 
    $$\{\mathrm{colim}_KW_K, G\}^\cJ\cong\mathrm{lim}_K \{W_K,G\}^\cJ.$$
\end{enumerate}
\end{defn}

\begin{rem}\label{rem:weighted_limitcolimit_hom_2varadj}
For any complete and cocomplete category $\cM$, and any $\cI$ and $\cJ$ small categories,  the weighted limit and colimit bifunctors together with the $\mathrm{hom}$ bifunctor
$$(\cM^\cI)^\op\times\cM^\cJ\longrightarrow\Set^{\cI^{\op}\times\cJ}$$
form a two variable adjunction.
\end{rem}

\begin{defn}\label{def:latching_matching_objects} Let $\cC$ be a Reedy category, and $c$ an object in $\cC$. Consider also a diagram  $X\in\mathcal{D}^\cC$, where $\mathcal{D}$ is complete and cocomplete, then:
\begin{enumerate}
    \item the latching object of the diagram $X$ at $c$ is the weighted colimit 
    $$L^cX\coloneqq\partial\cC^c\ast_\cC X,$$
    \item the matching object of the diagram $X$ at $c$ is the weighted limit 
    $$M^cX\coloneqq\{\partial\cC_c,X\}^\cC.$$
\end{enumerate}
\end{defn}

\begin{defn}\label{def:matching_latching_maps}  Let $\cC$ be a Reedy category, and $c$ an object in $\cC$. Consider also a diagram  $X\in\mathcal{D}^\cC$, where $\mathcal{D}$ is a complete and cocomplete category, then:
\begin{enumerate}
    \item the latching map is $\ell^c\colon L^cX\to X(c)$ induced by the boundary inclusion $\partial\cC^c\hookrightarrow \cC^c$, and
    \item the matching map is $m^c\colon X(c)\to M^cX$ induced by the boundary inclusion $\partial\cC_c\hookrightarrow\cC_c$.
\end{enumerate}
\end{defn}

\begin{defn}\label{def:relative_matching_latching_maps}
 Let $\cC$ be a Reedy category, $c$ an object in $\cC$, and $\mathcal{D}$ a complete and cocomplete category. Consider also two diagrams $X,Y$ in $\mathcal{D}^\cC$ and a natural transformation $f\colon X\to Y$, then:
 \begin{enumerate}
     \item the relative latching map $\widehat{\ell}^c f$ is given by the universal property of pushouts as illustrated in the diagram below
\begin{tz}
\node[](1) {$L^cX$}; 
\node[right of=1](2) {$X(c)$}; 
\node[below of=1](1') {$L^cY$}; 
\node[below of=2](2') {$\bullet$}; 

\draw[->] (1) to node[above,la]{$\ell^c$} (2); 
\draw[->] (1) to node[left,la]{$L^cf$} (1'); 
\draw[->] (2) to (2'); 
\draw[->] (1') to (2');

\node[below right of=2',xshift=.5cm](5) {$Y(c)$};
\draw[->,bend left] (2) to node[right,la]{$f^c$} (5); 
\draw[->,bend right] (1') to node[below,la]{$\ell^c$} (5); 
\draw[->,dashed] (2') to node[above,la,pos=0.4,xshift=2pt]{$\widehat{\ell}^cf$} (5);
\node at ($(2')-(.3cm,-.3cm)$) {$\ulcorner$};
\end{tz}
    \item the relative matching map $\widehat{m}^cf$ is given by the universal property of pullbacks as illustrated in the diagram below. 

    \begin{tz}
\node[](1) {$\bullet$}; 
\node[right of=1](2) {$M^cX$}; 
\node[below of=1](1') {$Y(c)$}; 
\node[below of=2](2') {$M^cY$}; 

\draw[->] (1) to (2); 
\draw[->] (1) to (1'); 
\draw[->] (2) to node[right, la]{$M^cf$} (2'); 
\draw[->] (1') to node[below, la]{$m^c$} (2');

\node[above left of=1,xshift=-.5cm](5) {$X(c)$};
\draw[->,bend left] (5) to node[above,la]{$m^c$} (2); 
\draw[->,bend right] (5) to node[below,la, xshift=-2pt]{$f^c$} (1'); 
\draw[->,dashed] (5) to node[above,la,pos=0.5,xshift=2pt]{$\widehat{m}^cf$} (1);
\node at ($(1)-(-.3cm,.3cm)$) {$\lrcorner$};
\end{tz}
 \end{enumerate}
 
    relative latching/matching maps
\end{defn}

\begin{notation}
    Let $\cC$ be a small category. We denote by $\mathcal{B}$ the set of all boundary inclusions $\mathcal{B}=\{\partial\cC_c\to\cC_c \text{ for all } c\in\cC\}$
\end{notation}

\begin{theorem}\label{thm:Reedy_wfs}
Let $\cC$ be a Reedy category, and let $\mathcal{D}$ be a small complete and cocomplete category together with a weak factorization system $(\mathcal{L},\mathcal{R})$ on it. Then there is a a weak factorization system $(\mathcal{L}[\cC], \mathcal{R}[\cC])$ in the category of functors $\mathcal{D}^\mathcal{C}$ whose classes are defined as follows. 

\begin{enumerate}
    \item A map $\alpha\colon X\to Y$ in $\mathcal{D}^\cC$ is in the left class $\mathcal{L}[\cC]$ if its relative latching maps $\widehat{\ell}^c \alpha$ in $\mathcal{D}$, for all $c\in\cC$, are in the left class $\cL$.
    \item A map $\alpha\colon X\to Y$ in $\mathcal{D}^\cC$ is in the right class $\mathcal{R}[\cC]$ if its relative matching maps $\widehat{m}^c \alpha$ in $\mathcal{D}$, for all $c\in\cC$, are in  the right class $\mathcal{R}$.
\end{enumerate}

Moreover, whenever $(\mathcal{L},\mathcal{R})$ is cofibrantly generated by $\mathcal{J}$, the weak factorization system $(\mathcal{L}[\cC], \mathcal{R}[\cC])$ is cofibrantly generated by $\mathcal{B} \hat{\ast} \mathcal{J}=\{b\hat{\ast}j, \text{ for all } b\in\mathcal{B},j\in\mathcal{J}\}$ where $\cB$ is the set of boundary inclusions in $\cD^\cC$.
\end{theorem}

When working with model structures we care about Quillen functors and bifunctors, as they preserve the right amount of information by preserving some of the distinguished classes of morphisms. Since such classes come from weak factorization system, the question of whether we can abstract this to categories with weak factorization systems instead of a full model structure arises. In the following we see a way to do this.

\begin{defn}\label{def:left_Leibniz_bifunctor} Let $\cV,\cM,$ and $\cN$ be cocomplete categories each equipped with a weak factorization system: $(\cL,\cR), (\cL',\cR'),$ and $(\cL'',\cR'')$ respectively. A \emph{left Leibniz bifunctor} is a bifunctor $\otimes\colon\cV\times\cM\to\cN$ that is
\begin{rome}
    \item cocontinuous in each variable, and
    \item has the \emph{Leibniz property}: $\otimes$-pushout products of a map in $\cL$ with a map in $\cL'$ are in $\cL''$.
\end{rome}

Dually, a bifunctor between complete categories equipped with weak factorization systems is a \emph{right Leibniz bifunctor} if it is continuous in each variable and if pullback cotensors of maps in the right classes land in the right class. 
\end{defn}

\begin{lem}{\cite[C.2.11]{elements}}\label{lem:Leibniz_two_variable_adj} If the bifunctors 
$$\cV\times\cM\xrightarrow{\otimes}\cN,\hspace{1em} \cV^\op\times\cN\xrightarrow{\{,\}}\cM,\text{\hspace{1em} and\hspace{1em}} \cM^\op\times\cN\xlongrightarrow{\mathrm{hom}}\cV$$
define a two-variables adjunction, and $(\cL,\cR), (\cL',\cR'),$ and $(\cL'',\cR'')$ are three weak factorization systems on $\cV,\cM,$ and $\cN$ respectively, then the following are equivalent. 
\begin{enumerate}
    \item $\otimes\colon\cV\times\cM\to\cN$ defines a left Leibniz bifunctor.
    \item $\{,\}\colon \cV^\op\times\cN\to\cM$ defines a right Leibniz bifunctor. 
    \item $\mathrm{hom}\colon \cM^\op\times\cN\to\cV$ defines a right Leibniz bifunctor.
\end{enumerate}
\end{lem}

\begin{lem}{\cite[C.5.16]{elements}}\label{lem:weightedlimcolim_Leibniz} For any complete and cocomplete category $\cM$ with a weak factorization system $(\cL,\cR)$, and any Reedy category $\cJ$ the weighted colimit and weighted limit bifunctors
$$\mathrm{colim}\_\colon\Set^\cL\times\cM^{\cJ^\op}\to\cM\hspace{2em} and \hspace{2em} \mathrm{lim}\_\colon (\Set^\cJ)^\op\times\cM^\cJ\to\cM$$
are respectively left and right Leibniz bifunctors with respect to the Reedy weak factorization systems.
\end{lem}

\section{First order logic with dependent sorts (FOLDS)}\label{sec:folds}

First order logic with dependent sorts ($\folds$) was introduced by Makkai in \cite{makkai} in order to address in a systematic way the problem of determining what statements in higher category theory are equivalence invariant, with the right notion of equivalence in every case. Makkai's approach to do so is to create a language in which only statements that are invariant under the right notion of equivalence can be expressed in a meaningful way. These ideas have been used more recently to formulate and prove a general ``univalence principle'', see \cite{univalence_pple}.

To materialize these ideas in a particular framework one must develop three key components:
\begin{enumerate}
    \item\label{folds_language} The syntax, or language, in which we will express statements.
    \item\label{folds_objects} The abstract mathematical objects whose theory we want to study.
    \item\label{folds_interpretation} A way to interpret such language in terms of said mathematical objects.
\end{enumerate}

In this section we will focus on the construction of \cref{folds_language} and \cref{folds_interpretation}, and will make use of a more intuitive approach to \cref{folds_objects}.

We will not present a precise treatment of $\folds$ in this document, we follow \cite[Section 11.2]{elements} for the pieces we need. The interested reader may find an extensive treatment of $\folds$ in Makkai's original paper \cite{makkai}, and a short account with a more philosophical perspective in \cite{unfolding_folds}.

\subsection{Signature categories}

The language provided by Makkai's $\folds$ is determined by a given signature, that we now introduce. 

\begin{defn} A category $\cI$ is an \emph{inverse} category if there exists a functor $\deg\colon\cI\to\omega^\op$ that reflects identities.
\end{defn}

\begin{defn} A \emph{simple inverse} category is an inverse category $\cI$ that has the finite fan-out property, that is, that for any object $K$ of $\cI$, there are only finitely many morphisms with domain $K$.
\end{defn}

\begin{defn}
    A \emph{$\folds$ signature category} is a simple inverse category with a distinguished subset of maximal objects that we call relation symbols. We write a dot over the maximal objects that are relation symbols (note that there may not be any maximal object).
\end{defn}

\begin{terminology} We refer to the objects of a given $\folds$ signature as \emph{kinds}, whether or not they are relation symbols.
\end{terminology}

We denote $\folds$ signatures with an $\cL$, for ``language''.

\begin{rem}Let $\cL$ be a signature for $\folds$. Since $\cL$ is an inverse category, it comes equipped with a functor $\mathrm{deg}\colon\cL\to\omega^{\mathrm{op}}$ that determine ``levels'' for the set of objects of $\cL$. The objects at level $0$ are not the domain of non-identity morphisms.
\end{rem}

\begin{rem}\label{rem:dependency_maps}
Given a $\folds$ signature $\cL$, the morphisms between the objects embody the dependencies in the structure being described. More precisely, each morphism from one object to another of lower degree (including composites when pertinent) indicate a  parameter of the latter that the terms of the former depend on. See \cref{rem:dependency_cats} for an example in the signature category for $\cat$.
\end{rem}

\begin{rem}
Note that every object in a $\folds$ signature is finitely many ``steps'' from an object at degree $0$. This allow us to make definitions by recursion in the degree of the kinds.
\end{rem}

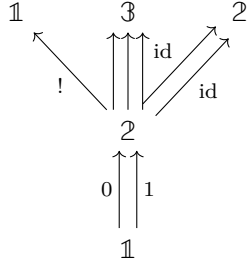
\begin{defn}\label{def:Lstructure}
For a given $\folds$ signature $\cL$, an $\cL$-structure is a functor $M\colon\cL\to \Set$ such that for each relation symbol $\dot{R}$ in $\cL$, the map induced by the family of non-identity arrows with domain $\dot{R}$ (depicted below) is a mono.
\begin{tz}
    \node[](1)      {$M\dot{R}$};
    \node[right of=1,xshift=1.5cm,yshift=-0.3cm](2)     {$\prod\limits_{p\colon \dot{R}\xrightarrow{\neq} K_p} MK_p$};
    \draw[>->] (1) to ($(2)+(-1cm,0.3cm)$);
\end{tz}    
\end{defn}

Before diving into more details of Makkai's dependent sorts, let us illustrate the previous notions in examples for categories and $2$-categories. These appear in \cite[Section 11.2]{elements}, and are originally due to Makkai.

\subsection*{The case of categories}

\begin{defn}\label{def:Icat}
The $\folds\,$ signature for categories, $\cL_\cat$ is given by 
\[
\begin{tikzcd}[row sep=large] 
\dot{I} \arrow[dr, "i" description] & \dot{T} \arrow[d, shift right=.6em, "\ell" description] \arrow[d, "r" description] \arrow[d, shift left=.6em, "c" description] & \dot{E} \arrow[dl, shift right=.3em, "\ell" description] \arrow[dl, shift left=.3em, "r" description] \\ 
& A \arrow[d, shift left=.3em, "t" description] \arrow[d, shift right=.3em, "s" description] \\ 
& O
\end{tikzcd}
\]
satisfying the relations
\begin{align*}
    s\cdot i &=t\cdot i         &           s\cdot\ell&=s\cdot \ell       &       s\cdot\ell &=s\cdot r\\
            &                   &           t\cdot c&=t\cdot r          &       t\cdot\ell     &=t\cdot r\\
            &                   &           t\cdot\ell&=s\cdot r           &       \\
\end{align*}

\end{defn}

The informal interpretation of this signature is that symbol $O$ represents the sort of objects, and $A$ that of arrows. While the relation symbols $\dot{I}, \dot{T},\text{ and }\dot{E}$ are respectively encoding identity arrows, ternary composition relations, and equality of parallel arrows. With this in mind, the equalities between composites in $\cL_\cat$ are clear.

\begin{rem}\label{rem:dependency_cats}
Let us look at \cref{rem:dependency_maps} in this particular case and preview syntax to be introduced in \cref{subsec:sorts_and_variables}. The simplest case is that of arrows: to write that the symbol $f$ is of kind  $A$ is not enough to write $f\colon A$, as $A$ is dependant on the kind $O$. Indeed, there are two arrows from $A$ to $O$, which means that we must consider first $x\colon O$ and $y\colon O$ to be able to write $f\colon A(x,y)$ indicating that $x$ denotes the domain of $f$, and $y$ its codomain. 

For $\dot{E}$, it depends on two parameters in $A$, which in turn each depends on two parameters in $O$, except the relations on $\cL_\cat$ tell us that sources and targets of the arrows coincide. So every element in $\dot{E}$ depends on $x,y\colon O$ and $f,g\colon A(x,y)$.
\end{rem}

The following construction gives us a natural way to obtain $\cL_\cat$-structures in a way that we recover categories as we know them.

\begin{constr}\label{def:Dcat}
    Let us now consider the functor $D_\cat\colon \cL_\cat^\op\to\cat$ whose image is represented by 

\[
\begin{tikzcd}[row sep=large]
\mathbb{1}      &\mathbb{3}         &\mathbb{2}\\
                &\mathbb{2}\arrow[ul, "!"] \arrow[u, shift left=.5em, ""]\arrow[u]\arrow[u, shift right=.5em, ""]\arrow[ur, shift left=.3em, "\id"]\arrow[ur, shift right=.3em, "\id"']       &\\
                &\mathbb{1}\arrow[u, shift left=.3em, "0"]\arrow[u, shift right=.3em, "1"']
\end{tikzcd}
\]
\end{constr}

Mapping out of $D_\cat$ gives us the functor $M\_\colon\cat\to\Set^{\cL_\cat}$ that takes values in the full subcategory of $\cL_\cat$-structures. With this, for any category $\cC$, the image of each kind is exactly what we want: $M_\cC O\coloneqq\hom(\mathbb{1},\cC)$ is the set of objects, $M_\cC A\coloneqq\hom(\mathbb{2},\cC)$ is the set of arrows, and $M_\cC \dot{I}, M_\cC \dot{T}$, and $M_\cC \dot{E}$ encodes identity arrows, commutative triangles, and equality of arrows. 

\begin{rem}\label{Kan_for_Dcat}
The above constructions fit together in the following diagram\textemdash the full power of setting things in this way will be apparent in \cref{sec:homotopy_meets_folds}.

\begin{tz}
\node[](1)  {$\cL^\op_{\cat}$};
\node[below of=1,xshift=-1.5cm,yshift=-0.60cm](2) {$\cat$};
\node[below of=1,xshift=1.5cm,yshift=-0.60cm](3) {$\Set^{\cL_{\cat}}$};

\draw[->] (1) to node[left,la]{$D_\cat$} (2);
\draw[->] (1) to node[right,la]{$\yo$} (3);

\draw[->, bend right=20] ($(2.east)-(-4pt,5pt)$) to node[below,la]{$\hom(D_\cat-,\cat)$} ($(3.west)-(0,5pt)$);
\draw[->, bend right=20] ($(3.west)+(0,5pt)$) to node[above,la]{$\Lan_\yo D_\cat$} ($(2.east)+(4pt,5pt)$);

\node[la] at ($(2.east)!0.5!(3.west)$) {$\bot$};
\end{tz}

\end{rem}

\subsection*{The case of $2$-categories}

\begin{defn}\label{def:L2cat}The $\folds\,$ signature for $2$-categories, $\cL_{\twocat}$ is given by 
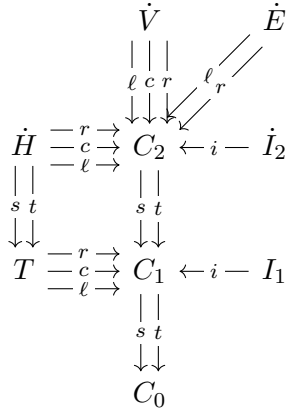
\[
\begin{tikzcd}[row sep=large]
& \dot{V}  \arrow[d, shift right=.6em, "\ell" description] \arrow[d, "c" description] \arrow[d, shift left=.6em, "r" description] &  \dot{E} \arrow[dl, shift right=.3em, "\ell" description] \arrow[dl, shift left=.3em, "r" description] \\
\dot{H} \arrow[r, shift right=.6em, "\ell" description] \arrow[r, "c" description] \arrow[r, shift left=.6em, "r" description] \arrow[d, shift left=.3em, "t" description] \arrow[d, shift right=.3em, "s" description]& C_2 \arrow[d, shift left=.3em, "t" description] \arrow[d, shift right=.3em, "s" description]& \dot{I}_2 \arrow[l, "i" description] \\
T  \arrow[r, shift right=.6em, "\ell" description] \arrow[r, "c" description] \arrow[r, shift left=.6em, "r" description] & C_1\arrow[d, shift left=.3em, "t" description] \arrow[d, shift right=.3em, "s" description] & I_1 \arrow[l, "i" description] \\
& C_0
\end{tikzcd}
\]

Naturally, to encode the structure of $2$-categories we must ask $\cL_{\twocat}$ to satisfy several relations. Namely, the relations of its underlying category (except for equality of parallel arrows, since that is not a relevant notion in this context),
\begin{align*}
    s\cdot i &=t\cdot i         &           s\cdot\ell&=s\cdot \ell  &\\
            &                   &           t\cdot c&=t\cdot r    &\\
            &                   &           t\cdot\ell&=s\cdot r,    &     \\
\end{align*}
the relations describing the shape of the $2$-cells
\begin{align*}
s\cdot s &=s\cdot t     &t\cdot s&=t\cdot t,
\end{align*}
relations describing vertical composition, identities, and equality between parallel $2$-cells, 
\begin{align*}
    s\cdot c &= s\cdot \ell     &s\cdot i&=t\cdot i     &s\cdot \ell&= s\cdot r\\
    t\cdot c &=t\cdot r         &         &             &t\cdot\ell&=t\cdot r\\
    t\cdot\ell&=s\cdot r,
\end{align*}
and finally relations linking horizontal composition of $2$-cells with composition of $1$-cells
\begin{align*}
    s\cdot\ell&=\ell\cdot s      &s\cdot r&=r\cdot s            &s\cdot c&=c\cdot s\\
    t\cdot\ell&=\ell\cdot t        &t\cdot r&=r\cdot t          &t\cdot c&=c\cdot t.
\end{align*} 
    
\end{defn}

Similarly to what we did for categories, we construct a diagram with values in $\twocat$. 

\begin{constr}\label{def:E2cat}

Now consider the functor $E_{\twocat}\colon\cL_{\twocat}^\op\to\twocat$ defined by

\[
\begin{tikzcd}
& \Sigma[\mathbb{3}]  \arrow[from=d, shift right=.6em] \arrow[from=d] \arrow[from=d, shift left=.6em] &  \Sigma[\mathbb{2}] \arrow[from=dl, shift right=.3em] \arrow[from=dl, shift left=.3em] \\
\mathbb{H}_{=} \arrow[from=r, shift right=.6em] \arrow[from=r] \arrow[from=r, shift left=.6em] \arrow[from=d, shift left=.3em] \arrow[from=d, shift right=.3em]& \Sigma[\mathbb{2}] \arrow[from=d, shift left=.3em] \arrow[from=d, shift right=.3em]& \mathbb{2} \arrow[from=l] \\
\mathbb{3}  \arrow[from=r, shift right=.6em] \arrow[from=r] \arrow[from=r, shift left=.6em] & \mathbb{2} \arrow[from=d, shift left=.3em] \arrow[from=d, shift right=.3em]\arrow[r]       &\mathbb{1}\\
& \mathbb{1} & 
\end{tikzcd} \qquad
\]

where
\[ \quad 
\Sigma\mathbb{2} \coloneqq
\begin{tikzcd} \bullet \arrow[r, bend left] \arrow[r, bend right] \arrow[r, phantom, "\scriptstyle\Downarrow"] & \bullet 
\end{tikzcd}, \quad 
\Sigma[\mathbb{3}] \coloneqq
\begin{tikzcd}\bullet \arrow[r] \arrow[r, bend left=50, "\Downarrow"'] \arrow[r, bend right=50, "\Downarrow"] & \bullet
\end{tikzcd}, 
\quad \text{and} \quad
 \mathbb{H}_{=} \coloneqq
\begin{tikzcd} 
\bullet \arrow[r, bend left]  \arrow[r, bend right] \arrow[r, phantom, "\scriptstyle\Downarrow"] & \bullet \arrow[r, bend left] \arrow[r, bend right] \arrow[r, phantom, "\scriptstyle\Downarrow"] & \bullet 
\end{tikzcd}
\]

\end{constr}

\subsection{Dependent sorts and their variables}\label{subsec:sorts_and_variables} We are finally ready to start introducing the language determined by any given signature category $\cL$. For this we also follow \cite[Section 11.2]{elements}, which defines sorts and variables by mutual recursion and introduces the variables belonging to a \emph{context}. The variables can be introduced separately as done in \cite{makkai} and \cite{unfolding_folds}, although that approach assumes greater familiarity with first order logic.

\begin{defn}
Given a $\folds$ signature $\cL$, a context is an $\cL$-structure $\Gamma\colon\cL\to\Set$ in which the sets $\Gamma K$ associated to each kind $K$ in $\cL$ are disjoint.
\end{defn}

We proceed to simultaneously define sorts and their variables.

\begin{defn} Consider a $\folds$ signature $\cL$ and a context $\Gamma\colon\cL\to\Set$. 
\begin{itemize}
    \item Each kind $K$ of degree zero defines a \emph{sort}, that we also denote by $K$, whose \emph{variables} are the elements of the set $\Gamma K$. We  write ``$x: K$'' to mean that $x\in\Gamma K$, that is, that $x$ is a variable belonging to the sort $K$ in context $\Gamma$.
    \item For each kind $K$ of degree one, the matching map $m^K$ of definition \cref{def:matching_latching_maps} becomes 
    \begin{tz}
        \node[](1) {$\Gamma K$};
        \node[right of=1,xshift=1.5cm,yshift=-0.3cm](2) {$\prod\limits_{p\colon K\xrightarrow{\neq} K_p} \Gamma K_p$};
        \draw[->] (1) to node[above,la]{$m^K$} ($(2)+(-0.8cm,0.3cm)$);
    \end{tz}
    where the product on the right is over all arrows in $\cL$ with domain $K$ and codomain of degree zero. For any family of variables $\{x_p\colon K_p\}_{p\colon K\xrightarrow{\neq} K_p}$, there is a sort $K\langle x_p\rangle$ whose variables are the elements of the fiber
    \begin{tz}
        \node[](1) {$\Gamma K_{\langle x_p\rangle}$};
        \node[right of=1,xshift=1.5cm](2) {$\Gamma K$};
        \node[below of=1,yshift=-.3cm](3) {$1$};
        \node[below of=2,yshift=-.3cm](4) {$\prod\limits_{p\colon K\xrightarrow{\neq} K_p} \Gamma K_p$};

        \draw[->] (1) to (2);
        \draw[->] (1) to (3);
        \draw[->] (2) to node[right,la]{$m^K$} (4);
        \draw[->] (3) to node[below,la]{$\langle x_p\rangle$} (4);
    \end{tz}
    The variables of $K\langle x_p\rangle$ are said to \emph{depend on} the variables $x_p\colon K_p$.
    \item Finally, let $K$ be a kind of higher degree. A family of variables $$\{x_p\colon K_p\langle x_{qp}\rangle\}_{p\colon K\xrightarrow{\neq} K_p}$$ is \emph{compatible} if $\langle x_p\rangle\in \prod_{p\colon K\xrightarrow{\neq} K_p} \Gamma K_p $ belongs to the matching object $\partial^K\Gamma$; in other words, this says that the higher degree variables in the list depend on the lower degree ones in the way determiend by the dependency relations codifying in the $\folds$ signature. Given a compatible family of variables, there is a sort $K\langle x_p\rangle$ whose variables are the elements of the fiber of the matching map $m^K$ over $\langle x_p\rangle\in\partial^K\Gamma$.
    \begin{tz}
        \node[](1) {$\Gamma K_{\langle x_p\rangle}$};
        \node[right of=1,xshift=1.5cm](2) {$\Gamma K$};
        \node[below of=1,yshift=-.3cm](3) {$1$};
        \node[below of=2,yshift=-.3cm](4) {$\partial^K\Gamma$};

        \draw[->] (1) to (2);
        \draw[->] (1) to (3);
        \draw[->] (2) to node[right,la]{$m^K$} (4);
        \draw[->] (3) to node[below,la]{$\langle x_p\rangle$} (4);
    \end{tz}
\end{itemize}
\end{defn}

When specifying a sort, the list of variables it depends on can quickly become impractical to handle, and for that reason it is often the case that only the variables of highest degree are listed\textemdash exploiting the fact that the lower degree variables can be deduced from the sorts to which the higher degree ones belong.

\begin{ex}Consider the $\folds$ signature category $\cL_\cat$, and a context $\Gamma\colon\cL\to \Set$. For the kind $O$ we have the sort $\Gamma O$ whose variables would represent the objects of a category. For the kind $A$, the sort is given by the pullback

\begin{tz}
    \node[](1) {$\Gamma A_{\langle x,y\rangle}$};
    \node[right of=1,xshift=1.5cm](2) {$\Gamma A$};
    \node[below of=1,yshift=-.3cm](3) {$1$};
    \node[below of=2,yshift=-.3cm](4) {$\Gamma O\times \Gamma O$};

    \draw[->] (1) to (2);
    \draw[->] (1) to (3);
    \draw[->] (2) to node[right,la]{$(s,t)$} (4);
    \draw[->] (3) to node[below,la]{$\langle x,y\rangle$} (4);
\end{tz}

The variables of this sort are what correspond to arrows between from $x$ to $y$ \textemdash as specified in the definition, the variables of this sort depend on the variables $x,y$. 

For the sort associated to the relation symbol $\dot{T}$, we consider the pullback below

\begin{tz}
    \node[](1) {$\Gamma \dot{T}_{\langle f,g,h\rangle}$};
    \node[right of=1,xshift=1.5cm](2) {$\Gamma \dot{T}$};
    \node[below of=1,yshift=-.3cm](3) {$1$};
    \node[below of=2,yshift=-.3cm](4) {$\Gamma A\times_{\Gamma O} \Gamma A\times_{\Gamma O}\Gamma A\times_{\Gamma O}$};

    \draw[->] (1) to (2);
    \draw[->] (1) to (3);
    \draw[->] (2) to (4);
    \draw[->] (3) to node[below,la]{$\tau$} (4);
\end{tz}
where $\tau\colon 1\to \Gamma A\times_{\Gamma O} \Gamma A\times_{\Gamma O}\Gamma A\times_{\Gamma O}$ picks a triangle 
\begin{tz}
\node[](1)      {$x$};
\node[right of=1](2)    {$z$};
\node[above of=1,xshift=0.75cm,yshift=-0.75cm](3)  {$y$};

\draw[->] (1) to node[la,above,pos=0.3]{$f$}    (3);
\draw[->] (3) to node[la,above,pos=0.7]{$g$}    (2);
\draw[->] (1) to node[la,below]{$h$}    (2);
\end{tz}

For the sort associated to the relation symbol $\dot{E}$, we consider the pullback below

\begin{tz}
    \node[](1) {$\Gamma \dot{E}_{\langle f,g\rangle}$};
    \node[right of=1,xshift=1.5cm](2) {$\Gamma \dot{E}$};
    \node[below of=1,yshift=-.3cm](3) {$1$};
    \node[below of=2,yshift=-.3cm](4) {$\Gamma A _{\Gamma O}\times_{\Gamma O} \Gamma A$};

    \draw[->] (1) to (2);
    \draw[->] (1) to (3);
    \draw[->] (2) to (4);
    \draw[->] (3) to node[below,la]{$\alpha$} (4);
\end{tz}

where the map $\alpha\colon 1\to \Gamma A _{\Gamma O}\times_{\Gamma O} \Gamma A$ pics two parallel arrows

\begin{tz}
\node[](1) {$x$};
\node[right of=1](2) {$y$};

\draw[->] (1) to [bend left=25] node[above,la]{$f$} (2);
\draw[->] (1) to [bend right=25] node[below,la]{$g$} (2);
\end{tz}

Finally, for the sort associated to the relation symbol codifying identity arrows $\dot{I}$, we consider the pullback 

\begin{tz}
    \node[](1) {$\Gamma \dot{I}_{\langle x\rangle}$};
    \node[right of=1,xshift=1.5cm](2) {$\Gamma \dot{I}$};
    \node[below of=1,yshift=-.3cm](3) {$1$};
    \node[below of=2,yshift=-.3cm](4) {$\mathrm{eq}(\Gamma A \rightrightarrows\Gamma O)$};

    \draw[->] (1) to (2);
    \draw[->] (1) to (3);
    \draw[->] (2) to (4);
    \draw[->] (3) to node[below,la]{$x$} (4);
\end{tz}
\end{ex}

Having presented the building blocks of the formal language associated to a $\folds$ signature $\cL$, we now proceed to introduce its formulae and sentences. We also build these by recursion, starting with atomic formulae. 

\begin{defn}\label{def:atomic_formula} Let $\cL$ be $\folds$ signature category and $\Gamma\colon\cL\to\Set$ a context. An \emph{atomic formula} in the logic with dependent sorts is an entity of the form $\dot{R}\langle x_p\rangle$, where $\dot{R}$ is a relation symbol and the variables $\{x_p\colon K_p\langle x_{qp}\rangle\}_{p\colon \dot{R}\xrightarrow{\neq}K_p}$ is a compatible family, meaning that each such family defines an element $\langle x_p\rangle\in\partial^{\dot{R}}\Gamma$.
\end{defn}

The main rule to understand predicates in $\folds$ is that the quantification, either universal or existential, can only be asserted over the variables in a specified sort, provided the variables in the predicate  under consideration do not depend on the variable being quantified over.

\begin{defn}
We define \emph{formulae} $\phi$ and their sets of \emph{free variables} $\var{\phi}$ by simultaneous recursion.
\begin{itemize}
    \item An atomic formula is a formula, and the variables of $\dot{R}\langle x_p\rangle$ are the $x_p$.
\end{itemize}
Compound formulae are defined inductively from other formulae via the following rules:
\begin{itemize}
    \item $\top,\bot$ are formulae, with $\var(\top)=\var(\bot)\coloneqq\emptyset$. 
    \item Formulae may be combined using the sentential connectives $\wedge,\vee,$ and $\to$, in which case the variables are given by unions: e.g.\ $\var(\phi\to\psi)=\var(\phi)\cup\var(\psi).$
    \item When $\phi$ is a formula and $x$ is a variable such that no variable in $\var(\phi)$ depends on $x$\textemdash it is, however, permitted that $x\in\var{\phi}$\textemdash then $\forall x\phi$ and $\exists x \phi$ are well-formed formulae whose variables are given by the set
    $$\var(\forall x\phi)=\var(\exists x\phi)\coloneqq (\var(\phi)-\{x\})\cup\mathrm{dep}(x)$$
    formed by removing $x$ from the variables of $\phi$ if it appears and then adding all the variables on which $x$ depends if they do not already appear.
\end{itemize}
\end{defn}

\begin{rem}
    Given a context $\Gamma$ and a formula $\phi$, we have that $\var(\phi)\subset\Gamma$ is another context.
\end{rem}

\begin{defn}A formula $\phi$ is a \emph{sentence} when $\var(\phi)=\emptyset$. 
\end{defn}

\begin{defn}Let $\cL$ be a $\folds$ signature, $M$ an $\cL$-structure, and $\Gamma\colon\cL\to\Set$ an $\cL$-context. An \emph{evaluation} of $\Gamma$ in $M$ is given by a natural transformation $\alpha\colon\Gamma\to M$. 
\end{defn}

\begin{rem}
    An evaluation of an $\cL$-context $\Gamma$ in an $\cL$-structure $M$ defines an interpretation of its variables.
\end{rem}

We now define what it means for an $\cL$-structure $M$ to satisfy a formula $\phi$ under a given interpretation of its variables $\alpha\colon\var(\phi)\to M$, by induction in the complexity of the formula $\phi$. Note that it suffices to interpret the variables of a given formula.

\begin{defn}
    An $\cL$-structure $M$ \emph{satisfies} an atomic formula $\dot{R}\langle x_p\rangle$ under an interpretation if and only if the tuple $\langle \alpha x_p\rangle\in \prod_{p\colon\dot{R}\to K_p}MK_p$ lies in the subset $M\dot{R}$, in which case we write $M\vDash \dot{R}[\alpha].$
\end{defn}

The sentences $\top$ and $\bot$ have no variables, so their semantics are independent of any interpretation. Any $\cL$-structure satisfies $\top$ and no $\cL$-structure satisfies $\bot$.

\begin{defn}
    An $\cL$-structure $M$ \emph{satisfies} compound formulas built from $\phi$ and $\psi$ under an interpretation $\alpha\colon\var(\phi)\cup\var(\psi)\to M$ according to the following rules.
    \begin{enumerate}
        \item $M\vDash (\phi\wedge\psi)[\alpha]$ if and only if $M\vDash\phi[\alpha]$ and $M\vDash\psi[\alpha]$.
        \item $M\vDash (\phi\vee\psi)[\alpha]$ if and only if $M\vDash\phi[\alpha]$ or $M\vDash\psi[\alpha]$.
        \item $M\vDash (\phi\to\psi)[\alpha]$ if and only if whenever $M\vDash\phi[\alpha]$ then also $M\vDash\psi[\alpha]$.
    \end{enumerate}
\end{defn}

\begin{defn}
    Consider a formula $\forall x\phi$ where $x\colon K\langle x_p\rangle$ together with an interpretation $\alpha$ in an $\cL$-structure $\alpha\colon\var(\forall x\phi)\to M$.\footnote{Note that $x\notin\var(\forall x\phi)$, so $\alpha$ does not give an interpretation of the variable $x$ itself.} 
    \begin{itemize}
        \item We say that $M$ \emph{satisfies} the formula $\forall x\phi$ under the interpretation $\alpha$ if for all $a$ in the fiber of $MK\to\partial^K M$ over $\langle\alpha x_p\rangle$, $M\vDash\phi(a/x)[\alpha]$. That is, $M\vDash\forall x\phi[\alpha]$ if any $a$ with appropriate dependencies can be substituted for $x$ in the interpretation of $\phi$ to yield a formula that $M$ satisfies.
        \item We say that $M$ \emph{satisfies} the formula $\exists x\phi$ under the interpretation $\alpha$ if there is some $a$ in the fiber of $MK\to\partial^K M$ over $\langle\alpha x_p\rangle$, such that $M\vDash\phi(a/x)[\alpha]$. That is, $M\vDash\exists x\phi[\alpha]$ if some $a$ with appropriate dependencies can be substituted for $x$ in the interpretation of $\phi$ to yield a formula that $M$ satisfies.
    \end{itemize}    
\end{defn}

\begin{rem} For the relation symbols, we will often write $\dot{R}(r)$ instead of $\exists r:\dot{R}$ to say that it's inhabited by $r$.
\end{rem}

Now that we have introduced the language, in the next examples we illustrate the difference between sentence and formula.

\begin{ex}
$\qquad$
\begin{enumerate}    
\item The formula 
$$\exists f: A(x,x).\dot{I}(f)$$
in $\cL_\cat$ is not a sentence as the set of free variables is $\{x\}\neq\emptyset$; this says that the particular object $x$ has an identity. A sentence saying that for every $x\in O$ there is an identity arrow would be
$$\forall x: O.\exists f: A(x,x).\dot{I}(f)$$

\item The formula
    $$\exists\tau :\dot{T}(f,g,h,x,y,z)$$
    is not a sentence, as its free variables are $\{x,y,z,f,g,h\}$. However, it is a sentence its universal closure:
    $$\forall x: O.\forall y: O.\forall z: O.\forall f: A(x,y).\forall g: A(y,z).\exists h: A(x,z).\dot{T}(f,g,h)$$
saying that for every pair of composable arrows, there is a composite.

\item The formula 
$$\exists g:A(y,x).\exists\id_y:A(y,y).\exists\id_x:A(x,x).\dot{T}(f,g,\id_x)\wedge\dot{T}(g,f,\id_y)\wedge\dot{I}(\id_y)\wedge\dot{I}(\id_x)$$
is not a sentence, as its free variables are $\{x,y,f\}$. We can bind all its variables as below, to obtain the sentence ``every arrow has an inverse'', which is the universal closure of the formula above:
$$\forall x: O.\forall y\colon O.\forall f: A(x,y).\exists g: A(y,x).\exists\id_y:A(y,y).\exists\id_x:A(x,x)$$
$$\hspace{10em}.\dot{T}(f,g,\id_x)\wedge\dot{T}(g,f,\id_y)\wedge\dot{I}(\id_y)\wedge\dot{I}(\id_x)$$
\end{enumerate}
\end{ex}

We now move towards describing when two different $\cL$-structures satisfy the same formulae.

\begin{defn}\label{def:fiberwise_surj}
A natural transformation of $\cL$-structures, $\rho\colon M\to N$, is fiberwise surjective if for every object $K$ in $\cL$ the map to the pullback in the ``matching square'' is an epimorphism. 

\begin{tz}
\node[](1) {$\bullet$}; 
\node[right of=1,xshift=0.5cm](2) {$M^KX$}; 
\node[below of=1](1') {$Y(K)$}; 
\node[below of=2](2') {$M^KY$}; 

\draw[->] (1) to (2); 
\draw[->] (1) to (1'); 
\draw[->] (2) to node[right, la]{$M^K\rho$} (2'); 
\draw[->] (1') to node[below, la]{$m^K$} (2');

\node[above left of=1,xshift=-.5cm](5) {$X(K)$};
\draw[->,bend left] (5) to node[above,la]{$m^K$} (2); 
\draw[->,bend right] (5) to node[below,la, xshift=-2pt]{$\rho^K$} (1'); 
\draw[->>,dashed] (5) to node[above,la,pos=0.5,xshift=2pt]{$\widehat{m}^K\rho$} (1);
\node at ($(1)-(-.3cm,.3cm)$) {$\lrcorner$};
\end{tz}

\end{defn}

\begin{ex}
Let us unpack what this means in a particular case. Consider $\cL_\cat$ the signature category for categories, let $M,N\colon\cL_\cat\to \Set$ be two $\cL_\cat$-structures, and $\rho\colon M\to N$ a fibjerwise surjective natural transformation. 

For the kind $O$, the condition as depicted below
\begin{tz}
\node[](1) {$N(O)$}; 
\node[right of=1,xshift=0.5cm](2) {$\ast$}; 
\node[below of=1](1') {$N(O)$}; 
\node[below of=2](2') {$\ast$}; 

\draw[->] (1) to (2); 
\draw[->] (1) to (1'); 
\draw[->] (2) to (2'); 
\draw[->] (1') to  (2');

\node[above left of=1,xshift=-.5cm](5) {$M(O)$};

\draw[->,bend left] (5) to  (2); 
\draw[->,bend right] (5) to node[below,la, xshift=-2pt]{$\rho_O$} (1'); 
\draw[->>,dashed] (5) to (1);
\node at ($(1)-(-.3cm,.3cm)$) {$\lrcorner$};
\end{tz}
codifies surjectivity on objects. 

For the kind $A$:
\begin{tz}
\node[](1) {$\{(x,y,f)\}$}; 
\node[right of=1,xshift=1.5cm](2) {$M(O)\times M(O)$}; 
\node[below of=1](1') {$N(A)$}; 
\node[below of=2](2') {$N(O)\times N(O)$}; 

\draw[->] (1) to (2); 
\draw[->] (1) to (1'); 
\draw[->] (2) to node[right,la]{$\rho_O\times\rho_O$} (2'); 
\draw[->] (1') to node[below,la]{$(s,t)$}  (2');

\node[above left of=1,xshift=-.5cm](5) {$M(A)$};

\draw[->,bend left] (5) to node[above,la]{$(s,t)$} (2); 
\draw[->,bend right] (5) to node[below,la, xshift=-2pt]{$\rho_A$} (1'); 
\draw[->>,dashed] (5) to (1);
\node at ($(1)-(-.3cm,.3cm)$) {$\lrcorner$};
\end{tz}
codifies fullness.

For $I$, the fiberwise surjectivity condition becomes
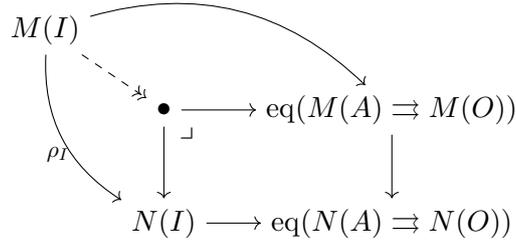
\begin{tz}
\node[](1) {$\bullet$}; 
\node[right of=1,xshift=1.5cm](2) {$\mathrm{eq}(M(A)\rightrightarrows M(O))$}; 
\node[below of=1](1') {$N(I)$}; 
\node[below of=2](2') {$\mathrm{eq}(N(A)\rightrightarrows N(O))$}; 

\draw[->] (1) to (2); 
\draw[->] (1) to (1'); 
\draw[->] (2) to (2'); 
\draw[->] (1') to  (2');

\node[above left of=1,xshift=-.5cm](5) {$M(I)$};

\draw[->,bend left] (5) to  (2); 
\draw[->,bend right] (5) to node[below,la, xshift=-2pt]{$\rho_I$} (1'); 
\draw[->>,dashed] (5) to (1);
\node at ($(1)-(-.3cm,.3cm)$) {$\lrcorner$};
\end{tz}
which codifies reflection of identities locally, this is, only for endomorphisms.

For $E$, this condition is 
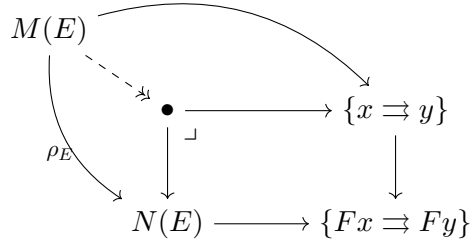
\begin{tz}
\node[](1) {$\bullet$}; 
\node[right of=1,xshift=1.5cm](2) {$\{x\rightrightarrows y\}$}; 
\node[below of=1](1') {$N(E)$}; 
\node[below of=2](2') {$\{Fx\rightrightarrows Fy\}$}; 

\draw[->] (1) to (2); 
\draw[->] (1) to (1'); 
\draw[->] (2) to (2'); 
\draw[->] (1') to  (2');

\node[above left of=1,xshift=-.5cm](5) {$M(E)$};

\draw[->,bend left] (5) to  (2); 
\draw[->,bend right] (5) to node[below,la, xshift=-2pt]{$\rho_E$} (1'); 
\draw[->>,dashed] (5) to (1);
\node at ($(1)-(-.3cm,.3cm)$) {$\lrcorner$};
\end{tz}
which encodes faithfulness.

For $T$, this condition is 
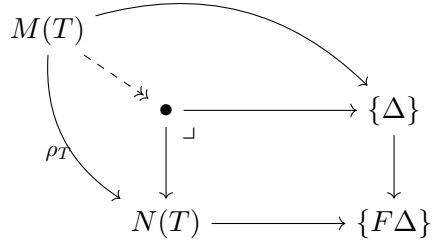
\begin{tz}
\node[](1) {$\bullet$}; 
\node[right of=1,xshift=1.5cm](2) {$\{\Delta\}$}; 
\node[below of=1](1') {$N(T)$}; 
\node[below of=2](2') {$\{F\Delta\}$}; 

\draw[->] (1) to (2); 
\draw[->] (1) to (1'); 
\draw[->] (2) to (2'); 
\draw[->] (1') to  (2');

\node[above left of=1,xshift=-.5cm](5) {$M(T)$};

\draw[->,bend left] (5) to  (2); 
\draw[->,bend right] (5) to node[below,la, xshift=-2pt]{$\rho_T$} (1'); 
\draw[->>,dashed] (5) to (1);
\node at ($(1)-(-.3cm,.3cm)$) {$\lrcorner$};
\end{tz}
where $\{\Delta\}$ is the set of not necessarily commutative triangles, and $M(T)$ is the set of commutative triangles, and analogously for $N$. This condition encodes, knowing that there is such a triangle, reflection of commutativity.
\end{ex}

\begin{defn}Let $\cL$ be a $\folds$ signature, let $M,N\colon\cL\to\Set$ be two $\cL$-structures, and $\Gamma\colon \cL\to \Set$ a context together with given interpretations $\alpha\colon\Gamma\to M$ and $\beta\colon\Gamma\to N$. The $\cL$-structures $M$ and $N$ are $\cL$-equivalent in context $\Gamma$ if there is a diagram of fiberwise surjections $\sigma$ and $\rho$ of $\cL$-structures under the context $\Gamma$

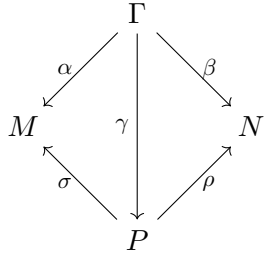
\begin{tz}
\node[](1) {$\Gamma$};
\node[below of=1,xshift=-1.5cm](2) {$M$};
\node[below of=1,xshift=1.5cm](3) {$N$};
\node[below of=3,xshift=-1.5cm](4) {$P$};

\draw[->] (1) to node[left,la]{$\alpha$} (2);
\draw[->] (1) to node[right,la]{$\beta$} (3);
\draw[->] (1) to node[left,la]{$\gamma$} (4);
\draw[->] (4) to node[left,la]{$\sigma$} (2);
\draw[->] (4) to node[right,la]{$\rho$} (3);
\end{tz}
In this case we write $M\simeq_{\cL}^{\Gamma} N$.
\end{defn} 

\begin{theorem}\label{thm:same_formulae}
If $M$ and $N$ are $\cL$-equivalent in a context $\Gamma$
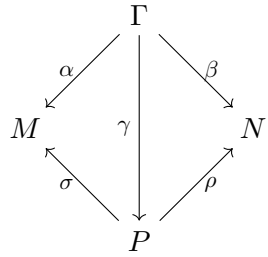
\begin{tz}
\node[](1) {$\Gamma$};
\node[below of=1,xshift=-1.5cm](2) {$M$};
\node[below of=1,xshift=1.5cm](3) {$N$};
\node[below of=3,xshift=-1.5cm](4) {$P$};

\draw[->] (1) to node[left,la]{$\alpha$} (2);
\draw[->] (1) to node[right,la]{$\beta$} (3);
\draw[->] (1) to node[left,la]{$\gamma$} (4);
\draw[->] (4) to node[left,la]{$\sigma$} (2);
\draw[->] (4) to node[right,la]{$\rho$} (3);
\end{tz}
then for all formulae $\phi$ with $\mathrm{var}(\phi)\subset\Gamma$, $M\vDash \phi[\alpha]$ if and only if $N\vDash\phi[\beta]$.
\end{theorem}

\begin{proof}
    This proof is done by induction. See, for example, \cite[Theorem 11.2.28]{elements}.
\end{proof}

\section{Homotopy theory meets FOLDS}\label{sec:homotopy_meets_folds}

In forthcoming work \cite{HenryBardomiano}, summarized in \cite{HenryFolds}, Henry and Bardomiano Mart\'inez show that the language induced by a $\folds$ signature category can be meaningfully related with certain model structures. Following their ideas, we prove that if the $\folds$ signature is connected to a model structure as in condition $\bigtriangleup$, then a weak equivalence between fibrant objects preserves the truth of any mathematical statement that is expressible in the FOLDS language. 

As shown in \cref{sec:folds}, the formal language is completely determined by a $\folds$ signature category, the model structure enters into play when we interpret that language in $\cat$, $\twocat$, and $\dblcat$ depending on the case.

In this section we want to make sense of and complete the table below.

\begin{table}[ht]
    \centering
    \begin{tabularx}{0.98\textwidth}{c|c|c|c}
    Model structure & weak equiv. & fibrant objects & formal language \\ \hline
    
    Folk MS on $\cat$ & equivalences & all categories &  \begin{tabular}{@{}c@{}}equivalence invariant \\ $1$-category theory\end{tabular} \\
    \hline
    
    Lack's MS on $\twocat$ & biequivalences & all $2$-categories &  \begin{tabular}{@{}c@{}}biequivalence invariant \\ $2$-category theory\end{tabular}  \\ 
    \hline
    
     \begin{tabular}{@{}c@{}} MS on $\dblcat$ \\ (\cref{MS_equipments})\end{tabular} &  \begin{tabular}{@{}c@{}} double \\ biequivalences \end{tabular}& equipments & \begin{tabular}{@{}c@{}c@{}c@{}c@{}}double biequiv.\ invariant\\ ``equipment theory''\\ $\cup$\\ model independent \\ (some) $\infty$-category theory\footnotemark\end{tabular}\\ 
    \hline 
    \end{tabularx}
    \vspace{.1cm}
    \label{tab:MS_and_formal_languages}
\end{table}
\footnotetext{In future work Verity and Riehl will show that the virtual equipments of modules associated to an $\infty$-cosmos of $(\infty,1)$-categories is actually a (pseudo)equipment.}

\subsection{The nerve-realization construction}\label{nerve_construction}
The situation in \cref{Kan_for_Dcat} generalizes to $\twocat$ and $\dblcat$, and placing these elements into a more general construction will simplify some future computations, so let us look at the following setting. 

Let $\cE$ be a cocomplete category and $\cC$ a small category, together with a functor $D\colon\cC\to\cE$. Considering the Yoneda embedding we can assemble these functors into the span below-left.

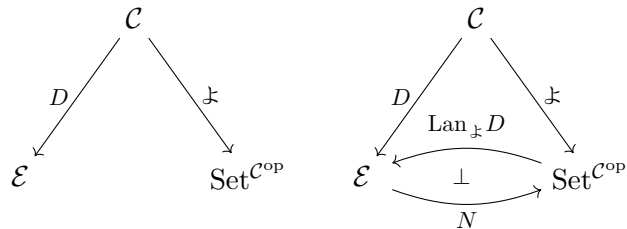
\begin{tz}
\node[](1)  {$\cC$};
\node[below of=1,xshift=-1.5cm,yshift=-0.60cm](2) {$\cE$};
\node[below of=1,xshift=1.5cm,yshift=-0.60cm](3) {$\Set^{\C^\op}$};

\draw[->] (1) to node[left,la]{$D$} (2);
\draw[->] (1) to node[right,la]{$\yo$} (3);

\node[right of=1,xshift=3cm](1)  {$\cC$};
\node[below of=1,xshift=-1.5cm,yshift=-0.60cm](2) {$\cE$};
\node[below of=1,xshift=1.5cm,yshift=-0.60cm](3) {$\Set^{\C^\op}$};

\draw[->] (1) to node[left,la]{$D$} (2);
\draw[->] (1) to node[right,la]{$\yo$} (3);

\draw[->, bend right=20] ($(2.east)-(-4pt,5pt)$) to node[below,la]{$N$} ($(3.west)-(0,5pt)$);
\draw[->, bend right=20] ($(3.west)+(0,5pt)$) to node[above,la]{$\Lan_\yo D$} ($(2.east)+(4pt,5pt)$);

\node[la] at ($(2.east)!0.5!(3.west)$) {$\bot$};

\end{tz}

Since $\cE$ is cocomplete, we know that the left Kan extension of the functor $D$ along the Yoneda embedding $\yo$ exists, let us denote it by $\Lan_\yo D$. Furthermore, when $\cE$ is locally small it accepts a right adjoint given by $N\coloneqq\hom(D-,-)$ that sends an object in $E\in\cE$ to $\hom(D-,E)$. In summary, the diagram above-left gives rise to a diagram as above-right, with $N\coloneqq\hom_\cE(D-,-)$.

\begin{rem} For future computations, it is worth highlighting that $\Lan_\yo D$ is the weighted colimit $\Lan_\yo D=-\ast_\cC D$.
\end{rem}

As we did for the case of $\cat$, in this paper we will use this setting when the category $\cC$ is the opposite of a $\folds$ signature category $\cL$ for a mathematical structure, and $\cE$ a category of models $\cM$ for such mathematical structure\textemdash for example, the category $\cat$ for $\cL_\cat$. As in \cref{def:Dcat}, suppose also that we have a functor $D\colon\cL^\op\to\cM$. Then the construction above becomes

\begin{equation}\label{nerve-construction-for-folds}
\end{equation}
\begin{tz}
\node[](1)  {$\cL^\op$};
\node[below of=1,xshift=-1.5cm,yshift=-0.60cm](2) {$\cM$};
\node[below of=1,xshift=1.5cm,yshift=-0.60cm](3) {$\Set^{\cL}$};

\draw[->] (1) to node[left,la]{$D$} (2);
\draw[->] (1) to node[right,la]{$\yo$} (3);

\draw[->, bend right=20] ($(2.east)-(-4pt,5pt)$) to node[below,la]{$\hom(D-,-)$} ($(3.west)-(0,5pt)$);
\draw[->, bend right=20] ($(3.west)+(0,5pt)$) to node[above,la]{$-\ast_\cC D$} ($(2.east)+(4pt,5pt)$);
\node[la] at ($(2.east)!0.5!(3.west)$) {$\bot$};
\end{tz}

\subsection{A canonical weak factorization system in $\Set^\cL$}\label{subsec:mono_epi}

For any $\folds$ signature category $\cL$, the canonical weak factorization system $(\mono,\epi)$ on $\Set$ whose left class consists of the monomorphisms and whose right class consists of the epimorphisms, induces a weak factorization system $(\mono[\cL],\epi[\cL])$ on $\Set^\cL$ as in \cref{thm:Reedy_wfs}. 

Moreover, since the weak factorization system $(\mono,\epi)$ in $\Set$ is cofibrantly generated by the unique map $!\colon\emptyset\to\ast$ we know by \cref{thm:Reedy_wfs} that the weak factorization system $(\mono[\cL],\epi[\cL])$ is cofibrantly generated by $\cB\, \widehat{\ast}\, (\{!\colon\emptyset\to\ast\})$, where in this case $\cB$ is the set of boundary incusions $\cB=\{\partial\cL_K\to\cL_K\colon \text{for all } K\in\cL\}$. Now, the map $!\colon\emptyset\to\ast$ is the unit of the tensor $-\widehat{\ast}-$ and thus the weak factorization system $(\mono[\cL],\epi[\cL])$ is simply generated by $\cB$.

\begin{lem}\label{lem:Reedy_mono_epi} Let $\cL$ be a $\folds$ signature category, and let us consider the weak factorization $(\mono,\epi)$ on $\Set$. The Reedy weak factorization $(\mono[\cL],\epi[\cL])$ on $\Set^\cL$ has as left class monomorphisms, and as right class fiberwise surjective natural transformations. 
\end{lem}

\begin{proof}
That the right class consists of the fiberwise surjective natural transformations is clear from the definitions of Reedy epimorphism and fiberwise surjection. To show that in this context Reedy monomorphisms are exactly monomorphisms, we recall that a map $\alpha\colon X\to Y$ in $\Set^\cL$ is a Reedy monomorphism if and only if for every $K$ in $\cL$ the map $$\widehat{\alpha}_K\colon X_K\bigsqcup_{L_KX}L_K Y\to Y_K$$ is a monomorphism too. Now, since $\cL$ is an inverse category, $L_K=\emptyset$ for every $K$ in $\cL$. Consequently, the condition of $\widehat{\alpha}_K$ being a monomorphism reduces to $\widehat{\alpha}_K\colon X_K\to Y_K$ being a monomorphism.
\end{proof}

\begin{cor}\label{rem:Yoneda_Reedy_cofibrant}
The Yoneda embedding $\yo\colon \cL^\op\to \Set^\cL$ is Reedy monomorphic.
\end{cor}

\begin{proof}
In fact, the latching maps of this functor are exactly the generating monomorphisms $\cB=\{\partial\cL_K\to\cL_K\colon \text{for all } K\in\cL\}$.
\end{proof}

\begin{rem}\label{latching_map_realization}
Consider now the construction in \cref{nerve_construction}, and let's call the left adjoint $|-|=\Lan_\yo D$\textemdash which we will mostly use in the form of $|-|=-\ast_{\cL^\op}D$. If we right lift the weak factorization system $(\mono[\cL], \epi[\cL])$ through the right adjoint $N=\hom(D-,-)$, we obtain a weak factorization system in $\cM$ cofibrantly generated by 
\begin{tz}
\node[](1)   {$|\partial\cL_K|$};
\node[below of=1](2)    {$|\cL_K|$};

\node[right of=1](1')     {$L_K D$};
\node[right of=2](2')     {$D(K)$};

\draw[->] (1) to (2);
\draw[->] (1') to (2');

\node[la] at ($(1)!0.5!(1')$) {$\cong$};
\node[la] at ($(2)!0.5!(2')$) {$\cong$};
\end{tz}
for all $K$ in $\cL$.
\end{rem}

\subsection{Interaction with homotopy theory}

In what follows we prove an abstraction of \cite[Lemma 11.2.24]{elements}; the main idea appears in \cite{HenryFolds}. For this, we focus on a scenario where the nerve-realization construction of \cref{nerve_construction} is homotopically well-behaved with respect to the weak factorization system on $\Set^\cL$ presented in \cref{subsec:mono_epi} and some model structure on the category $\cM$. We describe this precisely below, and we call it Condition $\bigtriangleup$ going forward. 

\subsection*{Condition $\bigtriangleup$} Let $\cL$ be a $\folds$ signature category, and let us consider a diagram $D\colon\cL^\op\to\cM$ where $\cM$ is a model category such that in the diagram below (see \cref{nerve_construction})
\begin{tz}
\node[](1)  {$\cL^\op$};
\node[below of=1,xshift=-1.5cm,yshift=-0.60cm](2) {$\cM$};
\node[below of=1,xshift=1.5cm,yshift=-0.60cm](3) {$\Set^{\cL}$};

\draw[->] (1) to node[left,la]{$D$} (2);
\draw[->] (1) to node[right,la]{$\yo$} (3);

\draw[->, bend right=20] ($(2.east)-(-4pt,5pt)$) to node[below,la]{$N$} ($(3.west)-(0,5pt)$);
\draw[->, bend right=20] ($(3.west)+(0,5pt)$) to node[above,la]{$-\ast_\cC D$} ($(2.east)+(4pt,5pt)$);

\node[la] at ($(2.east)!0.5!(3.west)$) {$\bot$};
\end{tz}
 the weak factorization system $(\Cof, \trfib)$ on $\cM$ is right-lifted from the Reedy weak factorization system $(\mono[\cL],\epi[\cL])$ on $\Set^\cL$ of \cref{lem:Reedy_mono_epi}. 

A key component in the proof of our desired result is Ken Brown's factorization lemma \cite{KenBrown_factorization}, see also \cite[C.1.6]{elements}, which we recall below.

\begin{lem}[Ken Brown's factorization]\label{brown_factorization} Let $\cM$ be a model category. Any map $f\colon X\to Y$ between fibrant objects factorizes as a weak equivalence followed by a fibration as shown below.

\begin{tz}
  \node[](1) {$P_f$};
  \node[above of=1,xshift=-1.3cm](2) {$X$};
  \node[above of=1,xshift=1.3cm](3) {$Y$};

  \draw[->] (2) to node[above,la]{$f$} (3);
  \draw[->] (2) to node[left,la]{$w$} node[right,la]{$\sim$} (1);
  \draw[->>] (1) to node[right,la]{$p$} (3);
\end{tz}
Moreover, the weak equivalence $w$ is constructed as the section of a trivial fibration $q\colon P_f\fwto X$.
\end{lem}

\begin{proof}
    See the proof in \cite[C.1.6]{elements}, and also \cite[Lemma 1.1.12]{Hovey}.
\end{proof}

Whenever we are in the situation of diagram \eqref{nerve-construction-for-folds}, when we input the diagram $D\colon\cL^\op\to\cM$ to the hom bifunctor 
\begin{tz}
\node[](1) {$(\cM^{\cL^\op})^\op\times\cM$};
\node[right of=1,xshift=3.5cm](2)         {$\Set^{\cL}$};

\draw[->] (1) to node[above,la]{$\mathrm{hom}$} (2);
\end{tz}
the resulting functor $M\_\coloneqq \mathrm{hom}(D-,-)$ gives us a map between $\cL$-structures.

\begin{prop}\label{thm:weak_equiv_to_Lequiv}
Let $\cL$ be a $\folds$ signature category, and consider a diagram $D\colon\cL^\op\to\cM$ together with a model structure on $\cM$ satisfying Condition $\bigtriangleup$. 
Then 
\begin{enumerate}
    \item weakly equivalent fibrant objects induce $\cL$-equivalent structures,
    \item if a weak equivalence between fibrant objects preserves a context, then the $\cL$-structures are equivalent in that context.
\end{enumerate}
\end{prop}
\begin{proof}
It suffices to consider $f\colon X\to Y$ a weak equivalence between fibrant objects in $\cM$. By Ken Brown's factorization and the 2-out-if-3 property for weak equivalences, the map $f\colon X\to Y$ admits a factorization

\begin{tz}
  \node[](1) {$P_f$};
  \node[above of=1,xshift=-1.3cm](2) {$X$};
  \node[above of=1,xshift=1.3cm](3) {$Y$};

  \draw[->] (2) to node[above,la]{$f$} node[below,la]{$\sim$} (3);
  \draw[->>] (1) to node[left,la]{$q$} node[right,la]{$\sim$} (2);
  \draw[->>] (1) to node[right,la]{$p$} node[left,la]{$\sim$} (3);
  \draw[dashed,->,bend right=45] (2) to node[left,la]{$w$} (1);
\end{tz}
where $p$ is a trivial fibration and $w$ is a section of a trivial fibration.

We know that the class of $\trfib$ of trivial fibrations in $\cM$ are right lifted from the Reedy weak factorization system $(\mono[\cL],\epi[\cL])$ on $\Set^\cL$\textemdash that is, $\trfib=(M\_ )^{-1}(\epi[\cL])$\textemdash and thus the span of trivial fibrations above is sent via $M\_$ to a span of fiberwise surjective natural transformations between $\cL$-structures as below.

\begin{tz}
  \node[](1) {$M_{P_f}$};
  \node[above of=1,xshift=-1.3cm](2) {$M_X$};
  \node[above of=1,xshift=1.3cm](3) {$M_Y$};

  \draw[->] (1) to  (2);
  \draw[->] (1) to  (3);
\end{tz}

To say that $f\colon X\to Y$ preserves a context $\Gamma$ means that, given interpretations $\alpha\colon\Gamma\to M_X$ and $\beta\colon \Gamma\to M_Y$, the diagram below commutes.
\begin{tz}
\node[](1) {$\Gamma$};
\node[below of=1,xshift=-1.5cm](2) {$M_X$};
\node[below of=1,xshift=1.5cm](3) {$M_Y$};

\draw[->] (1) to node[left,la]{$\alpha$} (2);
\draw[->] (1) to node[right,la]{$\beta$} (3);
\draw[->] (2) to node[below,la]{$M_f$} (3);
\end{tz}

Define an interpretation $\gamma\colon\Gamma\to M_{P_f}$ as the composite $\gamma=M_w\circ\alpha$.  Now the diagram below exhibits $M_X$ and $M_Y$ as $\cL$-equivalent in context $\Gamma$.

\begin{tz}
\node[](1) {$\Gamma$};
\node[below of=1,xshift=-1.5cm](2) {$M_X$};
\node[below of=1,xshift=1.5cm](3) {$M_Y$};
\node[below of=3,xshift=-1.5cm](4) {$M_{P_f}$};

\draw[->] (1) to node[left,la]{$\alpha$} (2);
\draw[->] (1) to node[right,la]{$\beta$} (3);
\draw[->] (1) to node[left,la]{$\gamma$} (4);
\draw[->] (4) to node[left,la]{$M_q$} node[right,la]{$\sim$} (2);
\draw[->] (4) to node[right,la]{$M_p$} node[left,la]{$\sim$} (3);
\end{tz}
Now the statement follows from \cref{thm:same_formulae}.
\end{proof}

\begin{prop}\label{under_triangle_D_Reedy_cofibrant} Let $\cL$ be a $\folds$ signature category, and consider a diagram $D\colon\cL^\op\to\cM$ together with a model structure on $\cM$. If $\cM$ is such that in the nerve-realization construction, condition $\bigtriangleup$ holds, then the diagram $D\colon \cL^\op\to \cM$ is Reedy cofibrant.  Conversely, when the latching maps of $D$ form a set of generating cofibrations for the weak factorization system $(\Cof,\trfib)$ in $\cM$, then condition $\bigtriangleup$ holds.
\end{prop}

\begin{proof}
This follows from \cref{latching_map_realization}, as we know that the Reedy weak factorization system on $\Set^\cL$ is cofibrantly generated by the boundary inclusions.
\end{proof}

Under the same situation as above, we can now deduce the result we were looking for. Colloquially, this theorem says that if $f\colon X\to Y$ is weak equivalence between fibrant objects, then the $\cL$-structures $M_{X}$ and $M_Y$ satisfy the same formulae.

\begin{theorem}\label{thm:we_then_structures_verify_same_formulate} Let $\cM$ be a model structure, and $f\colon X\to Y$ a weak equivalence between fibrant objects. Consider also a context $\Gamma$, together with two interpretations $\alpha\colon \Gamma\to M_X$ and $\beta\colon \Gamma\to M_Y$ that are compatible with the map $M_f\colon M_X\to M_Y$. Then $M_{X}\vDash \phi[\alpha]$ if and only if $M_Y\vDash\phi[\beta]$. 
\end{theorem}

\begin{proof} By \cref{thm:weak_equiv_to_Lequiv} we know that $f$ is sent to an equivalence of $\cL$-structures, which in turn implies that $M_X$ and $M_Y$ satisfy the same formulae by \cref{thm:same_formulae}.
\end{proof}

\subsection*{The case of $\cat$} 
The diagram $D_\cat$ that we presented in \cref{def:Dcat} is Reedy cofibrant. Moreover, the latching maps of $D$ associated to the kinds $O,A,$ and $\dot{E}$ recover the generating usual set of cofibrations of the folk model structure on $\cat$ (see for example \cite[Theorem 3.4]{Joyal_Tierney}).

\subsection*{The case of $\twocat$} As pointed out in \cite[Section 11.2]{elements}, the diagram $E_{\twocat}$ that we presented in \cref{def:E2cat} is not Reedy cofibrant. Indeed, when we compute the latching map of $E_{\twocat}$ corresponding to the kind $T$ we obtain the map 

\begin{tz}[node distance=0.8cm]
    \node[](1) {$\boldsymbol{\cdot}$}; 
    \node[right of=1,xshift=0.8cm](3) {$\boldsymbol{\cdot}$};
    \node[above of=1,xshift=0.8cm](2) {$\boldsymbol{\cdot}$};
    \draw[->] (1) to (2);
    \draw[->] (2) to (3);
    \draw[->] (1) to (3);
    \node[right of=1,yshift=0.3cm](4) {$\neq$}; 
    \node[right of=4,xshift=1cm](5) {$\mathbb{3}$};
    \node at ($(4)!0.6!(5)$) {$\longrightarrow$}; 
    \end{tz}
which is not a cofibration, as we won't be able to impose the commutativity condition in $1$-cells. One can check that the latching maps associated to $I_1$ and $\dot{H}$ are also not cofibrations. We consider a new functor $D_{\twocat}\colon\cL_{\twocat}^\op\to\twocat$ that modifies the image of the kinds $T,I_1,\dot{E}$ to take into consideration the restrictions on where we can express equality:

\[
\begin{tikzcd}
 & \Sigma[\mathbb{3}]  \arrow[from=d, shift right=.6em] \arrow[from=d] \arrow[from=d, shift left=.6em] &  \Sigma[\mathbb{2}] \arrow[from=dl, shift right=.3em] \arrow[from=dl, shift left=.3em] \\  \mathbb{H}_{\cong} \arrow[from=r, shift right=.6em] \arrow[from=r] \arrow[from=r, shift left=.6em] \arrow[from=d, shift left=.3em] \arrow[from=d, shift right=.3em]& \Sigma[\mathbb{2}] \arrow[from=d, shift left=.3em] \arrow[from=d, shift right=.3em]& \mathbb{2} \arrow[from=l]\\  \mathbb{T}  \arrow[from=r, shift right=.6em] \arrow[from=r] \arrow[from=r, shift left=.6em] & \mathbb{2} \arrow[from=d, shift left=.3em] \arrow[from=d, shift right=.3em] & \mathbb{A} \arrow[from=l] \\
  & \mathbb{1}
\end{tikzcd} \qquad
\]
where
\[
\mathbb{T} \coloneqq 
\begin{tikzcd}[sep=small] 
& \bullet \arrow[dr] \arrow[d, phantom, "\scriptstyle\cong"] \\
\bullet \arrow[ur] \arrow[rr] & ~ & \bullet
\end{tikzcd}, \quad 
\mathbb{A} \coloneqq 
\begin{tikzcd}[sep=small] \bullet \arrow[out=45, in=-45, loop,  "\cong"'] 
\end{tikzcd}, \quad 
\Sigma\mathbb{2} \coloneqq 
\begin{tikzcd} \bullet \arrow[r, bend left] \arrow[r, bend right] \arrow[r, phantom, "\scriptstyle\Downarrow"] & \bullet 
\end{tikzcd}, \]

\[ 
\Sigma[\mathbb{3}] \coloneqq 
\begin{tikzcd}\bullet \arrow[r] \arrow[r, bend left=50, "\Downarrow"'] \arrow[r, bend right=50, "\Downarrow"] & \bullet
\end{tikzcd}, \quad \text{and} \quad\mathbb{H}_{\cong} \coloneq
\begin{tikzcd} 
\bullet \arrow[r, bend left] \arrow[rr, bend left=50,  "\cong"'] \arrow[rr, bend right=50, "\cong"] \arrow[r, bend right] \arrow[r, phantom, "\scriptstyle\Downarrow"] & \bullet \arrow[r, bend left] \arrow[r, bend right] \arrow[r, phantom, "\scriptstyle\Downarrow"] & \bullet 
\end{tikzcd}
\]
Now this diagram does satisfy the condition in $\bigtriangleup$.

\subsection{The case of equipments}

In this subsection, we prove a result of equivalence invariance in the framework of equipments. 

\begin{defn}\label{def:Ldblcat}
The $\folds$ signature for double categories is given by

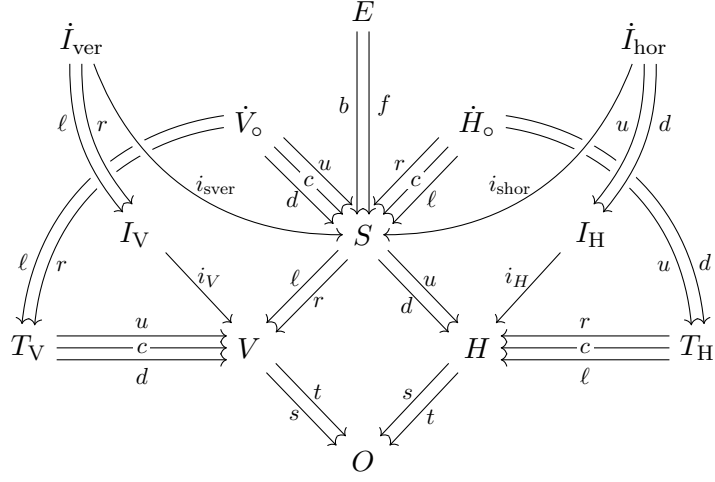
\begin{tz}

\node[](1) {$\dot{E}$};
\node[below of=1,yshift=-1.5cm](2) {$S$};
\node[below of=2,yshift=-1.5cm](3) {$O$};


\node[right of=1,xshift=2.2cm,yshift=-.4cm](A') {$\dot{I}_{\mathrm{hor}}$};
\node[right of=1,yshift=-1.5cm](B') {$\dot{H}_{\circ}$};

\node[right of=2,yshift=-1.5cm](E') {$H$};
\node[right of=E',yshift=1.5cm](C') {$I_\mathrm{H}$};
\node[right of=E',xshift=1.4cm](D') {$T_\mathrm{H}$};

\node[left of=1,xshift=-2.2cm,yshift=-.4cm](A) {$\dot{I}_{\mathrm{ver}}$};
\node[left of=1,yshift=-1.5cm](B) {$\dot{V}_{\circ}$};

\node[left of=2,yshift=-1.5cm](E) {$V$};
\node[left of=E,yshift=1.5cm](C) {$I_\mathrm{V}$};
\node[left of=E,xshift=-1.4cm](D) {$T_\mathrm{V}$};

\draw[->] ($(1.south)-(2.25pt,0)$) to node[left,la,pos=0.4]{$b$} ($(2.north)-(2.25pt,0)$);
\draw[->] ($(1.south)+(2.25pt,0)$) to node[right,la,pos=0.4]{$f$} ($(2.north)+(2.25pt,0)$);

\draw[->]  ($(B'.south west)+(-3.5pt,3.5pt)$) to node[above,la,pos=0.6]{$r$} ($(2.north east)+(-3.5pt,3.5pt)$);
\draw[->]  ($(B'.south west)$) to node[over,la]{$c$} ($(2.north east)$);
\draw[->]  ($(B'.south west)+(3.5pt,-3.5pt)$) to node[below,la,pos=0.4]{$\ell$} ($(2.north east)+(3.5pt,-3.5pt)$);

\draw[->]  ($(B.south east)+(3.5pt,3.5pt)$) to node[above,la,pos=0.6]{$u$} ($(2.north west)+(3.5pt,3.5pt)$);
\draw[->]  ($(B.south east)$) to node[over,la]{$c$} ($(2.north west)$);
\draw[->]  ($(B.south east)-(3.5pt,3.5pt)$) to node[below,la,pos=0.4]{$d$} ($(2.north west)-(3.5pt,3.5pt)$);

\draw[->]  ($(D'.west)+(0,4.5pt)$) to node[above,la]{$r$} ($(E'.east)+(0,4.5pt)$);
\draw[->]  ($(D'.west)$) to node[over,la]{$c$} ($(E'.east)$);
\draw[->]  ($(D'.west)-(0,4.5pt)$) to node[below,la]{$\ell$} ($(E'.east)-(0,4.5pt)$);

\draw[->]  ($(D.east)+(0,4.5pt)$) to node[above,la]{$u$} ($(E.west)+(0,4.5pt)$);
\draw[->]  ($(D.east)$) to node[over,la]{$c$} ($(E.west)$);
\draw[->]  ($(D.east)-(0,4.5pt)$) to node[below,la]{$d$} ($(E.west)-(0,4.5pt)$);

\draw[->]  ($(2.south east)+(1.75pt,1.75pt)$) to node[above,la,pos=0.6]{$u$} ($(E'.north west)+(1.75pt,1.75pt)$);
\draw[->]  ($(2.south east)-(1.75pt,1.75pt)$) to node[below,la,pos=0.4]{$d$} ($(E'.north west)-(1.75pt,1.75pt)$);

\draw[->]  ($(2.south west)+(-1.75pt,1.75pt)$) to node[above,la,pos=0.6]{$\ell$} ($(E.north east)+(-1.75pt,1.75pt)$);
\draw[->]  ($(2.south west)+(1.75pt,-1.75pt)$) to node[below,la,pos=0.4]{$r$} ($(E.north east)+(1.75pt,-1.75pt)$);

\draw[->]  ($(E'.south west)+(-1.75pt,1.75pt)$) to node[above,la,pos=0.6]{$s$} ($(3.north east)+(-1.75pt,1.75pt)$);
\draw[->]  ($(E'.south west)+(1.75pt,-1.75pt)$) to node[below,la,pos=0.4]{$t$} ($(3.north east)+(1.75pt,-1.75pt)$);

\draw[->]  ($(E.south east)-(1.75pt,1.75pt)$) to node[below,la,pos=0.4]{$s$} ($(3.north west)-(1.75pt,1.75pt)$);
\draw[->]  ($(E.south east)+(1.75pt,1.75pt)$) to node[above,la,pos=0.6]{$t$} ($(3.north west)+(1.75pt,1.75pt)$);

\draw[->]  ($(C'.south west)+(-1.75pt,1.75pt)$) to node[above,la,pos=0.6,xshift=-2pt]{$i_H$} ($(E'.north east)+(-1.75pt,1.75pt)$);

\draw[->]  ($(C.south east)+(1.75pt,1.75pt)$) to node[above,la,pos=0.6,xshift=2pt]{$i_V$} ($(E.north west)+(1.75pt,1.75pt)$);


\draw[->,bend left=38] ($(B'.east)-(0,2.25pt)$) to node[left,la,pos=0.8]{$u$} ($(D'.north)-(2.25pt,0)$);
\draw[->,bend left=38] ($(B'.east)+(0,2.25pt)$) to node[right,la,pos=0.803]{$d$} ($(D'.north)+(2.5pt,.5pt)$);

\draw[->,bend right=38] ($(B.west)-(0,2.25pt)$) to node[right,la,pos=0.8]{$r$} ($(D.north)+(2.25pt,0)$);
\draw[->,bend right=38] ($(B.west)+(0,2.25pt)$) to node[left,la,pos=0.803]{$\ell$} ($(D.north)+(-2.5pt,.5pt)$);

\draw[->,bend left=35,w] ($(A'.south)-(4.5pt,0)$) to node[above,la,pos=0.6,yshift=2pt]{$i_{\mathrm{shor}}$} ($(2.east)$);

\draw[->,bend left=20,w] ($(A'.south)$) to node[left,la,pos=0.4]{$u$} ($(C'.north)+(1pt,3pt)$);
\draw[->,bend left=20,w] ($(A'.south)+(4.5pt,0)$) to node[right,la,pos=0.36]{$d$} ($(C'.north)+(4.5pt,0)$);

\draw[->,bend right=35,w] ($(A.south)+(4.5pt,0)$) to node[above,la,pos=0.6,yshift=2pt]{$i_{\mathrm{sver}}$} ($(2.west)$);

\draw[->,bend right=20,w] ($(A.south)-(4.5pt,0)$) to node[left,la,pos=0.36]{$\ell$} ($(C.north)-(4.5pt,0)$);
\draw[->,bend right=20,w] ($(A.south)$) to node[right,la,pos=0.4]{$r$} ($(C.north)+(-1pt,3pt)$);
\end{tz}

For $\cL_\dblcat$ to encode the structure of double categories we must ask it to satisfy several relations.

It must verify the relations from the underlying vertical and horizontal ($1$-)categories: identities and composition of $1$-cells
\begin{align*}
    s \cdot i_\mathrm{V} &= t \cdot i_\mathrm{V}            &s\cdot u &=s\cdot c        &s \cdot i_\mathrm{H} &= t \cdot i_\mathrm{H}       &s\cdot \ell &=s\cdot c\\
                         &                                  &t\cdot d &=t\cdot c        &                     &                             &t\cdot r&=t\cdot c\\
                         &                                  &t\cdot u &=s\cdot c        &                     &                             &t\cdot \ell &=s\cdot r
\end{align*}

equations describing the shape of squares
\begin{align*}
    s\cdot u        &=s\cdot\ell    &s\cdot d       &=t\cdot\ell\\
    t\cdot u        &=s\cdot r      &t\cdot d       &=t\cdot r
\end{align*}
squares identities
\begin{align*}
    u\cdot u_\shor  &=i_\mathrm{H}\cdot u       &\ell\cdot i_\sver  &=i_\mathrm{V}\cdot\ell\\
    d\cdot i_\shor  &=i_\mathrm{H}\cdot d       &r\cdot i_\sver     &=i_\mathrm{V}\cdot r\\
    \ell\cdot i_\shor   &=r \cdot i_\shor       &r\cdot i_\sver     &=\ell\cdot i_\sver
\end{align*}


equality between two parallel squares
\begin{align*}
    u\cdot b    &=u\cdot f  &\ell\cdot b    &=\ell\cdot f\\
    d\cdot b    &=d\cdot f  &r\cdot b       &=r\cdot f
\end{align*}
also equations governing vertical composition of squares
\begin{align*}
u\cdot u    &=u\cdot c      &d\cdot d   &=d\cdot c      &d\cdot u   &=u\cdot d
\end{align*}
plus its relation with composition of vertical morphisms
\begin{align*}
    \ell\cdot u         &=u\cdot\ell    &\ell\cdot d    &=d\cdot\ell    &\ell\cdot c    &=c\cdot\ell\\
    r\cdot u            &=u\cdot r      &r\cdot d       &=d\cdot r      &r\cdot c       &=c\cdot r
\end{align*}
Similarly, equations governing horizontal composition of squares
\begin{align*}
    \ell\cdot c       &=\ell\cdot\ell       &r\cdot c   &=r\cdot r      &\ell\cdot t        &=\ell\cdot r
\end{align*}
plus its relation with composition of horizontal morphisms
\begin{align*}
    u\cdot c    &=c\cdot u   &u\cdot \ell    &=\ell\cdot u   &d\cdot r   &=r\cdot d\\
    d\cdot c    &=c\cdot d   &u\cdot r          &=r\cdot u  &d\cdot\ell     &=\ell\cdot d
\end{align*}

\end{defn}

\begin{rem}
In describing the relations needed in $\cL_\dblcat$ to describe the structure of double categories, we have spelled out the relations governing the structure of the horizontal and vertical $2$-categories. This follows by observing \cref{def:L2cat} and the relationships written in \cref{def:Ldblcat}
\end{rem}

\begin{defn}\label{Ddblcat}Consider the diagram $D_\dblcat\colon \cL_\dblcat^\op\to \dblcat$

\begin{tz}

\node[](1) {$\mathbb{S}$};
\node[below of=1,yshift=-1.5cm](2) {$\mathbb{S}$};
\node[below of=2,yshift=-1.5cm](3) {$\mathbbm{1}$};


\node[right of=1,xshift=2.2cm,yshift=-.4cm](A') {$\bV\mathbbm{2}$};
\node[right of=1,yshift=-1.5cm](B') {$C_{\bH}$};

\node[right of=2,yshift=-1.5cm](E') {$\bH\mathbbm{2}$};
\node[right of=E',yshift=1.5cm](C') {$\bH A$};
\node[right of=E',xshift=1.4cm](D') {$\bH T$};

\node[left of=1,xshift=-2.2cm,yshift=-.4cm](A) {$\bH\mathbbm{2}$};
\node[left of=1,yshift=-1.5cm](B) {$C_{\bV}$};

\node[left of=2,yshift=-1.5cm](E) {$\bV\mathbbm{2}$};
\node[left of=E,yshift=1.5cm](C) {$\bV A$};
\node[left of=E,xshift=-1.4cm](D) {$\bV T$};

\draw[->] ($(2.north)-(2.25pt,0)$) to node[left,la,pos=0.6]{$b$} ($(1.south)-(2.25pt,0)$);
\draw[->] ($(2.north)+(2.25pt,0)$) to node[right,la,pos=0.6]{$f$} ($(1.south)+(2.25pt,0)$);

\draw[->]  ($(2.north east)+(-3.5pt,3.5pt)$) to node[above,la,pos=0.4]{$r$} ($(B'.south west)+(-3.5pt,3.5pt)$);
\draw[->]  ($(2.north east)$) to node[over,la]{$c$} ($(B'.south west)$);
\draw[->]  ($(2.north east)+(3.5pt,-3.5pt)$) to node[below,la,pos=0.6]{$\ell$} ($(B'.south west)+(3.5pt,-3.5pt)$);

\draw[->]  ($(2.north west)+(3.5pt,3.5pt)$) to node[above,la,pos=0.4]{$u$} ($(B.south east)+(3.5pt,3.5pt)$);
\draw[->]  ($(2.north west)$) to node[over,la]{$c$} ($(B.south east)$);
\draw[->]  ($(2.north west)-(3.5pt,3.5pt)$) to node[below,la,pos=0.6]{$d$} ($(B.south east)-(3.5pt,3.5pt)$);

\draw[->]  ($(E'.east)+(0,4.5pt)$) to node[above,la]{$r$} ($(D'.west)+(0,4.5pt)$);
\draw[->]  ($(E'.east)$) to node[over,la]{$c$} ($(D'.west)$);
\draw[->]  ($(E'.east)-(0,4.5pt)$) to node[below,la]{$\ell$} ($(D'.west)-(0,4.5pt)$);

\draw[->]  ($(E.west)+(0,4.5pt)$) to node[above,la]{$u$} ($(D.east)+(0,4.5pt)$);
\draw[->]  ($(E.west)$) to node[over,la]{$c$} ($(D.east)$);
\draw[->]  ($(E.west)-(0,4.5pt)$) to node[below,la]{$d$} ($(D.east)-(0,4.5pt)$);

\draw[->]  ($(E'.north west)+(1.75pt,1.75pt)$) to node[above,la,pos=0.4]{$u$} ($(2.south east)+(1.75pt,1.75pt)$);
\draw[->]  ($(E'.north west)-(1.75pt,1.75pt)$) to node[below,la,pos=0.6]{$d$} ($(2.south east)-(1.75pt,1.75pt)$);

\draw[->]  ($(E.north east)+(-1.75pt,1.75pt)$) to node[above,la,pos=0.4]{$\ell$} ($(2.south west)+(-1.75pt,1.75pt)$);
\draw[->]  ($(E.north east)+(1.75pt,-1.75pt)$) to node[below,la,pos=0.6]{$r$} ($(2.south west)+(1.75pt,-1.75pt)$);

\draw[->]  ($(3.north east)+(-1.75pt,1.75pt)$) to node[above,la,pos=0.4]{$s$} ($(E'.south west)+(-1.75pt,1.75pt)$);
\draw[->]  ($(3.north east)+(1.75pt,-1.75pt)$) to node[below,la,pos=0.6]{$t$} ($(E'.south west)+(1.75pt,-1.75pt)$);

\draw[->]  ($(3.north west)-(1.75pt,1.75pt)$) to node[below,la,pos=0.6]{$s$} ($(E.south east)-(1.75pt,1.75pt)$);
\draw[->]  ($(3.north west)+(1.75pt,1.75pt)$) to node[above,la,pos=0.4]{$t$} ($(E.south east)+(1.75pt,1.75pt)$);

\draw[->]  ($(E'.north east)+(-1.75pt,1.75pt)$) to node[above,la,pos=0.4,xshift=-2pt]{$i_H$} ($(C'.south west)+(-1.75pt,1.75pt)$);

\draw[->]  ($(E.north west)+(1.75pt,1.75pt)$) to node[above,la,pos=0.4,xshift=2pt]{$i_V$} ($(C.south east)+(1.75pt,1.75pt)$);


\draw[->,bend right=38] ($(D'.north)-(2.25pt,0)$) to node[left,la,pos=0.2]{$u$} ($(B'.east)-(0,2.25pt)$);
\draw[->,bend right=38] ($(D'.north)+(2.5pt,.5pt)$) to node[right,la,pos=0.197]{$d$} ($(B'.east)+(0,2.25pt)$);

\draw[->,bend left=38] ($(D.north)+(2.25pt,0)$) to node[right,la,pos=0.2]{$r$} ($(B.west)-(0,2.25pt)$);
\draw[->,bend left=38] ($(D.north)+(-2.5pt,.5pt)$) to node[left,la,pos=0.197]{$\ell$} ($(B.west)+(0,2.25pt)$);

\draw[->,bend right=35,w] ($(2.east)$) to node[above,la,pos=0.4,yshift=2pt]{$i_{\mathrm{shor}}$} ($(A'.south)-(4.5pt,0)$);
\draw[->,bend right=20,w] ($(C'.north)+(1pt,3pt)$) to node[left,la,pos=0.6]{$u$} ($(A'.south)$);
\draw[->,bend right=20,w] ($(C'.north)+(4.5pt,0)$) to node[right,la,pos=0.64]{$d$} ($(A'.south)+(4.5pt,0)$);

\draw[->,bend left=35,w] ($(2.west)$) to node[above,la,pos=0.4,yshift=2pt]{$i_{\mathrm{sver}}$} ($(A.south)+(4.5pt,0)$);
\draw[->,bend left=20,w] ($(C.north)-(4.5pt,0)$) to node[left,la,pos=0.64]{$\ell$} ($(A.south)-(4.5pt,0)$);
\draw[->,bend left=20,w] ($(C.north)+(-1pt,3pt)$) to node[right,la,pos=0.6]{$r$} ($(A.south)$);
\end{tz}
where $T$ and $A$ as the $2$-categories generated by the diagrams below

\[
T \coloneqq 
\begin{tikzcd}[sep=small] 
& \bullet \arrow[dr] \arrow[d, phantom, "\scriptstyle\cong"] \\
\bullet \arrow[ur] \arrow[rr] & ~ & \bullet
\end{tikzcd}, \quad 
A \coloneqq 
\begin{tikzcd}[sep=small] \bullet \arrow[out=45, in=-45, loop,  "\cong"'] 
\end{tikzcd},
 \]
the double category $\mathbb{S}$ is the double category generated by a square $\mathbb{S} = 
\begin{tikzcd} 
\bH\mathbbm{2}\times\bV\mathbbm{2}
\end{tikzcd},$ and 

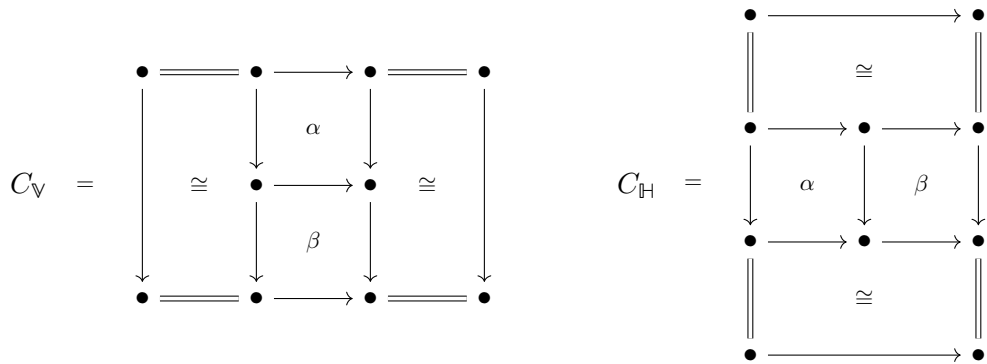
\begin{tz}
\node[](1)      {$\bullet$};
\node[right of=1](2)        {$\bullet$};
\node[right of=2](3)        {$\bullet$};
\node[right of=3](4)        {$\bullet$};

\node[below of=2](5)        {$\bullet$};
\node[below of=3](6)        {$\bullet$};

\node[below of=5](8)        {$\bullet$};
\node[left of=8](7)         {$\bullet$};
\node[below of=6](9)        {$\bullet$};
\node[right of=9](10)       {$\bullet$};

\draw[d] (1) to (2);
\draw[->] (2) to (3);
\draw[d] (3) to (4);

\draw[->] (5) to (6);

\draw[d] (7) to (8);
\draw[->] (8) to (9);
\draw[d] (9) to (10);

\draw[->] (1) to (7);
\draw[->] (2) to (5);
\draw[->] (3) to (6);
\draw[->] (4) to (10);
\draw[->] (5) to (8);
\draw[->] (6) to (9);

\node[left of=5,xshift=-1.5cm](N)      {$C_\bV$}; 

\node[la] at ($(N)!0.5!(5)-(0.75cm,0)$) {$=$};

\node[la] at ($(1)!0.5!(8)$) {$\cong$};
\node[la] at ($(2)!0.5!(6)$) {$\alpha$};
\node[la] at ($(5)!0.5!(9)$) {$\beta$};
\node[la] at ($(3)!0.5!(10)$) {$\cong$};


\node[right of=4,xshift=2cm,yshift=0.75cm](1)      {$\bullet$};
\node[below of=1](4)    {$\bullet$};
\node[right of=4](5)    {$\bullet$};
\node[right of=5](6)    {$\bullet$};

\node[above of=6](3)    {$\bullet$};

\node[below of=4](7)    {$\bullet$};
\node[below of=5](8)    {$\bullet$};
\node[below of=6](9)    {$\bullet$};

\node[below of=7](10)   {$\bullet$};
\node[below of=9](12)   {$\bullet$};

\draw[->] (1) to (3);
\draw[->] (4) to (5);
\draw[->] (5) to (6);
\draw[->] (7) to (8);
\draw[->] (8) to (9);
\draw[->] (10) to (12);

\draw[d] (1) to (4);
\draw[->] (4) to (7);
\draw[d] (7) to (10);

\draw[->] (5) to (8);

\draw[d] (3) to (6);
\draw[->] (6) to (9);
\draw[d] (9) to (12);

\node[left of=5,xshift=-1.5cm,yshift=-0.75cm](N)      {$C_\bH$}; 

\node[la] at ($(N)!0.5!(5)-(0.75cm,0.35cm)$) {$=$};

\node[la] at ($(1)!0.5!(6)$) {$\cong$};
\node[la] at ($(4)!0.5!(8)$) {$\alpha$};
\node[la] at ($(5)!0.5!(9)$) {$\beta$};
\node[la] at ($(7)!0.5!(12)$) {$\cong$};
\end{tz}

\end{defn}

We can now consider the nerve-realization construction for the case of double categories: we can build the diagram below.

\begin{tz}
\node[](1)  {$\cL_{\dblcat}^\op$};
\node[below of=1,xshift=-1.5cm,yshift=-0.60cm](2) {$\dblcat$};
\node[below of=1,xshift=1.5cm,yshift=-0.60cm](3) {$\Set^{\cL_{\dblcat}}$};

\draw[->] (1) to node[left,la]{$D_\dblcat$} (2);
\draw[->] (1) to node[right,la]{$\yo$} (3);

\draw[->, bend right=20] ($(2.east)-(-4pt,5pt)$) to node[below,la]{$\hom(D_{\dblcat}-,-)$} ($(3.west)-(0,5pt)$);
\draw[->, bend right=20] ($(3.west)+(0,5pt)$) to node[above,la]{$-\ast_\cC D_\dblcat$} ($(2.east)+(4pt,5pt)$);
\node[la] at ($(2.east)!0.5!(3.west)$) {$\bot$};
\end{tz}
with the model structure on $\dblcat$ being that of \cref{MS_equipments}, and the diagram $D_\dblcat$ the one defined in \cref{Ddblcat}. 

As before, mapping out of $D_\dblcat$ induces a functor $M_{\_}\colon\dblcat\to\Set^{\cL_{\dblcat}}$ whose image lies in the full subcategory of $\cL$-structures\textemdash this is true, as the latching maps for the diagram $D_\dblcat$ at the relation objects $\dot{I}_\ver, \dot{V}_\circ, \dot{E},\dot{H}_\circ,$ and $\dot{I}_\hor$ are epimorphisms.

\begin{lem}\label{lem:surviving_latching_maps} The latching maps of the diagram $D_{\dblcat}$ generate the weak factorization system $(\Cof,\Fib\cap\cW)$ on $\dblcat$ with the model structure of \cref{MS_equipments}. Moreover, it is enough to consider the latching maps associated to the kinds $O,V,H,S$ and $\dot{E}$.
\end{lem} 

\begin{proof}
We begin by proving that the latching map for the kinds $O,H,V,S,$ and $\dot{E}$, recover the set $\cI$ of generating cofibrations of \cref{def:cofib}. 

In order to compute each latching map, we begin by the corresponding latching object.  By \cref{latching_map_realization} we first want to describe the boundary of the representable $\partial\cL_K$ for each kind $K$ of $\cL_\dblcat$ and proceed to compute the colimit of the diagram $D_\dblcat$ weighted by $\partial\cL_K$. Recall that the boundary $\partial\cL_K$ can be seen as the subfunctor of $\cL_K$ containing all elements but $\id_K$. In particular, for any kind $H$ with $\deg H\geq\deg K$ we know that $\partial\cL_K(H)=\emptyset$\textemdash this means that we will only look at the ``fan down'' from the given kind. 

For $K=O$, this computation is very simple as we have that $\partial\cL_O(O)=\emptyset$, and the corresponding latching map is 
$$\emptyset\longrightarrow \mathbbm{1}.$$

For $K=H$, when we look at $\partial\cL_H$ we get 

\begin{tz}
\node[](1)      {$\cL(H,H)^{\setminus\id_H}$};
\node[right of=1,xshift=-0.1cm](2)       {$=\emptyset$};
\node[below of=1,xshift=-1.5cm,yshift=-0.5cm](3)         {$\cL(H,O)$};
\node[right of=3,xshift=-0.1cm](4)       {$=\{s,t\}$};

\draw[->] ($(1.south)-(14pt,0)$) to node[above,la,xshift=-2pt,pos=0.5]{$s$} ($(3.north)-(4pt,0)$);
\draw[->] ($(1.south)-(7pt,0)$) to node[below,la,pos=0.4]{$t$} ($(3.north)+(3pt,0)$);
\end{tz}
and then the corresponding latching map is the inclusion

\begin{tz}
\node[](1)      {$\mathbbm{1}\sqcup\mathbbm{1}$};
\node[right of=1](2)        {$\bH\mathbbm{2}$};

\draw[->] (1) to  (2);
\end{tz}    
Similarly, for $K=V$ we obtain the map $\mathbbm{1}\sqcup\mathbbm{1}\rightarrow\bV\mathbbm{2}$. 

For $K=S$, the weight $\partial\cL_S$ is illustrated in the following diagram 

\begin{tz}

\node[](2) {$\cL(S,S)^{\setminus\id_S}$};
\node[below of=2,yshift=-2cm](3) {$\cL(S,O)$};

\node[right of=2,xshift=1cm,yshift=-1.75cm](E') {$\cL(S,H)$};

\node[left of=2,xshift=-1cm,yshift=-1.75cm](E) {$\cL(S,V)$};

\node[right of=2,xshift=-0.1cm](2')     {$=\emptyset$};
\node[right of=E'](E'eq)   {$=\{u,d\}$};

\node[left of=E,xshift=0.1cm](Eeq)   {$\{\ell,r\}=$};

\node[below of=3,yshift=0.3cm](3eq) {$\left\{
\begin{aligned}
    &sl=su,\, sr=tu
    \\
    &tl=sd,\, tr=td
\end{aligned}
\right\}$};

\node[la] at ($(3)!0.5!(3eq)+(0,2pt)$) {$\verteq$};

\draw[->]  ($(2.south east)+(1.75pt,1.75pt)$) to node[above,la,pos=0.6]{$u$} ($(E'.north west)+(1.75pt,1.75pt)$);
\draw[->]  ($(2.south east)-(1.75pt,1.75pt)$) to node[below,la,pos=0.4]{$d$} ($(E'.north west)-(1.75pt,1.75pt)$);

\draw[->]  ($(2.south west)+(-1.75pt,1.75pt)$) to node[above,la,pos=0.6]{$\ell$} ($(E.north east)+(-1.75pt,1.75pt)$);
\draw[->]  ($(2.south west)+(1.75pt,-1.75pt)$) to node[below,la,pos=0.4]{$r$} ($(E.north east)+(1.75pt,-1.75pt)$);

\draw[->]  ($(E'.south west)+(-1.75pt,1.75pt)$) to node[above,la,pos=0.6]{$s$} ($(3.north east)+(-1.75pt,1.75pt)$);
\draw[->]  ($(E'.south west)+(1.75pt,-1.75pt)$) to node[below,la,pos=0.4]{$t$} ($(3.north east)+(1.75pt,-1.75pt)$);

\draw[->]  ($(E.south east)-(1.75pt,1.75pt)$) to node[below,la,pos=0.4]{$s$} ($(3.north west)-(1.75pt,1.75pt)$);
\draw[->]  ($(E.south east)+(1.75pt,1.75pt)$) to node[above,la,pos=0.6]{$t$} ($(3.north west)+(1.75pt,1.75pt)$);
\end{tz}
which gives rise to the diagram in $\dblcat$ below, whose colimit is the latching object of $D$ for $S$

\begin{tz}
\node[](1) {$\mathbbm{1}$};
\node[above of=1,yshift=-1cm](1')  {$s\ell=su$};

\node[right of=1,xshift=1cm](2){$\mathbbm{1}$};
\node[above of=2,yshift=-1cm](1')  {$sr=tu$};

\node[right of=2,xshift=1cm](3){$\mathbbm{1}$};
\node[above of=3,yshift=-1cm](3')  {$t\ell=sd$};

\node[right of=3,xshift=1cm](4){$\mathbbm{1}$};
\node[above of=4,yshift=-1cm](4')  {$tr=td$};

\node[left of=1,xshift=1cm,yshift=-2.6cm](6) {$\bV\mathbbm{2}$};
\node[below of=6,yshift=1cm](6')  {$r$};

\node[left of=6,xshift=1](5) {$\bV\mathbbm{2}$};
\node[below of=5,yshift=1cm](5')  {$\ell$};

\node[right of=4,xshift=-1cm,yshift=-2.6cm](7) {$\bH\mathbbm{2}$};
\node[below of=7,yshift=1cm](7')  {$u$};

\node[right of=7,xshift=1](8) {$\bH\mathbbm{2}$};
\node[below of=8,yshift=1cm](8')  {$d$};

\draw[->] (1) to node[la,left]{$s$} (5);
\draw[->] (1) to node[la, above,pos=0.1]{$s$} (7);

\draw[->,w] (2) to node[la,above,pos=0.1]{$s$} (6);
\draw[->,w] (2) to node[la,above,pos=0.1]{$t$} (7);

\draw[->,w] (3) to node[la,above,pos=0.1]{$t$} (5);
\draw[->,w] (3) to node[la,above,pos=0.1]{$s$} (8);

\draw[->] (4) to node[la,above]{$t$} (8);
\draw[->,w] (4) to node[la, above, pos=0.1]{$t$} (6);
\end{tz}

Then the corresponding latching map associated to $S$ is 

\begin{tz}
\node[](1)      {$\partial(\bH\mathbbm{2}\times\bV\mathbbm{2})$};
\node[right of=1,xshift=1.5cm](2)        {$\bH\mathbbm{2}\times\bV\mathbbm{2}$};
\draw[->] (1) to  (2);
\end{tz}    

For the kind $\dot{E}$, the weight $\partial\cL_{\dot{E}}$ is depicted in the diagram below.

\begin{tz}
\node[](1) {$\cL(\dot{E},\dot{E})^{\setminus\id_{\dot{E}}}$};
\node[below of=1,yshift=-1cm](2) {$\cL(\dot{E},S)$};
\node[below of=2,yshift=-2cm](3) {$\cL(\dot{E},O)$};

\node[right of=2,xshift=1cm,yshift=-1.75cm](E') {$\cL(\dot{E},H)$};

\node[left of=2,xshift=-1cm,yshift=-1.75cm](E) {$\cL(\dot{E},V)$};

\node[right of=1,xshift=-0.1cm](1')     {$=\emptyset$};
\node[right of=2,xshift=-0.1cm](2')     {$=\{b,f\}$};
\node[right of=E',xshift=.5cm](E'eq)   {$=\left\{
\begin{aligned}
    &uf=ub
    \\
    &df=db
\end{aligned}
\right\}$};

\node[left of=E,xshift=-1.8cm](Eeq)   {$\left\{
\begin{aligned}
    &suf=sub=s\ell f=s\ell b\\
    &sdf=sdb=t\ell b=t\ell b\\
    &tuf=tub=srf=srb\\
    &tdf=tdb=trb=trf
\end{aligned}
\right\}=$};

\node[below of=3,yshift=0.2cm](3eq) {$\left\{
\begin{aligned}
    &\ell b=\ell f
    \\
    &rb=rf
\end{aligned}
\right\}$};

\node[la] at ($(3)!0.5!(3eq)+(0,2pt)$) {$\verteq$};

\draw[->] ($(1.south)-(2.25pt,0)$) to node[left,la,pos=0.4]{$b$} ($(2.north)-(2.25pt,0)$);
\draw[->] ($(1.south)+(2.25pt,0)$) to node[right,la,pos=0.4]{$f$} ($(2.north)+(2.25pt,0)$);

\draw[->]  ($(2.south east)+(1.75pt,1.75pt)$) to node[above,la,pos=0.6]{$u$} ($(E'.north west)+(1.75pt,1.75pt)$);
\draw[->]  ($(2.south east)-(1.75pt,1.75pt)$) to node[below,la,pos=0.4]{$d$} ($(E'.north west)-(1.75pt,1.75pt)$);

\draw[->]  ($(2.south west)+(-1.75pt,1.75pt)$) to node[above,la,pos=0.6]{$\ell$} ($(E.north east)+(-1.75pt,1.75pt)$);
\draw[->]  ($(2.south west)+(1.75pt,-1.75pt)$) to node[below,la,pos=0.4]{$r$} ($(E.north east)+(1.75pt,-1.75pt)$);

\draw[->]  ($(E'.south west)+(-1.75pt,1.75pt)$) to node[above,la,pos=0.6]{$s$} ($(3.north east)+(-1.75pt,1.75pt)$);
\draw[->]  ($(E'.south west)+(1.75pt,-1.75pt)$) to node[below,la,pos=0.4]{$t$} ($(3.north east)+(1.75pt,-1.75pt)$);

\draw[->]  ($(E.south east)-(1.75pt,1.75pt)$) to node[below,la,pos=0.4]{$s$} ($(3.north west)-(1.75pt,1.75pt)$);
\draw[->]  ($(E.south east)+(1.75pt,1.75pt)$) to node[above,la,pos=0.6]{$t$} ($(3.north west)+(1.75pt,1.75pt)$);
\end{tz}

This gives us the diagram below in $\dblcat$, whose colimit is the latching object $\cL_{\dot{E}} D$. For the sake of readability, each element corresponding to $\cL(\dot{E},V)$ we will only write two of the terms in the equalities in the diagram above.

\begin{tz}
\node[](1) {$\mathbbm{1}$};
\node[above of=1,yshift=-1cm](1')  {$s\ell f=suf$};

\node[right of=1,xshift=1.5cm](2){$\mathbbm{1}$};
\node[above of=2,yshift=-1cm](2')  {$srb=tub$};

\node[right of=2,xshift=1.5cm](3){$\mathbbm{1}$};
\node[above of=3,yshift=-1cm](3')  {$t\ell b=sdb$};

\node[right of=3,xshift=1.5cm](4){$\mathbbm{1}$};
\node[above of=4,yshift=-1cm](4')  {$trf=tdf$};


\node[below of=1,xshift=0.7cm,yshift=-1.1cm](5) {$\bV\mathbbm{2}$};
\node[left of=5,xshift=0.4cm](5')  {$\ell b=\ell f$};

\node[right of=5,xshift=0.7cm](6) {$\bV\mathbbm{2}$};
\node[left of=6,yshift=0.3cm,xshift=0.7cm](6')  {$rb=rf$};

\node[below of=4,xshift=-0.7cm,yshift=-1.1cm](8) {$\bH\mathbbm{2}$};
\node[right of=8,xshift=-0.4cm](8')  {$df=db$};

\node[left of=8,xshift=-1cm](7) {$\bH\mathbbm{2}$};
\node[right of=7,yshift=0.3cm,xshift=-0.7cm](7')  {$uf=ub$};


\node[below of=5,xshift=-1.5cm,yshift=-1.4cm](9) {$\bH\mathbbm{2}\times\bV\mathbbm{2}$};
\node[below of=8,xshift=1.5cm,yshift=-1.4cm](10) {$\bH\mathbbm{2}\times\bV\mathbbm{2}$};

\draw[->] (1) to node[la,left]{$s$} (5);
\draw[->] (1) to node[la, above,pos=0.1]{$s$} (7);

\draw[->,w] (2) to node[la,left,pos=0.1]{$t$} (6);
\draw[->,w] (2) to node[la,right,pos=0.1]{$s$} (7);

\draw[->,w] (3) to node[la,above,pos=0.1]{$t$} (5);
\draw[->,w] (3) to node[la,above,pos=0.1]{$s$} (8);

\draw[->] (4) to node[la,right]{$t$} (8);
\draw[->,w] (4) to node[la, above, pos=0.1]{$t$} (6);

\draw[->] (5) to node[la,left]{$\ell$}  (9);
\draw[->] (6) to node[la,right,pos=0.8]{$r$} (9);

\draw[->] (7) to node[la,left]{$u$} (10);
\draw[->] (8) to node[la,right]{$d$} (10);

\draw[->,w] (5) to node[la,below,pos=0.8]{$\ell$} (10);
\draw[->,w] (6) to node[la,left,above,pos=0.7]{$r$} (10);

\draw[->,w] (7) to node[la,above,pos=0.7]{$r$} (9);
\draw[->,w] (8) to node[la,below,pos=0.8]{$\ell$} (9);

\end{tz}

This gives us 
\begin{tz}
\node[](1)      {$\bH\mathbbm{2}\times\bV\mathbbm{2} \sqcup_{\partial(\bH\mathbbm{2}\times\bV\mathbbm{2})}\bH\mathbbm{2}\times\bV\mathbbm{2}$};
\node[right of=1,xshift=3cm](2)        {$\bH\mathbbm{2}\times\bV\mathbbm{2}$};
\draw[->] (1) to  (2);
\end{tz}  

We have showed that the latching maps of the functor $D$ at the kinds $O,H,V,S,$ and $\dot{E}$, generate the cofibrations of the model structure in \cref{MS_equipments}, as they are exactly the set $\cI$ of \cref{def:cofib}. It remains to show that the remaining latching maps can be generated from these, which follows by inspection. 
\end{proof}

\begin{cor}\label{cor:Ddblcat_Reedy_cofibrant} The diagram $D_{\dblcat}$ is Reedy cofibrant, with the model structure on $\dblcat$ of \cref{MS_equipments}.
\end{cor}

\begin{proof}
It follows directly from \cref{under_triangle_D_Reedy_cofibrant}.
\end{proof}

\begin{rem}
    In the three examples worked here\textemdash categories, $2$-categories, and double categories\textemdash when we consider the weak factorization system considered $(\Cof,\trfib)$, we note that it is generated by the latching maps of the diagram $D$ for all the kinds on the ``fan down'' of the relational symbol encoding equality.
\end{rem}

We can now deduce the result we sought. Colloquially, this says that if $F\colon\bA\to\bB$ is double biequivalence between equipments, then the $\cL_\dblcat$-structures $M_{\bA}$ and $M_\bB$ satisfy the same formulae.

\begin{cor}\label{double_equiv_same_formulae}
    Let $F\colon\bA\to\bB$ be a double biequivalence between equipments, $\Gamma$ a context, and $\alpha\colon \Gamma\to M_{\bA}$ and $\beta\colon \Gamma\to M_\bB$ two interpretations compatible with $M_F\colon M_\bA\to M_\bB$. Then $M_{\bA}\vDash \phi[\alpha]$ if and only if $M_\bB\vDash\phi[\beta]$. 
\end{cor}

\begin{proof}
This is \cref{thm:we_then_structures_verify_same_formulate} where $\cM=\dblcat$ with the model structure of \cref{MS_equipments}.
\end{proof}

Although \cref{double_equiv_same_formulae} is the result we were after. It is worth noticing that by looking at the proof of \cref{thm:weak_equiv_to_Lequiv} we obtain as an intermediate step a double categorical counterpart of \cite[Lemma 11.2.24]{elements}. The proof given there applies almost verbatim in this case too as the result is at its core just facts about Reedy weak factorization systems, and we include it. It is also true that it is a special case of a more general result in \cite{HenryBardomiano}.

\begin{lem}\label{trivfib_to_fwsurjection}
    Let $\bA$ and $\bB$ be equipments, and $F\colon\bA\fwto\bB$ a double functor that is surjective on objects, full on horizontal and vertical morphisms, and fully faithful on squares. Then $F$ induces a fiberwise surjection $M_\bA\fto M_\bB$ of $\cL_{\dblcat}$-structures.
\end{lem}

\begin{proof}
We have shown that there is a weak factorization system (cofibrations, trivial fibrations) on $\dblcat$ whose trivial fibrations are the double functors that are surjective on objects, full on both horizontal and vertical morphisms, and fully faithful on squares\textemdash as in the statement. When we input the diagram $D_{\dblcat}\colon \cL_{\dblcat}^\op\to \Set$ built in \cref{Ddblcat} to the hom bifunctor 

\begin{tz}
\node[](1) {$(\dblcat^{\cL^\op_{\dblcat}})^\op\times\dblcat$};
\node[right of=1,xshift=3.5cm](2)         {$\Set^{\cL_{\dblcat}}$};

\draw[->] (1) to node[above,la]{$\mathrm{hom}$} (2);
\end{tz}
the resulting functor $M\_\coloneqq \mathrm{hom}(D_{\dblcat}-,-)$ gives us a map between $\cL_\dblcat$-structures.

Now, by \cref{rem:weighted_limitcolimit_hom_2varadj} we know that the bifunctor $\mathrm{hom}$ above is a right adjoint in a two-variable adjunction involving the (unenriched) weighted colimit bifunctor. Since \cref{lem:weightedlimcolim_Leibniz} guarantees that the weighted limit and colimit are Leibniz bifunctors, we use the equivalences of conditions in \cref{lem:Leibniz_two_variable_adj} to conclude that $\mathrm{hom}$ is a right Leibniz bifunctor with respect to the Reedy weak factorization systems built from any weak factorization system on $\dblcat$ and the (monomorphism, epimorphism) weak factorization system on $\Set$. 

Consequently, since $D_{\dblcat}$ is Reedy cofibrant, the functor $M\_\colon\dblcat\to\Set^{\dblcat}$ carries trivial fibrations to Reedy epimorphisms, which are exactly the fiberwise surjections.
\end{proof}

\bibliographystyle{alpha}
\bibliography{references}

\begin{thebibliography}{GMSV23}

\bibitem[ANST21]{univalence_pple}
Benedikt Ahrens, Paige~Randall North, Michael Shulman, and Dimitris Tsementzis.
\newblock The univalence principle.
\newblock {\em arXiv.org}, 2021.

\bibitem[BMH]{HenryBardomiano}
C\'esar Bardomiano~Mart\'inez and Simon Henry.
\newblock Homotopy languages.
\newblock In preparation.

\bibitem[B{\"o}h19]{Bohm}
Gabriella B{\"o}hm.
\newblock The {G}ray monoidal product of double categories.
\newblock {\em Appl. Categ. Structures}, 2019.

\bibitem[Bro73]{KenBrown_factorization}
Kenneth~S. Brown.
\newblock Abstract homotopy theory and generalized sheaf cohomology.
\newblock {\em Transactions of the American Mathematical Society},
  186:419--458, 1973.

\bibitem[CA91]{CarboniKellyWood}
Wood R.~J. Carboni~A., Kelly G.~M.
\newblock A $2$-categorical approach to change of base and geometric morphisms
  i.
\newblock {\em Cahiers de Topologie et Géométrie Différentielle
  Catégoriques}, 32(1):47--95, 1991.

\bibitem[CG10]{Shulman_Cruttwell2010}
Shulman Michael~A. Cruttwell~G.S.H.
\newblock A unified framework for generalized multicategories.
\newblock {\em Theory and Applications of Categories [electronic only]},
  24:580--655, 2010.

\bibitem[Cis06]{cisinski}
Denis-Charles Cisinski.
\newblock {\em Les pr\'efaisceaux comme mod\`eles des types d'homotopie}.
\newblock Number 308 in Ast\'erisque. Soci\'et\'e math\'ematique de France,
  2006.

\bibitem[DPP10]{DawsonParePronk}
Robert Dawson, Robert Paré, and Dorette Pronk.
\newblock The span construction.
\newblock {\em Theory and Applications of Categories}, 24:302--377, 08 2010.

\bibitem[Ehr63]{ehresmann0}
Charles Ehresmann.
\newblock Cat\'{e}gories structur\'{e}es.
\newblock {\em Ann. Sci. \'{E}cole Norm. Sup. (3)}, 80:349--426, 1963.

\bibitem[FPP08]{FPP}
Thomas~M. Fiore, Simona Paoli, and Dorette Pronk.
\newblock Model structures on the category of small double categories.
\newblock {\em Algebr. Geom. Topol.}, 8(4):1855--1959, 2008.

\bibitem[GKR20]{GKR_lifting_AMS}
Richard Garner, Magdalena Kędziorek, and Emily Riehl.
\newblock Lifting accessible model structures.
\newblock {\em Journal of topology}, 13(1):59--76, 2020.

\bibitem[GMSV23]{fibrantly_induced}
L\'eonard Guetta, Lyne Moser, Maru Sarazola, and Paula Verdugo.
\newblock Fibrantly-induced model structures.
\newblock arXiv:2301.07801, 2023.

\bibitem[GP04]{GPAdjointsDblCats}
Marco Grandis and Robert Pare.
\newblock Adjoint for double categories.
\newblock {\em Cahiers de Topologie et G\'eom\'etrie Diff\'erentielle
  Cat\'egoriques}, 45(3):193--240, 2004.

\bibitem[Gra19]{Grandis}
Marco Grandis.
\newblock {\em Higher Dimensional Categories: From Double To Multiple
  Categories}.
\newblock World Scientific Publishing Company, 2019.

\bibitem[Hen]{HenryFolds}
Simon Henry.
\newblock The language of a model structure.
\newblock Seminar talk,
  \url{https://www.uwo.ca/math/faculty/kapulkin/seminars/hottestfiles/Henry-2020-01-23-HoTTEST.pdf}.

\bibitem[HKRS17]{HKRS}
Kathryn Hess, Magdalena K\c{e}dziorek, Emily Riehl, and Brooke Shipley.
\newblock A necessary and sufficient condition for induced model structures.
\newblock {\em J. Topol.}, 10(2):324--369, 2017.

\bibitem[Hov99]{Hovey}
Mark Hovey.
\newblock {\em Model categories}, volume~63 of {\em Mathematical Surveys and
  Monographs}.
\newblock American Mathematical Society, Providence, RI, 1999.

\bibitem[JT91]{Joyal_Tierney}
Andr{\'e} Joyal and Myles Tierney.
\newblock Strong stacks and classifying spaces.
\newblock In Aurelio Carboni, Maria~Cristina Pedicchio, and Guiseppe Rosolini,
  editors, {\em Category Theory}, pages 213--236, Berlin, Heidelberg, 1991.
  Springer Berlin Heidelberg.

\bibitem[KS06]{KellyStreet}
G.~M. Kelly and Ross Street.
\newblock Review of the elements of 2-categories.
\newblock In {\em Category Seminar}, Lecture Notes in Mathematics, pages
  75--103. Springer Berlin Heidelberg, Berlin, Heidelberg, 2006.

\bibitem[Mak95]{makkai}
Michael Makkai.
\newblock First order logic with dependent sorts with applications to category
  theory.
\newblock \url{http://www.math.mcgill.ca/makkai/folds/foldsinpdf/FOLDS.pdf},
  1995.

\bibitem[Mar17]{unfolding_folds}
Jean-Pierre Marquis.
\newblock Unfolding folds: A foundational framework for abstract mathematical
  concepts.
\newblock In {\em Categories for the Working Philosopher}. Oxford University
  Press, Oxford, 2017.

\bibitem[Mos19]{Moser}
Lyne Moser.
\newblock Injective and projective model structures on enriched diagram
  categories.
\newblock {\em Homology Homotopy Appl.}, 21(2):279--300, 2019.

\bibitem[Mos20]{lyne}
Lyne Moser.
\newblock A double $(\infty,1)$-categorical nerve for double categories.
\newblock arXiv:2007.01848, 2020.

\bibitem[MSV]{MSVdouble_equivalences}
Lyne Moser, Maru Sarazola, and Paula Verdugo.
\newblock Double equivalences.
\newblock In preparation.

\bibitem[MSV22]{MSV}
Lyne Moser, Maru Sarazola, and Paula Verdugo.
\newblock A $2\mathrm{Cat}$-inspired model structure for double categories.
\newblock {\em Cah. Topol. G\'{e}om. Diff\'{e}r. Cat\'{e}g.},
  LXIII(2):184--236, 2022.

\bibitem[MSV23]{whi}
Lyne Moser, Maru Sarazola, and Paula Verdugo.
\newblock A model structure for weakly horizontally invariant double
  categories.
\newblock {\em Algebr. Geom. Topol.}, 23(4):1725--1786, 2023.

\bibitem[Qui67]{Quillen}
Daniel~G. Quillen.
\newblock {\em Homotopical algebra}.
\newblock Lecture Notes in Mathematics, No. 43. Springer-Verlag, Berlin-New
  York, 1967.

\bibitem[RV13]{Reedy_RV}
Emily Riehl and Dominic~R. Verity.
\newblock The theory and practice of reedy categories.
\newblock {\em arXiv: Category Theory}, 2013.

\bibitem[RV22]{elements}
Emily Riehl and Dominic Verity.
\newblock {\em Elements of $\infty$-Category Theory}.
\newblock Cambridge Studies in Advanced Mathematics. Cambridge University
  Press, 2022.

\bibitem[Shu]{Shulman_equipments}
Michael Shulman.
\newblock Post on the n-cat cafe.
\newblock \url{https://golem.ph.utexas.edu/category/2009/11/equipments.html}.

\bibitem[Shu08]{Shulman_framed_bicategories}
Michael Shulman.
\newblock Framed bicategories and monoidal fibrations.
\newblock {\em Theory and applications of categories}, 20(18):650--, 2008.

\bibitem[Ver11]{Verity}
Dominic Verity.
\newblock Enriched categories, internal categories and change of base.
\newblock {\em Reprints in theory and applications of categories}, (20):1--266,
  2011.
\newblock Originally published as: Ph.D. thesis, Cambridge University, 1992.

\bibitem[Woo82]{Wood}
R.~J. Wood.
\newblock Abstract pro arrows i.
\newblock {\em Cahiers de Topologie et Géométrie Différentielle
  Catégoriques}, 23(3):279--290, 1982.

\end{thebibliography}

\end{document}